\documentclass[12pt]{amsart}
\usepackage{graphicx,amssymb,latexsym,eepic,bm,color}
\usepackage[dvips,centering,includehead,width=15.1cm,height=22.7cm]{geometry}

\newcommand{\darkred}[1]{#1}
\newcommand{\lightgreen}[1]{#1}

\numberwithin{equation}{section}

\newtheorem{theorem}{Theorem}[section]
\newtheorem{proposition}[theorem]{Proposition}

\newtheorem{corollary}[theorem]{Corollary}
\newtheorem{lemma}[theorem]{Lemma}

\theoremstyle{definition}

\newtheorem{remark}[theorem]{Remark}
\newtheorem{example}[theorem]{Example}
\newtheorem{definition}[theorem]{Definition}
\newtheorem{problem}[theorem]{Problem}

\newcommand{\CC}{\mathbb{C}}
\newcommand{\PP}{\mathbb{P}}
\newcommand{\ZZ}{\mathbb{Z}}
\newcommand{\RR}{\mathbb{R}}

\newcommand{\Dcal}{\mathcal{D}}
\newcommand{\tDcal}{\widetilde\Dcal}
\newcommand{\Pcal}{\mathcal{P}}

\newcommand{\divv}{\operatorname{div}}

\newcommand{\ooo}{\multiput(0,0)(10,0){3}{\circle{2}}}
\newcommand{\oooo}{\multiput(0,0)(10,0){4}{\circle{2}}}

\newcommand{\Eee}{\put(1,0){\line(1,0){8}}}
\newcommand{\eEe}{\put(11,0){\line(1,0){8}}}
\newcommand{\eeE}{\put(21,0){\line(1,0){8}}}

\newcommand{\EEe}{\qbezier(0.8,0.6)(4,5)(10,5)\qbezier(10,5)(16,5)(19.2,0.6)}
\newcommand{\eEE}{\qbezier(10.8,0.6)(14,5)(20,5)\qbezier(20,5)(26,5)(29.2,0.6)}

\newcommand{\EEE}{\qbezier(0.6,0.8)(4,7)(15,7)
\qbezier(15,7)(26,7)(29.4,0.8)}

\newcommand{\eOe}{\qbezier(10.8,0.6)(15,4)(19.2,0.6)
                  \qbezier(10.8,-0.6)(15,-4)(19.2,-0.6)}

\newcommand{\eeO}{\qbezier(20.8,0.6)(25,4)(29.2,0.6)
                  \qbezier(20.8,-0.6)(25,-4)(29.2,-0.6)}

\newcommand{\eOO}{\qbezier(10.8,0.6)(14,5)(20,5)\qbezier(20,5)(26,5)(29.2,0.6)
                  \qbezier(10.8,-0.6)(14,-5)(20,-5)\qbezier(20,-5)(26,-5)(29.2,-0.6)}

\begin{document}


\title[Labeled floor diagrams]{Labeled floor diagrams for plane curves}

\author{Sergey Fomin}
\address{\hspace{-.3in} Department of Mathematics, University of Michigan,
Ann Arbor, MI 48109, USA}
\email{fomin@umich.edu}

\author{Grigory Mikhalkin}
\address{\hspace{-.3in} Section de math\'ematiques, Universit\'e de Gen\`eve,
1211 Geneva, Switzerland}
\email{grigory.mikhalkin@unige.ch}

\date{\today
}

\thanks{Partially supported by NSF grant DMS-0555880 (S.~F.), 
by the SNSF (G.~M.), 
and by the Clay Mathematical Institute Book Fellowship (G.~M.).}



\begin{abstract}
Floor diagrams are a class of weighted 
oriented graphs introduced by E.~Brugall\'e and
the second author. 
Tropical geometry arguments lead to combinatorial descriptions of
(ordinary and relative) Gromov-Witten invariants 
of projective spaces in terms of  
floor diagrams and their generalizations.
In a number of cases, these descriptions can be used to obtain
explicit (direct or recursive) formulas for 
the corresponding enumerative invariants. 
In particular, we use this approach to enumerate 
rational curves of given degree passing through a collection
of points on the complex plane and having maximal tangency to a given
line. 
Another application of the combinatorial approach is a proof of a
conjecture by P.~Di~Francesco--C.~Itzykson and 
L.~G\"ottsche that in the case of a fixed cogenus,
the number of plane curves of degree~$d$ passing through suitably many
generic points is given by a polynomial in~$d$, assuming that~$d$ is
sufficiently large.
Furthermore, the proof provides a method for computing these ``node
polynomials.'' 

A \emph{labeled floor diagram} is obtained by labeling the vertices of
a floor diagram by the integers $1,\dots,d$ in a manner compatible
with the orientation. We show that labeled floor diagrams of genus 0
are equinumerous to labeled trees, and therefore counted by the
celebrated Cayley's formula. The corresponding bijections
lead to interpretations of the Kontsevich numbers (the genus-0
Gromov-Witten invariants of the projective plane) in terms of certain
statistics on trees.
\end{abstract}

\maketitle

\tableofcontents


\section*{Introduction}

The primary purpose of this paper is to advertise a
general paradigm for solving~a large class of problems of classical
enumerative geometry. Although the main ingredients of the 
approach described herein have already appeared in the literature, 
a coherent presentation, complete with new convincing applications,
has been lacking. 
Our goal is to fill the gap, and in doing so win over a few converts.

\smallskip

A typical problem of enumerative geometry asks for the number of
geometric objects, say complex algebraic varieties of specified kind,
which satisfy a number of incidence or tangency constraints. 
In many cases (admittedly subject to limitations, both technical and
intrinsic), one can reduce such a problem to its
\emph{tropical} counterpart, that is, to the problem of weighted 
enumeration of certain polyhedral complexes known as tropical varieties. 
This reduction 
constitutes the first phase of a solution. 

The goal of the second phase, which can be called \emph{discretization},
is to replace piecewise-linear objects of tropical geometry by purely
combinatorial ones. If this is done successfully, one obtains a 
manifestly positive combinatorial rule 
(similar in spirit to the various Littlewood-Richardson-type rules in
Schubert Calculus) that identifies the answer to the original
geometric problem as the number of combinatorial objects of a
particular, complicated but explicit, kind. 

The third and final phase is one of purely \emph{combinatorial
enumeration} of the relevant discrete objects. 
Ideally, albeit seldom, it yields a formula for the
numbers in question, or some associated generating function. 
Otherwise, a recursion would do, or else an equation 
(algebraic, differential, or functional) for the generating function. 
At the very least, one would like to 
relate the objects to be enumerated to some more familiar combinatorial
gadgets, placing the problem within a well developed context. 

\smallskip

In this paper, we discuss this approach as it applies to the problem
of enumerating plane complex algebraic curves with given
properties, or more precisely, the problem of computing the
Gromov-Witten invariants, both ordinary and relative, of the complex
projective plane~$\PP^2$. 
Recall that the \emph{Gromov-Witten invariant}~$N_{d,g}$ is the number of 
irreducible curves of degree~$d$ and genus~$g$
passing through a fixed generic configuration of $3d+g-1$ points on~$\PP^2$.
A more general \emph{relative} Gromov-Witten invariant
$N_{d,g}(\lambda,\rho)$ is the number of 
such curves which, besides passing through appropriately many generic
points, satisfy tangency conditions with respect to a given line~$L$. 
The two integer partitions $\lambda$ and $\rho$ describe the degrees
of tangency at two generic collections of points on~$L$;
the points in the first collection are fixed while 
those in the second one are allowed to vary along~$L$. 

The tropical reduction for the problem of computing the
Gromov-Witten invariants~$N_{d,g}$ (resp., $N_{d,g}(\lambda,\rho)$) 
is accomplished by means of the \emph{correspondence theorem} 
established in~\cite{mikhalkin-05} (resp.,~\cite{mikhalkin-08}); 
see Theorem~\ref{th:geom-correspondence} (resp.,
Theorem~\ref{c-correspondence-alphabeta}). 
Even though this is the most substantial step among the three required
for a solution, we discuss it in less detail as the topic is already  
well covered in the existing literature. 

For the problem of computing the invariants~$N_{d,g}$, 
the discretization reduction has been recently given by
E.~Brugall\'e and the second author~\cite{brugalle-mikhalkin}
(see also an excellent exposition~\cite{brugalle}, in French), 
by establishing a bijection between the tropical curves in question
and certain \emph{markings} (``decorations'') 
of a particular kind of weighted acyclic graphs called
\emph{floor diagrams}; see Theorem~\ref{th:trop=combin}. 
This result extends to the relative setting as well;
see Theorem~\ref{th:trop=combin-two-ptns}. 

\smallskip

Apart from a review of the aforementioned results, 
the bulk of this paper is dedicated to the ``third phase'' of
combinatorial enumeration, as it applies to the problem at hand. 
Even though the general problem of (weighted) enumeration of marked
floor diagrams appears too unwieldy to allow for an explicit solution, 
the latter can be achieved in a number of particular instances.  
In Theorem~\ref{th:Nd0-ODE}, 
we enumerate irreducible rational curves
of given degree passing through a collection
of points on the complex plane and having \emph{maximal tangency} to a
given line; 
the point of tangency can be either prescribed or left unspecified. 
We also compute the relative
Gromov-Witten invariants associated with nonsingular or uninodal
curves (Corollaries~\ref{cor:nonsing-rel} and~\ref{cor:uninodal-rel}),
and with curves passing through a triple of collinear points 
(Corollary~\ref{cor:collinear-triple}). 

\smallskip

If the number of nodes $\delta=\frac{(d-1)(d-2)}{2}-g$ is fixed while
$d$ varies, determining $N_{d,g}$ 
as a function of~$d$ is a classical problem with
venerable history; see Section~\ref{sec:node-polynomials} and
references therein. 
In 1994, P.~Di~Francesco and C.~Itzykson
hypothesized~\cite{difrancesco-itzykson} 
that, for $\delta$~fixed and $d$~sufficiently large, 
the Gromov-Witten invariant~$N_{d,g}$
(or equivalently the corresponding \emph{Severi degree})
is a polynomial in~$d$ (necessarily of degree~$2\delta$). 
A~more explicit version of this conjecture was proposed by 
L.~G\"ottsche \cite{goettsche}.
The cases $\delta\le 6$ of G\"ottsche's conjecture were established 
by I.~Vainsencher~\cite{vainsencher}; 
then S.~Kleiman and R.~Piene~\cite{kleiman-piene} extended these
results to $\delta\le 8$.
In Section~\ref{sec:node-polynomials},
we give a proof of this conjecture based on the combinatorial rule of
Theorem~\ref{th:bm-curves}. 
(Alternative proofs have been proposed in unpublished preprints by 
Y.~Choi~\cite{choi-97}, A.-K.~Liu~\cite{liu}, and Y.-J.~Tzeng~\cite{tzeng}.) 
We stress that our proof does not merely establish polynomiality of
Severi degrees: it provides a method for computing those ``node
polynomials'' 
explicitly and directly, without resorting to interpolation. 
Our method has been implemented, with a few improvements, by
F.~Block~\cite{block-fpsac}, who computed the node polynomials for all
$\delta\le 13$. 

\smallskip

We 
reformulate the rule given 
in~\cite{brugalle-mikhalkin} in the language of 
 \emph{labeled floor diagrams} 
obtained by labeling the vertices of a floor diagram of degree~$d$ 
by the integers $1,\dots,d$ in a manner compatible with orientation.
This point of view, applied consistently thoughout the paper,
is not merely a matter of language or convenience.   
In Theorem~\ref{th:cayley}, we show 
that labeled floor diagrams of degree~$d$ and genus~$g=0$
are equinumerous to \emph{labeled trees} on $d$ vertices,
and hence counted by the celebrated Cayley's formula.
The corresponding bijections between labeled floor diagrams and trees 
yield interpretations of the 
numbers $N_{d,0}$ in terms of certain statistics on trees.

\smallskip

Two well known recursive formulas for the Gromov-Witten
invariants of the projective plane are  
due to M.~Kontsevich~\cite{kontsevich-manin} (for $N_{d,0}$) and 
to L.~Caporaso and J.~Harris~\cite{caporaso-harris}
(for $N_{d,g}(\lambda,\rho)$), respectively; see
Section~\ref{sec:kontsevich-caporaso-harris}.  
Each of these formulas can in principle be obtained directly from the
corresponding combinatorial rule  
(Theorem~\ref{th:bm-curves} and
Theorem~\ref{th:combin-rule-two-ptns}, respectively). 
While the derivation of the Caporaso-Harris recursion is relatively
straightforward (see~\cite{arroyo-brugalle-lopez,brugalle}),
deducing Kontsevich's formula \eqref{eq:kontsevich}
seems to require substantial technical effort aimed at replicating the
original argument of Kontsevich's in a purely combinatorial setting. 

The paper by Caporaso and Harris contains two recursive formulas:
the recursion \cite[Theorem~1.1]{caporaso-harris}  
for the number of all (not necessarily irreducible) curves with given
properties, and another recursion in
\cite[Section~1.4]{caporaso-harris} for the relative Gromov-Witten
numbers $N_{d,g}(\lambda,\rho)$. 
The first recursion does not lead
to a manifestly positive rule for the numbers
$N_{d,g}(\lambda,\rho)$ as it has to be followed by
an involved inclusion-exclusion procedure 
to restrict the count to irreducible curves.
The second recursion can in principle be used to obtain a positive
rule albeit an exceedingly cumbersome one. 

\smallskip

The paper is organized as follows.
Section~\ref{sec:labeled-floor-diagrams} introduces labeled floor
diagrams, 
and culminates in a reformulation of the main result
of~\cite{brugalle-mikhalkin} (see Theorem~\ref{th:bm-curves})
in this language.
Section~\ref{sec:tropical-curves} reviews the basics of plane tropical
curves.
Section~\ref{sec:combin-corresp-thms} discusses various versions of
the correspondence theorems, both geometric and combinatorial. 
Applications to computation of relative Gromov-Witten invariants
are presented in Section~\ref{sec:Ndg-rel-examples}. 
The proof of the polynomiality conjecture of Di Francesco--Itzykson 
and G\"ottsche is given in
Section~\ref{sec:node-polynomials}. 
Section~\ref{sec:enumeration} is devoted to enumeration of
labeled floor diagrams and related objects. 
In Section~\ref{sec:conjectures}, we formulate a number of conjectures
and open problems.
In Section~\ref{sec:welschinger}, we briefly discuss the
related problem of determining the Welschinger invariants 
for \emph{real} plane algebraic curves. 


In order to make the text accessible to both algebraic geometers and
enumerative combinatorialists,
we tried to make it as self-contained as possible, 
and in particular introduce the basic relevant background without
referring to outside sources. 

\medskip

We thank Florian Block, Philippe Di~Francesco, Bill Fulton, Rahul
Pandharipande, Ragni Piene, Michael Shapiro, Alek Vainshtein, and Ravi
Vakil for valuable comments. 


\section{Labeled floor diagrams and their markings}
\label{sec:labeled-floor-diagrams}

\subsection{Preliminaries on floor diagrams}

We use standard combinatorial terminology; see, e.g.,
\cite{van-lint-wilson, ec1, ec2}.

\begin{definition}[\emph{Labeled floor diagram and its multiplicity}]
\label{def:lfd}
Let $d>0$ and $g\ge 0$~be integers.
A~(connected) \emph{labeled floor diagram}~$\Dcal$
of degree~$d$ and genus~$g$
is a connected oriented graph $G=(V,E)$ on
a linearly ordered $d$-element vertex set~$V$ 
together with a \emph{weight} function $w:E\to\ZZ_{>0}\,$,
such that the following conditions are satisfied:
\begin{itemize}
\item 
\emph{(Genus)}
The edge set $E$ consists of $d+g-1$ edges.
Equivalently, 
the first Betti number of~$G$ is equal
to~$g$, assuming that $G$ is stripped of orientation and viewed as a
topological space (a $1$-dimensional simplicial complex).
\item 
\emph{(Compatibility with linear ordering on~$V$)}
Each edge in $E$ is directed from a vertex $u$ to a vertex $v>u$.
Thus $G$ is acyclic, and in particular has no loops.
Multiple edges are allowed.
\item 
\emph{(Divergence)}
For each vertex $v\in V$, we have
\begin{equation}
\label{eq:div}
\divv(v)\,\stackrel{\rm def}{=\!=}\,
\sum_{v\stackrel{e}{\longrightarrow}\circ} w(e)
-\sum_{\circ\stackrel{e}{\longrightarrow}v} w(e)
\le 1,
\end{equation}
where the first sum (respectively, the second one)
is over all edges $e$ directed away from~$v$
(respectively, towards~$v$).
\end{itemize}
The number
\begin{equation}
\label{eq:mult-D}
\mu(\Dcal)=\mu^\CC(\Dcal)=\prod_{e\in E} (w(e))^2
\end{equation}
is called the (complex) \emph{multiplicity} of a labeled floor
diagram~$\Dcal$.
\end{definition}

\begin{remark}
An (unlabeled) \emph{floor diagram}, introduced
in~\cite{brugalle-mikhalkin} (in a more general setting of curves
in~$\CC\PP^n$),
is essentially a labeled floor diagram considered up to an isomorphism
of weighted oriented graphs.
There are also other discrepancies with the setup
in~\cite{brugalle-mikhalkin}, but they can be
viewed as a matter of convention.
In this paper, we work exclusively with labeled floor diagrams.
In our opinion, this approach is more natural
from a combinatorial perspective.
\end{remark}

All labeled floor diagrams with $\le 4$ vertices are listed in 
Appendix~A. 
Each vertex set is ordered left to right; each edge is oriented
towards the right.

\begin{example}
\label{example:lfd4}
An example of a labeled floor diagram $\Dcal$ is shown below:
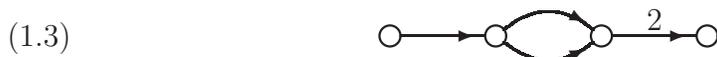
\begin{equation}
\label{eq:fd4}
\begin{picture}(80,15)(30,-4)\setlength{\unitlength}{4pt}\thicklines
\oooo\Eee\eOe\eeE
\put(25,1.5){\makebox(0,0){$2$}}
\put(7,0){\vector(1,0){1}}
\put(17.5,1.75){\vector(2,-1){1}}
\put(17.5,-1.75){\vector(2,1){1}}
\put(27,0){\vector(1,0){1}}
\end{picture}
\end{equation}
This labeled floor diagram has degree $d=4$ and genus $g=1$.
It has $d=4$ vertices, $d+g-1=4$ edges, vertex divergencies $\divv(v)$
equal to $1,1,0,-2$, respectively, 
and multiplicity~$\mu(\Dcal)=4$. 
\end{example}

\begin{definition}
[\emph{Marking of a labeled floor diagram}] 
\label{def-marking}
Let $\Dcal$ be a labeled floor diagram of
degree~$d$ and genus~$g$, as in Definition~\ref{def:lfd}. 
A~\emph{marking} of~$\Dcal$ is a ``combinatorial decoration''
of~$\Dcal$ obtained by the following procedure. 
(We will illustrate the steps of this procedure using the diagram from
Example~\ref{example:lfd4}.) 

\textbf{Step~1.}
For each vertex $v\in V$, introduce $1-\divv(v)$
(cf.~\eqref{eq:div})
new distinct vertices,
and connect $v$ with each of them by a single edge directed away
from~$v$:
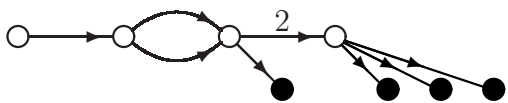
 \begin{center}
\begin{picture}(80,45)(30,-25)\setlength{\unitlength}{4pt}\thicklines
\oooo\Eee\eOe\eeE
\put(25,1.5){\makebox(0,0){$2$}}
\put(7,0){\vector(1,0){1}}
\put(17.5,1.75){\vector(2,-1){1}}
\put(17.5,-1.75){\vector(2,1){1}}
\put(27,0){\vector(1,0){1}}

\put(20.7,-0.7){\line(1,-1){4}}
\put(20.7,-0.7){\vector(1,-1){3}}
\put(25,-5){\circle*{2}}
\put(30.6,-0.8){\line(1,-1){4}}
\put(30.8,-0.6){\line(2,-1){9}}
\put(30.9,-0.4){\line(3,-1){14}}
\put(30.6,-0.8){\vector(1,-1){2.5}}
\put(30.8,-0.6){\vector(2,-1){5}}
\put(30.9,-0.4){\vector(3,-1){7.5}}
\put(35,-5){\circle*{2}}
\put(40,-5){\circle*{2}}
\put(45,-5){\circle*{2}}

\end{picture}
\end{center}

\textbf{Step~2.}
Split each original edge $e\in E$ in two,
by inserting an extra vertex in the middle of~$e$;
the resulting two edges inherit $e$'s orientation and weight:
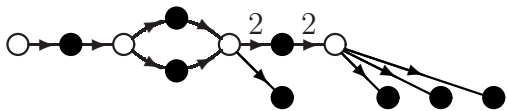
\begin{center}
\begin{picture}(80,45)(30,-25)\setlength{\unitlength}{4pt}\thicklines
\oooo\Eee\eOe\eeE
\put(22.5,2){\makebox(0,0){$2$}}
\put(27.5,2){\makebox(0,0){$2$}}
\put(2.5,0){\vector(1,0){1}}
\put(7.5,0){\vector(1,0){1}}
\put(17.5,1.75){\vector(2,-1){1}}
\put(17.5,-1.75){\vector(2,1){1}}
\put(12.5,1.75){\vector(2,1){1}}
\put(12.5,-1.75){\vector(2,-1){1}}
\put(22.5,0){\vector(1,0){1}}
\put(27.5,0){\vector(1,0){1}}
\put(5,0){\circle*{2}}
\put(15,2.5){\circle*{2}}
\put(15,-2.5){\circle*{2}}
\put(25,0){\circle*{2}}
\put(20.7,-0.7){\line(1,-1){4}}
\put(20.7,-0.7){\vector(1,-1){3}}
\put(25,-5){\circle*{2}}
\put(30.6,-0.8){\line(1,-1){4}}
\put(30.8,-0.6){\line(2,-1){9}}
\put(30.9,-0.4){\line(3,-1){14}}
\put(30.6,-0.8){\vector(1,-1){2.5}}
\put(30.8,-0.6){\vector(2,-1){5}}
\put(30.9,-0.4){\vector(3,-1){7.5}}
\put(35,-5){\circle*{2}}
\put(40,-5){\circle*{2}}
\put(45,-5){\circle*{2}}

\end{picture}
\end{center}
\noindent
Let $\widetilde G=(\widetilde V,\widetilde E)$ denote the acyclic
directed graph obtained after Steps~1--2.
It is easy to see that $\widetilde G$
has $3d+g-1$ vertices and $3d+2g-2$ edges.

\textbf{Step~3.}
Extend the linear ordering on~$V$ to~$\widetilde V$
so that, as before, each edge in~$\widetilde E$
is directed from a smaller to a larger vertex:
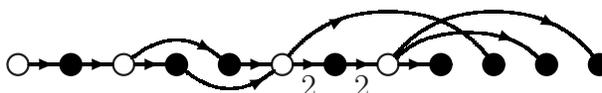
\begin{center}
\begin{picture}(80,45)(30,-15)\setlength{\unitlength}{4pt}\thicklines
\put(32.5,-2){\makebox(0,0){$2$}}
\put(27.5,-2){\makebox(0,0){$2$}}
\multiput(0,0)(10,0){2}{\circle{2}}
\multiput(25,0)(10,0){2}{\circle{2}}
\multiput(5,0)(10,0){2}{\circle*{2}}
\multiput(20,0)(10,0){4}{\circle*{2}}
\multiput(45,0)(10,0){2}{\circle*{2}}
\put(1,0){\line(1,0){8}}
\put(11,0){\line(1,0){3}}
\put(2.5,0){\vector(1,0){1}}
\put(7.5,0){\vector(1,0){1}}
\put(12.5,0){\vector(1,0){1}}
\qbezier(10.8,0.6)(15,4)(19.2,0.6)
\put(17.5,1.75){\vector(2,-1){1}}
\qbezier(15.8,-0.6)(20,-4)(24.2,-0.6)
\put(22.5,-1.75){\vector(2,1){1}}
\put(21,0){\line(1,0){3}}
\put(22.5,0){\vector(1,0){1}}
\put(26,0){\line(1,0){8}}
\put(27.5,0){\vector(1,0){1}}
\put(32.5,0){\vector(1,0){1}}
\put(36,0){\line(1,0){3}}
\put(37.5,0){\vector(1,0){1}}

\qbezier(25.6,0.8)(29,5)(35,5)\qbezier(35,5)(41,5)(44.4,0.8)
\put(30,4){\vector(2,1){1}}

\qbezier(35.6,0.8)(39,5)(45,5)\qbezier(45,5)(51,5)(54.4,0.8)
\put(50,4.1){\vector(2,-1){1}}

\qbezier(35.8,0.6)(38,3)(42.5,3)\qbezier(42.5,3)(47,3)(49.2,0.6)
\put(46,2.6){\vector(3,-1){1}}

\end{picture}
\end{center}
\noindent
%
The resulting object $\tDcal$ is called a \emph{marked floor diagram},
or a \emph{marking} of the original labeled floor diagram~$\Dcal$.
Thus $\tDcal$~is a directed graph~$\widetilde G=(\widetilde
V,\widetilde E)$ as above, together with a linear order on~$\widetilde V$
that extends the linear order on~$V$.
More precisely, we consider $\tDcal$ up to an automorphism that
fixes~$V$, so that linear orderings on~$\widetilde V$ which produce
isomorphic weighted directed graphs are viewed as the same marking
of~$\Dcal$.

\end{definition}


The number of markings of $\Dcal$ is denoted by~$\nu(\Dcal)$.

\begin{example}
In our running example (see~\eqref{eq:fd4}), we have $\nu(\Dcal)=6$.
These $6$~markings can be obtained from the diagram
in~\eqref{eq:mfd-12} by relocating the right endpoint of the edge
connecting the 6th vertex (counting from the left) to the 10th vertex
to any of the 5 alternative positions. 

Note that switching the 4th and 5th vertices does not change the
marking since it produces an isomorphic object. 
\end{example}

For many more examples, see Appendix~A. 

\medskip

Since the vertex set $\widetilde V$ is linearly ordered,
it is convenient to identify it with $\{1,2,\dots,3d+g-1\}$. 
(This also takes care of the isomorphism issue.) 
Also, we do not have to indicate the orientation of the edges once the
ordering of vertices has been fixed.
So for example the marked floor diagram $\widetilde\Dcal$ above can
just as well be drawn without the arrows:
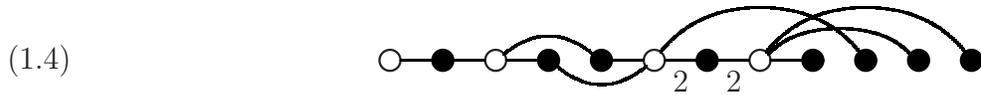
\begin{equation}
\label{eq:mfd-12}
\begin{picture}(80,30)(30,-3)\setlength{\unitlength}{4pt}\thicklines
\put(32.5,-2){\makebox(0,0){$2$}}
\put(27.5,-2){\makebox(0,0){$2$}}
\multiput(0,0)(10,0){2}{\circle{2}}
\multiput(25,0)(10,0){2}{\circle{2}}
\multiput(5,0)(10,0){2}{\circle*{2}}
\multiput(20,0)(10,0){4}{\circle*{2}}
\multiput(45,0)(10,0){2}{\circle*{2}}
\put(1,0){\line(1,0){8}}
\put(11,0){\line(1,0){3}}
\qbezier(10.8,0.6)(15,4)(19.2,0.6)
\qbezier(15.8,-0.6)(20,-4)(24.2,-0.6)
\put(21,0){\line(1,0){3}}
\put(26,0){\line(1,0){8}}
\put(36,0){\line(1,0){3}}

\qbezier(25.6,0.8)(29,5)(35,5)\qbezier(35,5)(41,5)(44.4,0.8)

\qbezier(35.6,0.8)(39,5)(45,5)\qbezier(45,5)(51,5)(54.4,0.8)

\qbezier(35.8,0.6)(38,3)(42.5,3)\qbezier(42.5,3)(47,3)(49.2,0.6)

\end{picture}
\end{equation}
\smallskip

We note that even though the underlying graph $\widetilde G$
of a marked floor diagram~$\widetilde\Dcal$ is naturally a Hasse
diagram of a partially ordered set,
it would be incorrect to define the markings of~$\Dcal$ simply
as linear extensions 
of this poset
because such a definition would ignore the
condition of compatibility of a linear extension with the original
linear order on~$V$.

\subsection{Combinatorial rules for Gromov-Witten invariants and
  Severi \hbox{degrees}}
\label{sec:Ndg+Severi}

The following result can be seen to be a restatement,
in the language introduced above,
of the first claim in \cite[Theorem~1]{brugalle-mikhalkin}.

\begin{theorem} 
\label{th:bm-curves}
The Gromov-Witten invariant~$N_{d,g}$ is equal to
\begin{equation}
\label{eq:N=sum-mu-nu}
N_{d,g}=\sum_\Dcal \mu(\Dcal)\,\nu(\Dcal),
\end{equation}
the sum over all labeled floor diagrams~$\Dcal$ of degree~$d$ and
genus~$g$.
\end{theorem}

To rephrase, the number~$N_{d,g}$ is obtained by enumerating marked
floor diagrams $\tDcal$ of degree~$d$ and genus~$g$, each taken with
its multiplicity $\mu(\tDcal)=\mu(\Dcal)$. 

The origins of Theorem~\ref{th:bm-curves} lie in the realm of
tropical geometry. 
They are discussed in Section~\ref{sec:combin-corresp-thm-ordinary}, 
following a review of the relevant background on tropical curves in 
Section~\ref{sec:tropical-curves}.  

\pagebreak[3]

\begin{example}
\label{example:N30}
The values of $\mu(\Dcal)$ and $\nu(\Dcal)$ for all diagrams~$\Dcal$ 
with $d\le 4$ are listed in Appendix~A. 
The formula~\eqref{eq:N=sum-mu-nu} then gives:
\smallskip

$N_{1,0}=1\cdot 1=1$ (unique line through $2$ generic
  points in~$\CC^2$),
\smallskip

$N_{2,0}=1\cdot 1=1$ (unique conic through $5$ generic
  points in~$\CC^2$),
\smallskip

$N_{3,0}=1\cdot 5+4\cdot 1+1\cdot 3=12$
(rational cubics through $8$ generic
  points in~$\CC^2$),
\smallskip

$N_{4,0}=1\cdot 40+4\cdot8+\cdots
+1\cdot15
=620$
(rational quartics, $11$ generic
  points in~$\CC^2$), 
\smallskip

$N_{3,1}=1\cdot 1=1$ (unique cubic through $9$ generic
  points in~$\CC^2$),
\smallskip

$N_{4,1}=1\cdot 26+4\cdot 4+\cdots +1\cdot 6=225$ 
(elliptic quartics, $12$ generic
  points in~$\CC^2$),
\smallskip

$N_{4,2}=1\cdot 3+1\cdot 5+\cdots +4\cdot 2=27$ 
(genus~$2$ quartics, $13$ generic
  points in~$\CC^2$),
\smallskip

$N_{4,3}=1\cdot 1=1$ 
(unique quartic through $14$ generic
  points in~$\CC^2$). 
\end{example}

Several approaches have been suggested to the computation of the
Gromov-Witten numbers~$N_{d,g}$, most notably the Caporaso-Harris
recursive algorithm~\cite{caporaso-harris}; 
see also \cite{gathmann03, goettsche, kleiman-piene, vainsencher, vakil2000}. 
The values~$N_{d,g}$ for small $d$ and~$g$ have been tabulated many
times over (see \emph{ibid.}); we do it again in
Figure~\ref{fig:Ndg}. 

\begin{figure}[htbp]
\begin{center}
\begin{tabular}{c|c|c|c|c|c|c}
        & $d=1$ & $d=2$ & $d=3$ & $d=4$ & $d=5$ & $d=6$ 
\\ 
\hline 
&&&&&&\\[-.15in]
$N_{d,0}$ & 1 & 1 & 12 & 620 & 87304 & 26312976 \\
$N_{d,1}$ & 0 & 0 & 1  & 225 & 87192 & 57435240 \\
$N_{d,2}$ & 0 & 0 & 0  & 27  & 36855 & 58444767 \\
$N_{d,3}$ & 0 & 0 & 0  & 1   & 7915  & 34435125 \\
$N_{d,4}$ & 0 & 0 & 0  & 0   & 882   & 12587820 \\
$N_{d,5}$ & 0 & 0 & 0  & 0   & 48    & 2931600 \\
$N_{d,6}$ & 0 & 0 & 0  & 0   & 1    & 437517
\end{tabular}
\end{center}
\caption{Gromov-Witten invariants $N_{d,g}$ for $d\le 6$ and $g\le 6$}
\label{fig:Ndg}
\end{figure}

\begin{figure}[htbp]
\begin{center}
\begin{tabular}{c|c|c|c|c|c|c}
        & $d=1$ & $d=2$ & $d=3$ & $d=4$ & $d=5$ & $d=6$ 
\\ 
\hline 
&&&&&&\\[-.15in]
$N^{d,0}$ & 1 & 1 & 1  & 1   & 1     & 1        \\
$N^{d,1}$ & 0 & \emph{3} & 12 & 27  & 48    & 75       \\
$N^{d,2}$ & 0 & 0 & \emph{21} & 225 & 882   & 2370     \\
$N^{d,3}$ & 0 & 0 & \emph{15} & \emph{675} & 7915  & 41310    \\
$N^{d,4}$ & 0 & 0 & 0  & \emph{666} & \emph{36975} & 437517   \\
$N^{d,5}$ & 0 & 0 & 0  & \emph{378} & \emph{90027} & \emph{2931831}  \\
$N^{d,6}$ & 0 & 0 & 0  & \emph{105} & \emph{109781}& \emph{12597900}   
%
\end{tabular}
\end{center}
\caption{Severi degrees $N^{d,\delta}$ for $d\le 6$ and $\delta\le 6$. 
The numbers in italics include reducible curves.}
\label{fig:severi-degrees}
\end{figure}

Closely related to the numbers $N_{d,g}$ are the \emph{Severi
  degrees}~$N^{d,\delta}$, defined as follows:
$N^{d,\delta}$ is the number of (possibly reducible) degree~$d$ plane
curves which have $\delta$ nodes and pass through a generic configuration of
$\frac{d(d+3)}{2}-\delta$ points on the plane. 
Figure~\ref{fig:severi-degrees} shows the values $N^{d,\delta}$ for
  small $d$ and~$\delta$. 

For an irreducible plane curve of degree~$d$, genus~$g$, and $\delta$
double points, we have 
\[
\delta+g=\frac{(d-1)(d-2)}{2}; 
\]
for this reason, the number of nodes $\delta$~is called the
\emph{cogenus} of a curve 
(be it irreducible or not).
It follows from Bezout's theorem that a degree~$d$ nodal curve of
cogenus $\delta\le d-2$ must be irreducible; hence  
\begin{equation}
\label{eq:severi=Ndg}
N^{d,\delta}=N_{d,\frac{(d-1)(d-2)}{2}-\delta} \ \ \text{for $d\ge \delta+2$.} 
\end{equation}
More generally, the Severi degrees can be recovered from the
Gromov-Witten numbers by the following well known (and easy to
justify) procedure. 
Fix a finite set $M$ of cardinality 
$\frac{d(d+3)}{2}-\delta$; 
we can think of $M$ as an indexing set for a point configuration. 
Split $M$ into an unordered disjoint union of
subsets $M=\bigcup_j M_j$; 
each such choice of a splitting corresponds to a distribution of
the points in a configuration among the irreducible components of a curve. 
Then pick integers $d_j>0$ (to serve as degrees of those components) 
and $\delta_j\ge 0$ (their cogenera) so that the following natural conditions
are satisfied:
\begin{align}
\label{eq:total-degree} &\sum_j d_j=d, \\
\label{eq:total-cogenus}&\sum_j \delta_j + \sum_{\{j,j'\}} d_j \, d_{j'}=\delta \\
\notag&\quad \text{(the second sum is over unordered
pairs of distinct indices $j$ and~$j'$),} \\
\label{eq:splitting-points}&\text{each $M_j$ has cardinality 
$|M_j|=\frac{d_j(d_j+3)}{2}-\delta_j$.} 
\end{align}
Using the notation $g_j=\frac{(d_j-1)(d_j-2)}{2}-\delta_j$, we then
have: 
\begin{equation}
\label{eq:severi-through-ndg}
N^{d,\delta}=\sum_{M=\cup M_j} \,\,\sum_{(d_j,\delta_j)}
\,\prod_j N_{d_j,g_j}\,,
\end{equation}
the sum over all unordered splittings $M=\cup M_j$ and all choices of $d_j$'s
and $\delta_j$'s satisfying
\eqref{eq:total-degree}--\eqref{eq:splitting-points}. 

\begin{example}[\emph{The Number of the Beast}]
The number of $4$-nodal plane quartics through 
10~generic points is $N^{4,4}=666$, computed as follows.  
Direct inspection shows that there are precisely two kinds of
splittings of a 10-element set~$M$ that work for $d=\delta=4$:
\begin{itemize}
\item
$M\!=\!M_1\!\cup\! M_2$ with $|M_1|\!=\!2$, $d_1\!=\!1$, $\delta_1\!=\!0$, 
$g_1\!=\!0$, 
                     $|M_2|\!=\!8$, $d_2\!=\!3$, $\delta_2\!=\!1$, $g_2\!=\!0$\\
(a~line through 2 points and a rational cubic through 8 points); 
\item
$M\!=\!M_1\cup M_2$ with $|M_1|=|M_2|=5$, $d_1=d_2=2$,
  $\delta_1=\delta_2=0$, $g_1=g_2=0$\\
(two conics, each passing through 5 points in a configuration). 
\end{itemize}
This yields
$N^{4,4}=\binom{10}{2} N_{1,0} \,N_{3,0}+\frac12\binom{10}{5}N_{2,0}^2
=45\cdot 1\cdot 12+\frac12 \cdot 252 \cdot 1^2 = 666$.
\end{example}

Combining formula~\eqref{eq:severi-through-ndg} with
Theorem~\ref{th:bm-curves}, we obtain the following combinatorial rule. 

\begin{corollary}
\label{cor:severi-combin}
The Severi degree~$N^{d,\delta}$ is equal to 
\begin{equation}
\label{eq:severi-combin}
N^{d,\delta}=\sum_{M=\cup M_j} \,\,\sum_{(d_j,\delta_j,\Dcal_j)}
  \prod_j \mu(\Dcal_j)\,\nu(\Dcal_j),
\end{equation}
the sum over all unordered splittings $M=\cup M_j$ of the set
$\{1,2,\dots,\frac{d(d+3)}{2}-\delta\}$, 
all choices of integers $d_j>0$ and $\delta_j\ge 0$ satisfying
\eqref{eq:total-degree}--\eqref{eq:splitting-points}, 
and all collections of labeled floor diagrams~$\Dcal_j$, 
each of respective degree~$d_j$ and
genus~$g_j=\frac{(d_j-1)(d_j-2)}{2}-\delta_j$, 
and supported on the vertex set~$M_j$.
\end{corollary}

To rephrase, the combinatorial rule for Severi degrees is the same as
for the Gromov-Witten numbers~$N_{d,g}$ except that one needs to drop
the condition that the labeled floor
diagrams involved be connected. 

\section{Tropical curves}
\label{sec:tropical-curves}

\newcommand{\R}{\mathbb R}
\newcommand{\Z}{\mathbb Z}
\newcommand{\C}{\mathbb C}
\newcommand{\N}{\mathbb N}
\newcommand{\Log}{\operatorname{Log}}

Let us review  the basic notions of tropical curves, both abstract
and parametrized; see \cite{IMS, mikhalkin-05, mikhalkin-06} for
further details. 


Throughout this section, $\bar C$ is a topological space
homeomorphic to a compact one-dimen\-sion\-al cell complex, 
i.e., a finite graph.
We will assume 
that the underlying graph of~$\bar C$ has no loops, 
no vertices of degree~$2$, and at least one vertex of degree~$\ge 3$
in each connected component. 

\begin{definition}[\emph{Valencies}] 
A small neighborhood of a point $c\in\bar C$ is
homeomorphic to a union of $k$ distinct rays in
an affine space emanating from the same origin. 
We~call~$k$ the \emph{valency} of~$c$;
accordingly, $c$~is called $k$-\emph{valent}.
All but finitely many points in~$\bar C$ are 2-valent. 
Let $C$ denote the subset of~$\bar C$ obtained by removing all the
(finitely many) points of valency~1 (the \emph{univalent} vertices). 
\end{definition}

\begin{definition}[\emph{Tropical curves}]
A {\em tropical structure} on~$C$ is an inner complete metric on~$C$. 
It can be described by specifying the lengths
of all the edges of the underlying graph; 
these lengths are $+\infty$ for the edges incident 
to the removed univalent vertices (the \emph{unbounded} edges), 
and are positive real numbers for the remaining (\emph{bounded}) edges. 
A space $C$ as above endowed with a tropical structure is called a
(non-parametrized, or abstract) {\em tropical curve}.
Such a tropical curve is \emph{irreducible} if $C$ is connected.  
We call the first Betti number  of~$C$ the {\em genus} of a tropical
curve. 
\end{definition}

\begin{example}
\label{example:trop-curve}
Figure~\ref{fig:trop-curve}(a) shows a graph~$\bar C$ 
with 12~univalent and 12 trivalent vertices,
and with 12 bounded and 12 unbounded edges. 
Figure~\ref{fig:trop-curve}(b) shows an irreducible tropical curve
of genus $g=1$ obtained from~$\bar C$ by removing the univalent
vertices and assigning the lengths $\ell_1,\dots,\ell_{12}$ to its 
bounded edges.  
\end{example}

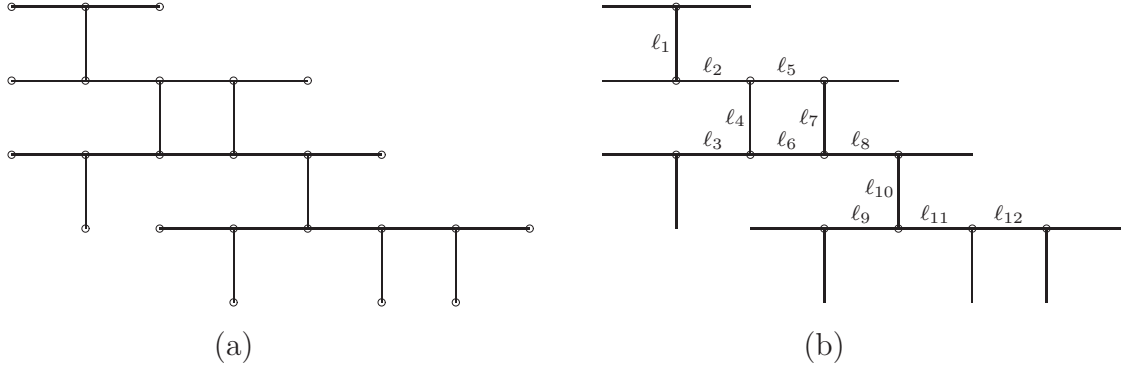
\begin{figure}[htbp]
\begin{center}
\setlength{\unitlength}{2.8pt}
\begin{picture}(70,48)(0,-5)
\thinlines
\put(30,-6){\makebox(0,0){(a)}}
\put(0,40){\line(1,0){20}}
\put(0,30){\line(1,0){40}}
\put(0,20){\line(1,0){50}}
\put(20,10){\line(1,0){50}}

\put(10,10){\line(0,1){10}}
\put(10,30){\line(0,1){10}}
\put(20,20){\line(0,1){10}}
\put(30,0){\line(0,1){10}}
\put(30,20){\line(0,1){10}}
\put(40,10){\line(0,1){10}}
\put(50,0){\line(0,1){10}}
\put(60,0){\line(0,1){10}}

\multiput(0,40)(10,0){3}{\circle{1}}
\multiput(0,30)(10,0){5}{\circle{1}}
\multiput(0,20)(10,0){6}{\circle{1}}
\multiput(10,10)(10,0){7}{\circle{1}}
\multiput(30,0)(20,0){2}{\circle{1}}
\multiput(60,0)(20,0){1}{\circle{1}}

\end{picture}
\qquad
\begin{picture}(70,48)(0,-5)
\put(30,-6){\makebox(0,0){(b)}}
\thinlines
\put(0,40){\line(1,0){20}}
\put(0,30){\line(1,0){40}}
\put(0,20){\line(1,0){50}}
\put(20,10){\line(1,0){50}}

\put(10,10){\line(0,1){10}}
\put(10,30){\line(0,1){10}}
\put(20,20){\line(0,1){10}}
\put(30,0){\line(0,1){10}}
\put(30,20){\line(0,1){10}}
\put(40,10){\line(0,1){10}}
\put(50,0){\line(0,1){10}}
\put(60,0){\line(0,1){10}}

\put(10,30){\line(1,0){20}}
\put(10,20){\line(1,0){30}}
\put(30,10){\line(1,0){30}}

\put(10,30){\line(0,1){10}}
\put(20,20){\line(0,1){10}}
\put(30,20){\line(0,1){10}}
\put(40,10){\line(0,1){10}}

\multiput(10,40)(10,0){1}{\circle{1}}
\multiput(10,30)(10,0){3}{\circle{1}}
\multiput(10,20)(10,0){4}{\circle{1}}
\multiput(30,10)(10,0){4}{\circle{1}}

\put(8,35){\makebox(0,0){$\scriptstyle\ell_1$}}
\put(15,32){\makebox(0,0){$\scriptstyle\ell_2$}}
\put(15,22){\makebox(0,0){$\scriptstyle\ell_3$}}
\put(18,25){\makebox(0,0){$\scriptstyle\ell_4$}}
\put(25,32){\makebox(0,0){$\scriptstyle\ell_5$}}
\put(25,22){\makebox(0,0){$\scriptstyle\ell_6$}}
\put(28,25){\makebox(0,0){$\scriptstyle\ell_7$}}
\put(35,22){\makebox(0,0){$\scriptstyle\ell_8$}}
\put(35,12){\makebox(0,0){$\scriptstyle\ell_9$}}
\put(37.5,15.5){\makebox(0,0){$\scriptstyle\ell_{10}$}}
\put(45,12){\makebox(0,0){$\scriptstyle\ell_{11}$}}
\put(55,12){\makebox(0,0){$\scriptstyle\ell_{12}$}}

\end{picture}
\end{center}
\caption{A graph and a related tropical curve}
\label{fig:trop-curve}
\end{figure}

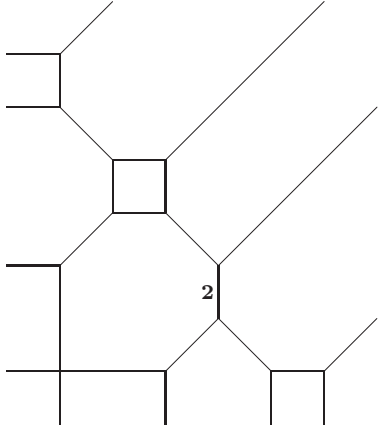
\begin{figure}[htbp]
\begin{center}
\setlength{\unitlength}{2pt}
\begin{picture}(70,80)(0,0)
\thinlines
\put(0,30){\line(1,0){10}}
\put(0,60){\line(1,0){10}}
\put(0,70){\line(1,0){10}}
\put(0,10){\line(1,0){30}}
\put(10,70){\line(1,1){10}}
\put(30,50){\line(1,1){30}}
\put(40,30){\line(1,1){30}}
\put(60,10){\line(1,1){10}}
\put(10,60){\line(1,-1){10}}
\put(30,40){\line(1,-1){10}}
\put(40,20){\line(1,-1){10}}
\put(20,40){\line(1,0){10}}
\put(20,50){\line(1,0){10}}
\put(50,10){\line(1,0){10}}
\put(10,30){\line(1,1){10}}
\put(30,10){\line(1,1){10}}
\put(10,60){\line(0,1){10}}
\put(20,40){\line(0,1){10}}
\put(30,40){\line(0,1){10}}
\put(40,20){\line(0,1){10}}
\put(10,0){\line(0,1){30}}
\put(30,0){\line(0,1){10}}
\put(50,0){\line(0,1){10}}
\put(60,0){\line(0,1){10}}

\put(38,25){\makebox(0,0){$\scriptstyle\mathbf{2}$}}
\end{picture}
\end{center}
\caption{The image of a plane tropical curve}
\label{fig:plane-trop-curve}
\end{figure}

\begin{definition}[\emph{Tropical morphisms and plane tropical
      curves}] 
\label{def:trop-morphism}
A map $h:C\to\R^n$ is called a \emph{tropical morphism} if it
satisfies the following properties:
\begin{itemize}
\item
$h$ is \emph{affine} along each edge~$e$ in~$C$. 
To be precise, let $a\in C$ be an endpoint of~$e$
(thus $a$~is not univalent);  
then there is a vector $\Delta_a(e)\!\in\!\RR^n$ such that the restriction
of $h$ to $e$ is given~by 
\[
h(c)=h(a)+\ell(a,c)\,\Delta_a(e);  
\]
here $\ell(a,c)$ denotes the length of the segment $[a,c]$ of the
edge~$e$. 
\item
the vectors $\Delta_a(e)$ have integer coordinates; 
\item
for a fixed vertex~$a$, the vectors $\Delta_a(e)$ satisfy the
\emph{balancing condition} (cf.~\cite{mikhalkin-06}) 
\begin{equation}
\label{eq:balancing}
\sum_e \Delta_a(e)=0, 
\end{equation}
where the sum is taken over all edges $e$ adjacent to~$a$. 
\end{itemize}
From now on we focus on the case~$n\!=\!2$.
A morphism from a tropical curve $C$ to~$\R^2$ is called a 
\emph{(parametrized) plane tropical curve}, or
a tropical curve in~$\R^2$.

Such a curve assigns positive integer \emph{weights} to the
edges in~$C$, as follows. 
The weight $w(e)$ of an edge $e$ is the greatest common divisor of the
coordinates of an integer vector~$\Delta_a(e)$. 
In view of~\eqref{eq:opposite-vectors} below, 
this does not depend on the choice of~$a$. 
\end{definition}

The sets $h(C)\subset\R^2$ obtained as the images of tropical
morphisms were originally introduced by Aharony, Hanany, and
Kol~\cite{aharony-hanany-kol} under the name of {\em $(p,q)$-webs}. 

\begin{example}[cf.\ Example~\ref{example:trop-curve}]
\label{example:plane-trop-curve}
Figure~\ref{fig:plane-trop-curve} shows the image of a particular
morphism (a plane tropical curve) $h:C\to\RR^2$ where $C$ is 
the tropical curve from Figure~\ref{fig:trop-curve}. 
The weights of all edges are~1 except for one edge  of weight~2 (the edge
whose length is~$\ell_{10}$). 
The integer vectors $\Delta_a(e)$ can be read off the picture as follows:
$\Delta_a(e)$ is $w(e)$ times the primitive vector of the segment or
ray representing~$e$,  
pointing away from the image of~$a$. 
Thus in this example, the values taken by $\Delta_a(e)$ are 
$(0,\pm1)$, $(\pm1,0)$, $(\pm1,\pm1)$, and $(0,\pm2)$. 
\end{example}

If $e$ is a \emph{bounded} edge, i.e., an edge connecting two vertices $a$
and $b$ in~$C$, then 
\begin{equation}
\label{eq:opposite-vectors}
\Delta_a(e)+\Delta_b(e)=0. 
\end{equation}
If $e$ is \emph{unbounded}, then it has a unique endpoint $a\in C$;
consequently, the notation $\Delta(e)=\Delta_a(e)$
is unambiguous. 
Let $E^\circ$ denote the set of all unbounded edges. 
It follows from \eqref{eq:balancing} and~\eqref{eq:opposite-vectors}
that 
\[
\sum_{e\in E^\circ} \Delta(e)=0. 
\]

\begin{definition}[\emph{Degree of a plane tropical curve}] 
For a vector $v\!=\!(p,q)\in\ZZ^2$, set 
\[
\langle v\rangle=
\max(p,q,0). 
\]
The \emph{(projective) degree} of a plane tropical curve $h:C\to\RR^2$
is defined by 
\begin{equation}
\label{eq:deg-h}
\deg h=
\sum_{e\in E^\circ} \langle\Delta(e)\rangle.
\end{equation}
Thus the degree depends on the map~$h$---unlike the genus, 
which only depends on the curve~$C$.
\end{definition}

To illustrate, the plane curve in
Example~\ref{example:plane-trop-curve} has degree $d=4$. 

\begin{remark}
\label{rem:deg-simple}
For all curves $h$ to be considered in this paper (say of degree~$d$), 
the collection of vectors $\{\Delta(e)\}_{e\in E^\circ}$ 
consists of $d$ copies of each of the three vectors $(-1,0)$,
$(0,-1)$, and~$(1,1)$,  
so that formula \eqref{eq:deg-h} yields $\deg h=d$. 
\end{remark}


\section{The combinatorial correspondence theorems}
\label{sec:combin-corresp-thms}

\subsection{Combinatorial rule for the ordinary Gromov-Witten
  invariants}
\label{sec:combin-corresp-thm-ordinary}
The correspondence theorem of Tropical Geometry~\cite{mikhalkin-05} 
reduces the problems of counting plane complex curves with prescribed 
properties to the appropriate tropical versions of the same problems,
that is, to (weighted) enumeration of certain plane tropical curves. 
In this section, we describe a setting that leads to 
a bijection between such tropical curves
and some purely combinatorial objects, namely marked floor diagrams of
Section~\ref{sec:labeled-floor-diagrams}. 
Combining the two constructions, we then obtain a ``combinatorial
correspondence theorem'' that allows one to answer 
questions in enumerative geometry of the complex plane in 
direct combinatorial terms.

\begin{definition}[\emph{Elevators}] 
\label{def:elevators}
Let $h:C\to\RR^2$ be a plane tropical curve. 
An edge $e$ in~$C$ (either bounded or unbounded) 
is called an \emph{elevator} of~$h$ 
if the image $h(e)\subset\RR^2$ is \emph{vertical}, 
i.e., if the vector(s) $\Delta_a(e)$ are nonzero and parallel to $(0,1)$.
The two possible orientations of an elevator~$e$ are naturally called
\emph{up} and~\emph{down}. 
\end{definition}

\begin{definition}[\emph{Floor diagram of a plane tropical curve}]
A {\em floor} of a plane tropical curve $h:C\to\R^2$ 
is a connected component (in~$C$) 
of the union of all edges which are not elevators. 
The \emph{floor diagram} of $h$ is an oriented weighted
graph $\Dcal(h)$ obtained from $C$ by removing 
(the interiors of) all unbounded edges, 
collapsing each floor to a single vertex,
orienting all remaining edges 
(which correspond to bounded elevators) downwards,
and keeping their weights. 
Thus, the vertices of $\Dcal(h)$ correspond to the floors, and
the edges to the bounded elevators (directed downwards). 
\end{definition}

\begin{example}[cf.\ Examples~\ref{example:trop-curve}
and~\ref{example:plane-trop-curve}] 
The plane tropical curve $h:C\to\RR^2$ whose image is shown
in Figure~\ref{fig:plane-trop-curve} has 8 elevators: 
4~bounded and 4~unbounded. 
They are precisely the 8 vertical edges in
Figure~\ref{fig:trop-curve}(b), which are represented by the
8~vertical segments and rays in 
Figure~\ref{fig:trop-curve}(a). 

The 4 floors of $h$ 
are formed by the horizontal edges of $C$ as shown in 
Figure~\ref{fig:trop-curve}(b). 
The floor diagram $\Dcal(h)$ is obtained by removing the unbounded
edges, contracting each floor to a point, directing the 4 remaining
edges downwards, and assigning weights 1,1,1,2 to them. 
The result is a diagram isomorphic to the one shown
in~\eqref{eq:fd4}. 
\end{example}

In general, a floor diagram of a plane tropical curve may not admit
a labeling satisfying the conditions in Definition~\ref{def:lfd}, 
so it might not correspond to any labeled floor diagram in the sense
of that definition. In particular, both the acyclicity condition and 
the divergence condition~\eqref{eq:div} cannot be guaranteed.
It turns out however that the floor diagrams of plane curves 
passing through point configurations of certain kind 
always satisfy the requisite conditions, as we explain next.

\begin{definition}[\emph{Vertically stretched configurations}]
\label{def:vert-stretched}
A $(3d-1+g)$-element set 
\[
\Pcal=\{(x_1,y_1),(x_2,y_2),\dots,(x_{3d-1+g},y_{3d-1+g})\}
\subset\RR^2
\]
is called a \emph{vertically stretched $(d,g)$-configuration} 
if 
\begin{equation}
\label{eq:vert-stretched}
\begin{array}{l}
x_1<x_2<\cdots<x_{3d-1+g}\,,\\[.05in]
y_1<y_2<\cdots<y_{3d-1+g}\,,\\[.05in]
\displaystyle\min_{i\neq j}|y_i-y_j| > (d^3+d)\cdot \max_{i\neq j}|x_i-x_j|. 
\end{array}
\end{equation}
Such a configuration $\Pcal$ comes equipped with a
``downwards'' linear order in which higher points precede the lower
ones. 
\end{definition}

\begin{remark}
The conditions in Definition~\ref{def:vert-stretched} can be relaxed
by removing the second string of inequalities, 
changing the coefficient $(d^3+d)$, and/or removing the $\min$ and
the~$\max$.  
We do not make an attempt to determine the weakest possible conditions
that ensure the desired properties of associated tropical curves; 
instead, we chose the version that
makes subsequent arguments as simple as possible. 
\end{remark}

We say that a plane tropical curve \emph{passes through} a
configuration~$\Pcal\subset\RR^2$ if the image of the curve
contains~$\Pcal$, i.e., $h(C)\supset\Pcal$. 

The following key lemma is a restatement of a result that can be extracted
from \cite[Section~5]{brugalle-mikhalkin08}. 

\pagebreak[3]

\begin{lemma}
\label{lem:floor-stretched}
Let $\Pcal$ be a 
vertically stretched $(d,g)$-configuration, 
and $h$ an irreducible plane tropical curve of degree~$d$ and genus~$g$
passing through~$\Pcal$. 
Then:
\begin{itemize}
\item
each floor of $h$ contains a unique point in~$\Pcal$;
\item
the linear ordering of the floors induced
from~$\Pcal$ makes $\Dcal(h)$ a labeled floor diagram of
degree~$d$ and genus~$g$ (cf.\ Definition~\ref{def:lfd});
\item
each elevator of $h$ contains a unique point in $\Pcal$;
\item 
the linear ordering of the floors and elevators induced
from~$\Pcal$ produces a marking of~$\Dcal(h)$, 
denoted by $\tDcal(h,\Pcal)$ 
(cf.\ Definition~\ref{def-marking}).
\end{itemize}
\end{lemma}

To amplify, such a curve $h$ has exactly $d$ floors,
so that $\Dcal(h)$ has $d$ vertices; 
each floor is contractible, so $\Dcal(h)$, like~$h$, has genus~$g$;
the graph $\Dcal(h)$ is acyclic; and it satisfies the divergence
condition.
 
Thus, each of the $d$ floors of~$h$ is a graph of a
continuous piecewise-linear function, with slopes at the left and
right ends equal to $0$ and~$1$, respectively. 
See Figure~\ref{fig:floor-ddiag-of-trop-curve}. 

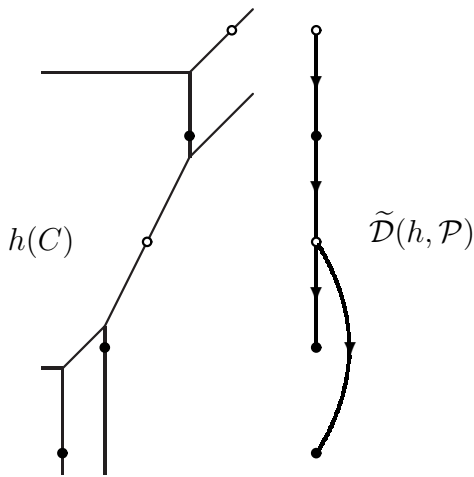
\begin{figure}[htbp]
\begin{center}
\setlength{\unitlength}{0.8pt}
\begin{picture}(140,220)(5,0)
\thicklines

\darkred{\linethickness{1pt}
\put(70,150){\line(0,1){40}}
}
\lightgreen{\linethickness{1pt}
\put(10,0){\line(0,1){50}}
\put(30,0){\line(0,1){70}}
}

\put(0,50){\line(1,0){10}}
\put(0,190){\line(1,0){70}}

\put(10,50){\line(1,1){20}}
\put(70,150){\line(1,1){30}}
\put(70,190){\line(1,1){18.6}}

\put(100,220){\line(-1,-1){8.6}}

\put(30,70){\line(1,2){19.1}}
\put(70,150){\line(-1,-2){19.1}}

\put(10,10){\circle*{4}}
\put(30,60){\circle*{4}}
\put(50,110){\circle{4}}
\put(70,160){\circle*{4}}
\put(90,210){\circle{4}}

\put(0,110){\makebox(0,0){$h(C)$}}

\darkred{\linethickness{1pt}
\put(130,158){\line(0,-1){46}}
\put(130,158){\vector(0,-1){27}}
\put(130,208){\line(0,-1){46}}
\put(130,208){\vector(0,-1){27}}
}

\lightgreen{\linethickness{1pt}
\put(130,108){\line(0,-1){46}}
\put(130,108){\vector(0,-1){27}}
\qbezier(131.2,11.6)(160,60)(131,108.4)
\put(146,56){\vector(0,-1){1}}
}

\put(130,10){\circle*{4}}
\put(130,60){\circle*{4}}
\put(130,110){\circle{4}}
\put(130,160){\circle*{4}}
\put(130,210){\circle{4}}

\put(155,110){{$\tDcal(h,\Pcal)$}}
\end{picture}
\end{center}
\caption{Marked floor diagram of a tropical curve}
\label{fig:floor-ddiag-of-trop-curve}
\end{figure}

\medskip

In fact, much more is true. 

\pagebreak[3]

\begin{theorem}
\label{th:trop=combin}
Let $\Pcal$ be a 
vertically stretched $(d,g)$-configuration. 
Then the mapping $h\!\mapsto\!\tDcal(h,\Pcal)$ establishes a bijective
correspondence between irreducible plane tropical curves of degree~$d$ and
genus~$g$ passing through~$\Pcal$ and 
the marked floor diagrams of  degree~$d$ and genus~$g$. 
\end{theorem}

\begin{proof}
Let $\tDcal$ be a marking of a labeled floor diagram $\Dcal$ of degree
$d$ and genus~$g$. 
We need to show that $\tDcal=\tDcal(h,\Pcal)$ for a unique 
plane tropical curve~$h$ of the same degree and genus passing
through~$\Pcal$.
To do that, let us describe the structure of such a curve in concrete
detail. Its construction from a given marked floor diagram~$\tDcal$
and a point configuration~$\Pcal$ will then proceed by ``reverse
engineering.'' 

Since the vertex set of $\tDcal$ and the configuration~$\Pcal$ are
linearly ordered 
sets of the same cardinality, there is a canonical order-preserving
bijection $v\mapsto p(v)$ between them. 
It will be convenient to distinguish between ``white'' and ``black''
points \hbox{$p(v)\in\Pcal$,} 
depending on whether $v$ comes from~$\Dcal$ or not. 
Each white vertex lies on the respective floor, and each black
vertex lies on the corresponding elevator. 

As we scan a floor~$F_v$ passing through a white vertex~$p(v)$ 
from its right end all the way to the left, 
the slope of the floor changes each time it meets an elevator.
Specifically, the process unfolds as follows.
The initial slope at the right end is~1. 
A elevator~$e_u$ arriving from above  
(respectively, from below)
and passing through a black vertex~$p(u)$ 
increases (respectively, decreases) the slope by~$w(e_u)$. 
Due to the way the points in~$\Pcal$ are placed, 
the order in which those elevators hit the floor~$F_v$ as we scan
it right to left is precisely the (top-down) order in which the
corresponding vertices~$u$ 
(that is, all the black vertices connected to~$v$) 
appear in~$\tDcal$. 
Consequently, the slope of a segment $S_u$ of~$F_v$ 
that is bounded on the right by an elevator~$e_u$ is uniquely
determined by the combinatorics of~$\tDcal$: 
it is equal to $1$ plus a signed sum of weights of all edges
in~$\tDcal$ connecting~$v$ with vertices $\le u$;
the edges arriving at~$v$ contribute with a positive sign while the
edges leaving~$v$, with a negative one. 
Furthermore, the right endpoint of~$S_u$ lies on the vertical line
passing through~$p(u)$ while the left endpoint lies on the vertical
line passing 
through~$p(u')$, where $u'$ is the vertex in~$\tDcal$ immediately
following~$u$ in the ordered list of vertices connected to~$v$. 
To~summarize, the combinatorics of~$\tDcal$ determines the slopes of
segments making up the floor~$F_v$ while the vertical lines passing through the
points~$p(u)$, for all $u$ connected to~$v$, determine the
$x$-coordinates of the breakpoints on~$F_v$. 
This defines~$F_v$ up to a vertical shift;
the latter is determined by the condition that $p(v)\in F_v$. 

The recipe for constructing a (necessarily unique)
tropical curve~$h$ with the desired
properties is now clear: each floor $F_v$ is described by the above
rule, and the vertical elevators are then drawn through the black
points, bounded by the appropriate floors 
(or going all the way down in case of
unbounded elevators). 
Since the slopes of floor segments cannot exceed~$d$,
condition~\eqref{eq:vert-stretched} guarantees that each floor~$F_v$ 
constructed in this way will fit below (respectively above) all black
vertices~$p(u)$ for $u\le v$ (resp.,~$v\le u$), ensuring that the
recipe works. 
It is also clear that the curve~$h$ constructed in this way will be
irreducible, and will have
the required genus and degree (for the latter, cf.\
Remark~\ref{rem:deg-simple}). 
\end{proof}

Theorem~\ref{th:trop=combin} is illustrated in
Appendix~B, 
which shows the 9~tropical rational
cubics passing through a vertically stretched $(3,0)$-configuration of
8~points, alongside their respective marked floor diagrams. 

\medskip


Theorem~\ref{th:geom-correspondence} below is a special case of the 
(geometric) ``correspondence theorem'' \cite[Theorem~1]{mikhalkin-05}.

For a positive real number~$t$,
let $\Log_t:(\C^\times)^2\to\R^2$ denote the map defined by 
\begin{equation}
\label{eq:Log-t}
\Log_t(z,w)=(\log_t|z|,\log_t|w|).
\end{equation}

\begin{theorem}
\label{th:geom-correspondence}
Let $\Pcal$ be a vertically stretched $(d,g)$-configuration. 
Let $t$ be a sufficiently large positive number. 
Then for any configuration ${\mathcal
  P^{\C}}\subset(\C^\times)^2$ of $3g-1+d$ points such that 
$\Log_t({\mathcal P^{\C}})=\Pcal$, there is a 
canonical surjective ``tropicalization'' 
map 
\[
\gamma\mapsto \operatorname{Trop}_{\Pcal^\CC,t}(\gamma)
\]
from the set of irreducible complex algebraic curves $\gamma$ 
of degree $d$ and genus~$g$ passing through ${\mathcal P^{\C}}$
to the set of irreducible plane tropical curves $h$ of degree $d$ and
genus $g$ passing through~$\Pcal$. 
Under this map, the preimage
  $\operatorname{Trop}_{\Pcal^\CC,t}^{-1}(h)$ 
of each such curve $h$ consists of 
$
\mu(\Dcal(h))$
distinct complex curves.
\end{theorem}

\begin{proof}
To adapt the correspondence theorem to our current setup,
fix the Newton polygon of the curves under consideration to be the set 
\[
\Delta_d=\{(i,j):i\ge0, j\ge0, i+j\le d\},  
\] 
and observe that for a plane tropical
curve $h$ passing through $\Pcal$, 
the complex multiplicity $\mu_{\C}(h)$ (as defined
in~\cite{mikhalkin-05}) 
is equal to the multiplicity of the labeled floor diagram $\Dcal(h)$
as defined in~\eqref{eq:mult-D}. 
\end{proof}

Combining Theorems~\ref{th:trop=combin}
and~\ref{th:geom-correspondence}, we obtain the following enhancement 
of Theorem~\ref{th:bm-curves}. 

\begin{theorem}[Combinatorial correspondence theorem for plane curves]
\label{th:combin-correspondence}\ \\
Let $\Pcal\subset\RR^2$ be a 
vertically stretched $(d,g)$-configuration. 
Let $t$ be sufficiently large. 
Let ${\Pcal^{\C}}\subset(\C^\times)^2$ be a configuration
of $3g-1+d$ points such that 
$\Log_t({\mathcal P^{\C}})=\Pcal$.
Then the composition 
\[
\gamma\mapsto \tDcal(\operatorname{Trop}_{\Pcal^\CC,t}(\gamma),\Pcal) 
\]
is a surjection from the set of irreducible complex algebraic
curves~$\gamma$ 
of degree $d$ and genus~$g$ passing through ${\mathcal P^{\C}}$
to the set of marked floor diagrams $\tDcal$ of degree~$d$ and genus~$g$. 
Under this map, the preimage of a marking~$\tDcal$  
of a labeled floor diagram~$\Dcal$ consists of 
$\mu(\Dcal)$ distinct complex curves.
Consequently, \eqref{eq:N=sum-mu-nu}~holds. 
\end{theorem}

\subsection{
Combinatorial rule for the relative Gromov-Witten invariants} 
In this section, we give a generalization of the combinatorial
correspondence theorem (Theorem~\ref{th:combin-correspondence}) 
to the problem of counting complex curves of given degree and genus
which, in addition to passing through a given point configuration, 
satisfy some prescribed tangency conditions. 
This will require a suitable generalization of the notion of a marking
of a labeled floor diagram. 

The tangency conditions we will be working with are described by
integer \emph{partitions}. 
In dealing with the latter, we will use the standard notational 
conventions of the combinatorial theory of partitions (see,
e.g., \cite[Section~1.3]{ec1} or \cite[Section~7.2]{ec2}) 
rather than those used by Caporaso and
Harris~\cite{caporaso-harris} in their classical treatment of the
subject. 
Recall that a \emph{partition} $\lambda\!=\!(\lambda_1,\lambda_2,\dots)$ 
of an integer $n\ge 0$ is a weakly decreasing sequence of
nonnegative integers~$\lambda_i$ 
whose sum is equal to~$n$: 
\[
|\lambda|=\lambda_1+\lambda_2+\cdots=n. 
\]
The nonzero $\lambda_i$'s are called the \emph{parts} of~$\lambda$.
The number of parts is the \emph{length} of~$\lambda$,
denoted~$\ell(\lambda)$. 
We write 
\begin{equation}
\label{eq:lambda-exp}
\lambda=\langle1^{\alpha_1} \,2^{\alpha_2}\cdots\rangle
\end{equation}
to express the fact that $\lambda$ has $\alpha_i$ parts equal to~$i$,
for each~$i$. 
Thus $\ell(\lambda)=\alpha_1+\alpha_2+\cdots$.


Our first goal is to extend Theorem~\ref{th:geom-correspondence} to 
a more general setting.
This will require some terminological preparation.

\begin{definition}[\emph{Tangency conditions described by pairs of
      partitions}] 
\label{def:tangency-two-partitions}
Let $\lambda$ and $\rho$ be two integer partitions. 
Let 
\begin{equation}
\label{eq:p-lambda}
\Pcal_\lambda=(p_1>p_2>\cdots>p_{\ell(\lambda)})
\end{equation}
be a configuration
of $\ell(\lambda)$ points in~$\R$, and 
let 
\[
\Pcal_\lambda^{\C}
=\{(x_1,0),(x_2,0),\dots,(x_{\ell(\lambda)},0)\}\subset
\C^{\times}\times\C
\]
be a configuration of points on the $x$-axis 
$\C^{\times}\times\{0\}\subset\C^{\times}\times\C$
such that 
$\log_t|x_i|=p_i$ for every~$i$. 
We say that a complex curve $\gamma$ has 
\emph{$(\lambda,\Pcal^{\C}_\lambda,\rho)$-tangency to the $x$-axis} if
$\gamma$~meets the latter 
at $\ell(\lambda)+\ell(\rho)$ points, as follows: 
\begin{itemize}
\item
$\gamma$ passes through $\Pcal_\lambda^{\C}$, 
with tangency to the $x$-axis of degree~$\lambda_i$ at each point $(x_i,0)$;
\item
$\gamma$ passes through some other $\ell(\rho)$ points on the
$x$-axis, with the degrees of tangency to the $x$-axis at those points
forming the partition~$\rho$. 
\end{itemize}
\end{definition}

\begin{definition}[\emph{Grounding conditions for plane tropical curves}] 
Let $h:C\to\R^2$ be a plane tropical curve of degree~$d$ and
genus~$g$. 
A \emph{ground elevator} of~$h$ 
is an unbounded elevator~$e$ 
(see Definition~\ref{def:elevators})
for which the vector~$\Delta(e)$ 
(see Definition~\ref{def:trop-morphism})
is a positive multiple of $(0,-1)$. 

We say that a tropical curve $h:C\to\R^2$
is \emph{$(\lambda,\Pcal_\lambda,\rho)$-grounded} (cf.~\eqref{eq:p-lambda})
if 
\begin{itemize}
\item
each point $(p_i,0)\in\Pcal_\lambda\times\{0\}$ is contained in a
ground elevator of weight~$\lambda_i$; 
\item
the weights of the remaining ground elevators form the partition~$\rho$.
\end{itemize}
See Figure~\ref{fig:two-trop-conics} for an example. 

We define the \emph{(complex) multiplicity} of such a tropical curve  
as the number 
\begin{equation}
\label{eq:mult-ground}
\mu_\rho(h)=\prod_e w^2(e) \prod_{i=1}^{\ell(\rho)} \rho_i\,,
\end{equation}
where $e$ runs over all bounded elevators. 
%
\end{definition}

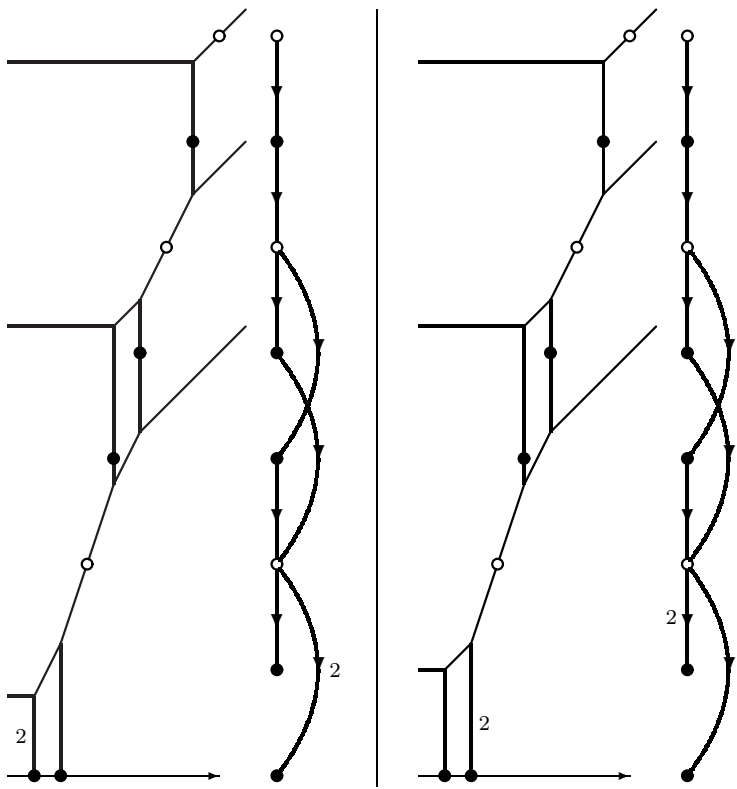
\begin{figure}[htbp]
\begin{center}
\begin{tabular}{c|c}
\setlength{\unitlength}{1pt}
\begin{picture}(90,290)(5,10)
\thicklines

\darkred{\linethickness{1pt}
\put(40,120){\line(0,1){60}}
\put(50,140){\line(0,1){50}}
\put(70,230){\line(0,1){50}}
}

\lightgreen{\linethickness{1pt}
\put(10,10){\line(0,1){30}}
\put(20,10){\line(0,1){50}}
}

\put(0,40){\line(1,0){10}}
\put(0,180){\line(1,0){40}}
\put(0,280){\line(1,0){70}}

\put(40,180){\line(1,1){10}}
\put(50,140){\line(1,1){40}}
\put(70,230){\line(1,1){20}}

\put(70,280){\line(1,1){8.6}}
\put(90,300){\line(-1,-1){8.6}}

\put(10,40){\line(1,2){10}}
\put(40,120){\line(1,2){10}}

\put(50,190){\line(1,2){9.1}}
\put(70,230){\line(-1,-2){9.1}}

\put(20,60){\line(1,3){9.4}}
\put(40,120){\line(-1,-3){9.4}}

\put(10,10){\circle*{4}}
\put(20,10){\circle*{4}}
\put(30,90){\circle{4}}
\put(40,130){\circle*{4}}
\put(50,170){\circle*{4}}
\put(60,210){\circle{4}}
\put(70,250){\circle*{4}}
\put(80,290){\circle{4}}

\thinlines
\put(0,10){\vector(1,0){80}}


\put(5,25){\makebox(0,0){$\scriptstyle{2}$}}

\end{picture}
\ \ 
\begin{picture}(20,290)(5,10)
\thicklines

\darkred{\linethickness{1pt}
\put(0,248){\line(0,-1){36}}
\put(0,230){\vector(0,-1){5}}
\put(0,288){\line(0,-1){36}}
\put(0,270){\vector(0,-1){5}}
\qbezier(1.2,91.6)(30,130)(1,168.4)
\put(16,130){\vector(0,-1){1}}
\put(0,208){\line(0,-1){36}}
\put(0,190){\vector(0,-1){5}}
\qbezier(1.2,131.6)(30,170)(1.2,208.4)
\put(16,170){\vector(0,-1){1}}
\put(0,128){\line(0,-1){36}}
\put(0,110){\vector(0,-1){5}}

}

\lightgreen{\linethickness{1pt}
\put(0,88){\line(0,-1){36}}
\put(0,70){\vector(0,-1){5}}
\qbezier(1.2,11.6)(30,50)(1,88.4)
\put(16,50){\vector(0,-1){1}}
}

\put(22,50){\makebox(0,0){$\scriptstyle{2}$}}

\put(0,10){\circle*{4}}
\put(0,50){\circle*{4}}
\put(0,90){\circle{4}}
\put(0,130){\circle*{4}}
\put(0,170){\circle*{4}}
\put(0,210){\circle{4}}
\put(0,250){\circle*{4}}
\put(0,290){\circle{4}}

\end{picture}
{\ \ }&{\ \ }
\setlength{\unitlength}{1pt}
\begin{picture}(90,290)(5,10)
\thicklines

\darkred{\linethickness{1pt}
\put(40,120){\line(0,1){60}}
\put(50,140){\line(0,1){50}}
\put(70,230){\line(0,1){50}}
}

\lightgreen{\linethickness{1pt}
\put(10,10){\line(0,1){40}}
\put(20,10){\line(0,1){50}}
}

\put(0,50){\line(1,0){10}}
\put(0,180){\line(1,0){40}}
\put(0,280){\line(1,0){70}}

\put(10,50){\line(1,1){10}}
\put(40,180){\line(1,1){10}}
\put(50,140){\line(1,1){40}}
\put(70,230){\line(1,1){20}}

\put(70,280){\line(1,1){8.6}}
\put(90,300){\line(-1,-1){8.6}}

\put(40,120){\line(1,2){10}}

\put(50,190){\line(1,2){9.1}}
\put(70,230){\line(-1,-2){9.1}}

\put(20,60){\line(1,3){9.4}}
\put(40,120){\line(-1,-3){9.4}}

\put(10,10){\circle*{4}}
\put(20,10){\circle*{4}}
\put(30,90){\circle{4}}
\put(40,130){\circle*{4}}
\put(50,170){\circle*{4}}
\put(60,210){\circle{4}}
\put(70,250){\circle*{4}}
\put(80,290){\circle{4}}

\thinlines
\put(0,10){\vector(1,0){80}}


\put(25,30){\makebox(0,0){$\scriptstyle{2}$}}

\end{picture}
\ \ 
\begin{picture}(20,290)(5,10)
\thicklines

\darkred{\linethickness{1pt}
\put(0,248){\line(0,-1){36}}
\put(0,230){\vector(0,-1){5}}
\put(0,288){\line(0,-1){36}}
\put(0,270){\vector(0,-1){5}}
\qbezier(1.2,91.6)(30,130)(1,168.4)
\put(16,130){\vector(0,-1){1}}
\put(0,208){\line(0,-1){36}}
\put(0,190){\vector(0,-1){5}}
\qbezier(1.2,131.6)(30,170)(1.2,208.4)
\put(16,170){\vector(0,-1){1}}
\put(0,128){\line(0,-1){36}}
\put(0,110){\vector(0,-1){5}}

}

\lightgreen{\linethickness{1pt}
\put(0,88){\line(0,-1){36}}
\put(0,70){\vector(0,-1){5}}
\qbezier(1.2,11.6)(30,50)(1,88.4)
\put(16,50){\vector(0,-1){1}}
}

\put(-6,70){\makebox(0,0){$\scriptstyle{2}$}}

\put(0,10){\circle*{4}}
\put(0,50){\circle*{4}}
\put(0,90){\circle{4}}
\put(0,130){\circle*{4}}
\put(0,170){\circle*{4}}
\put(0,210){\circle{4}}
\put(0,250){\circle*{4}}
\put(0,290){\circle{4}}

\end{picture}

\end{tabular}
\end{center}
\caption{Two tropical cubics passing through 6~vertically stretched points
and $(\lambda,\Pcal_\lambda,\rho)$-grounded
with $\lambda\!=\!\emptyset$, $\Pcal\!=\!\emptyset$, and $\rho\!=\!(2,1)$. 
The multiplicity of each tropical cubic is~2; 
thus $N_{3,1}(\emptyset,(2,1))\!=\!4$. 
}
\label{fig:two-trop-conics}
\end{figure}

Theorem~\ref{c-correspondence-alphabeta} below is a slight
generalization of Theorem~\ref{th:geom-correspondence} 
(which is a special case of the correspondence theorem
\cite[Theorem~1]{mikhalkin-05}), 
and can be proved in a similar way.
The proof will appear in~\cite{mikhalkin-08}. 

\begin{theorem}
\label{c-correspondence-alphabeta}
Let $\Pcal$ be a generic configuration of $2d-1+g+\ell(\rho)$ points
in~$\R^2$ satisfying the condition~\eqref{eq:vert-stretched}. 
Let~${\mathcal P^{\C}}\subset(\C^\times)^2$ be a configuration of
points such that $\Log_t({\mathcal P^{\C}})=\Pcal$.
Let configurations $\Pcal_\lambda$ and $\Pcal^\CC_\lambda$ be as in 
Definition~\ref{def:tangency-two-partitions}. 
Let $t\in\RR_{>0}$ be sufficiently large.

Then there is a canonical surjection 
from the set of irreducible complex algebraic curves
of degree $d$ and genus $g$ passing through ${\mathcal P^{\C}}$
and having $(\lambda,\Pcal^{\C}_\lambda,\rho)$-tangency to the $x$-axis
to the set of irreducible $(\lambda,\Pcal_\lambda,\rho)$-grounded
plane tropical curves of degree~$d$ and genus~$g$ passing
through~$\Pcal$. 
Under this surjection, the preimage of each such tropical curve~$h$ 
consists of $\mu_\rho(h)$ distinct complex curves.
\end{theorem}

We generalize Definition~\ref{def-marking}, as follows. 
(A~similar definition, with $g=0$, appeared in
\cite[Definition~4.2]{arroyo-brugalle-lopez}.) 

\begin{definition}[\emph{
Floor diagrams marked by pairs of partitions}]
Let $\Dcal$ be a labeled floor diagram of degree~$d$ and genus~$g$. 
Let $\lambda$ and $\rho$ be two partitions with $|\lambda|+|\rho|=d$.
A~$(\lambda,\rho)$-marking of~$\Dcal$ is a combinatorial decoration
of~$\Dcal$ obtained by the following modification of the procedure
used in Definition~\ref{def-marking}. 
We will illustrate the steps of this procedure using a running
example, in which $d=4$, $g=1$, $\lambda=(2)$, $\rho=(1,1)$,
and $\Dcal$ is the labeled floor diagram from
Example~\ref{example:lfd4}:
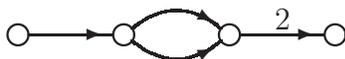
\[
\begin{picture}(80,15)(30,-4)\setlength{\unitlength}{4pt}\thicklines
\oooo\Eee\eOe\eeE
\put(25,1.5){\makebox(0,0){$2$}}
\put(7,0){\vector(1,0){1}}
\put(17.5,1.75){\vector(2,-1){1}}
\put(17.5,-1.75){\vector(2,1){1}}
\put(27,0){\vector(1,0){1}}
\end{picture}
\]

\textbf{Step~0.}
Introduce $\ell(\lambda)$ new vertices denoted
$v_1,\dots,v_{\ell(\lambda)}$:  
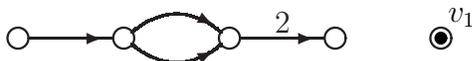
\[
\begin{picture}(80,20)(30,-4)\setlength{\unitlength}{4pt}\thicklines
\oooo\Eee\eOe\eeE
\put(25,1.5){\makebox(0,0){$2$}}
\put(7,0){\vector(1,0){1}}
\put(17.5,1.75){\vector(2,-1){1}}
\put(17.5,-1.75){\vector(2,1){1}}
\put(27,0){\vector(1,0){1}}
\put(40,0){\circle{2}}
\put(40,0){\circle*{1}}
\put(42,2){\makebox(0,0){$v_1$}}
\end{picture}
\]

\textbf{Step~1.}
For each original vertex $v$ in~$\Dcal$, introduce some number (possibly none)
of additional new vertices, and connect $v$ to each of them 
by a single edge directed away from~$v$.
In addition, introduce some (possibly none) edges directed from $v$
to $v_1,\dots,v_{\ell(\lambda)}$. 
Assign positive integer weights to all these new edges, so that
\begin{itemize}
\item
for each original vertex $v$ in~$\Dcal$, 
the total weight of all new edges (of both kinds) originating
at~$v$ is equal to $1-\divv(v)$;
\item
the weights of all edges arriving at $v_1,\dots,v_{\ell(\lambda)}$ 
form the partition~$\lambda$; 
\item
the weights of all other new edges form the partition~$\rho$. 
\end{itemize}
%
%
Thus, the total weight of all new edges is
equal to $\sum_v (1-\divv(v))=d=|\lambda|+|\rho|$. 
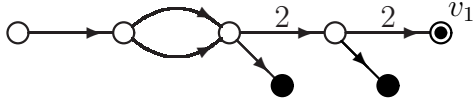
\begin{center}
\begin{picture}(80,45)(30,-25)\setlength{\unitlength}{4pt}\thicklines
\oooo\Eee\eOe\eeE
\put(25,1.5){\makebox(0,0){$2$}}
\put(35,1.5){\makebox(0,0){$2$}}
\put(7,0){\vector(1,0){1}}
\put(17.5,1.75){\vector(2,-1){1}}
\put(17.5,-1.75){\vector(2,1){1}}
\put(27,0){\vector(1,0){1}}
\put(37,0){\vector(1,0){1}}
\put(31,0){\line(1,0){8}}

\put(20.7,-0.7){\line(1,-1){4}}
\put(20.7,-0.7){\vector(1,-1){3}}
\put(25,-5){\circle*{2}}
\put(30.6,-0.8){\line(1,-1){4}}
\put(30.6,-0.8){\vector(1,-1){2.5}}
\put(35,-5){\circle*{2}}
\put(40,0){\circle{2}}
\put(40,0){\circle*{1}}
\put(42,2){\makebox(0,0){$v_1$}}

\end{picture}
\end{center}

\textbf{Step~2.}
Split each original edge $e$ of $\Dcal$ in two,
by inserting an extra vertex in the middle of~$e$;
the resulting two edges inherit $e$'s orientation and weight. 
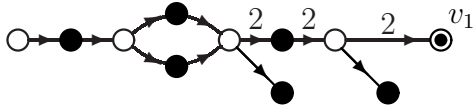
\begin{center}
\begin{picture}(80,45)(30,-25)\setlength{\unitlength}{4pt}\thicklines
\oooo\Eee\eOe\eeE
\put(22.5,2){\makebox(0,0){$2$}}
\put(27.5,2){\makebox(0,0){$2$}}
\put(35,1.5){\makebox(0,0){$2$}}
\put(37,0){\vector(1,0){1}}
\put(31,0){\line(1,0){8}}
\put(2.5,0){\vector(1,0){1}}
\put(7.5,0){\vector(1,0){1}}
\put(17.5,1.75){\vector(2,-1){1}}
\put(17.5,-1.75){\vector(2,1){1}}
\put(12.5,1.75){\vector(2,1){1}}
\put(12.5,-1.75){\vector(2,-1){1}}
\put(22.5,0){\vector(1,0){1}}
\put(27.5,0){\vector(1,0){1}}
\put(5,0){\circle*{2}}
\put(15,2.5){\circle*{2}}
\put(15,-2.5){\circle*{2}}
\put(25,0){\circle*{2}}
\put(20.7,-0.7){\line(1,-1){4}}
\put(20.7,-0.7){\vector(1,-1){3}}
\put(25,-5){\circle*{2}}
\put(30.6,-0.8){\line(1,-1){4}}
\put(30.6,-0.8){\vector(1,-1){2.5}}
\put(35,-5){\circle*{2}}
\put(40,0){\circle{2}}
\put(40,0){\circle*{1}}
\put(42,2){\makebox(0,0){$v_1$}}

\end{picture}
\end{center}
\noindent
Let $\widetilde G=(\widetilde V,\widetilde E)$ denote the acyclic
directed graph obtained after Steps~1--2.

\textbf{Step~3.}
Extend the linear ordering on the vertices of~$\Dcal$ to the
set~$\widetilde V$ 
so that, as before, each edge in~$\widetilde E$
is directed from a smaller to a larger vertex.
We also require $v_1$ to be the maximal element under
the linear ordering, $v_2$~the second largest, etc.: 
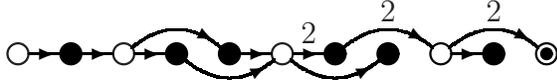
\begin{center}
\begin{picture}(80,45)(30,-15)\setlength{\unitlength}{4pt}\thicklines
\put(35,4){\makebox(0,0){$2$}}
\put(45,4){\makebox(0,0){$2$}}
\put(27.5,2){\makebox(0,0){$2$}}
\multiput(0,0)(10,0){2}{\circle{2}}
\multiput(25,0)(10,0){1}{\circle{2}}
\multiput(5,0)(10,0){2}{\circle*{2}}
\multiput(20,0)(10,0){2}{\circle*{2}}
\multiput(45,0)(10,0){1}{\circle*{2}}
\put(35,0){\circle*{2}}
\put(40,0){\circle{2}}
\put(50,0){\circle{2}}
\put(50,0){\circle*{1}}
\put(1,0){\line(1,0){8}}
\put(11,0){\line(1,0){3}}
\put(2.5,0){\vector(1,0){1}}
\put(7.5,0){\vector(1,0){1}}
\put(12.5,0){\vector(1,0){1}}
\qbezier(10.8,0.6)(15,4)(19.2,0.6)
\put(17.5,1.75){\vector(2,-1){1}}
\qbezier(15.8,-0.6)(20,-4)(24.2,-0.6)
\put(22.5,-1.75){\vector(2,1){1}}
\qbezier(25.8,-0.6)(30,-4)(34.2,-0.6)
\put(32.5,-1.75){\vector(2,1){1}}
\put(21,0){\line(1,0){3}}
\put(22.5,0){\vector(1,0){1}}
\put(26,0){\line(1,0){3}}
\put(27.5,0){\vector(1,0){1}}
\put(41,0){\line(1,0){3}}
\put(42.5,0){\vector(1,0){1}}

\qbezier(30.8,0.6)(35,4)(39.2,0.6)
\put(37.5,1.75){\vector(2,-1){1}}
\qbezier(40.8,0.6)(45,4)(49.2,0.6)
\put(47.5,1.75){\vector(2,-1){1}}

\end{picture}
\end{center}
\noindent
The resulting object $\tDcal$ is called a 
$(\lambda,\rho)$-\emph{marked floor diagram}
(of degree~$d$ and genus~$g$),
or a $(\lambda,\rho)$-\emph{marking} of the original labeled floor
diagram~$\Dcal$.
It is easy to see that $\tDcal$ has $2d+g-1+\ell(\lambda)+\ell(\rho)$
vertices and $2d+2g-2+\ell(\lambda)+\ell(\rho)$ edges. 
\end{definition}

The number of distinct $(\lambda,\rho)$-markings
of~$\Dcal$ is denoted by~$\nu_{\lambda,\rho}(\Dcal)$. 
We also denote (cf.~\eqref{eq:mult-ground})
\begin{equation}
\label{eq:mu-rho}
\mu_\rho(\tDcal)=\mu_\rho(\Dcal)=\mu(\Dcal)
\prod_{i=1}^{\ell(\rho)} \rho_i\,. 
\end{equation}

\begin{remark}
\label{rem:mfd-rho-simple}
Let $\lambda=\emptyset$ and $\rho=\langle 1^d \rangle$. 
In this special case, 
we recover the ordinary notion of a marked floor diagram
introduced in Definition~\ref{def-marking}. 
We also recover $\nu_{\emptyset,\langle 1^d \rangle}(\Dcal)=\nu(\Dcal)$
and $\mu_{\emptyset,\langle 1^d \rangle}(\Dcal)=\mu(\Dcal)$. 

More generally, let $\lambda=\langle1^k\rangle$ and 
$\rho=\langle1^{d-k}\rangle$.
Then a $(\lambda,\rho)$-marked floor diagram is nothing but an
ordinary marked floor diagram whose last $k$ vertices are sinks. 
\end{remark}


\begin{definition}[\emph{Relative Gromov-Witten invariants}] 
\label{def:G-W-two-partitions}
Let $N_{d,g}(\lambda,\rho)$ denote the number of irreducible complex
algebraic curves of degree~$d$ and genus~$g$ passing through 
a generic configuration 
of $2d-1+g+\ell(\rho)$ points
in~$\CC^2$ 
and having $(\lambda,\Pcal^{\C}_\lambda,\rho)$-tangency to the
$x$-axis (see Definition~\ref{def:tangency-two-partitions}), 
for a given generic collection $\Pcal^{\C}_\lambda$
of $\ell(\lambda)$ points on $\CC\times\{0\}$. 
\end{definition}

\pagebreak[3]

Thus $N_{d,g}(\lambda,\rho)$ counts irreducible plane curves $\gamma$
of given degree and genus intersecting 
a given line~$L$ (say the $x$-axis) with multiplicities described
by~$\lambda$ at a given collection of points on~$L$, 
and with multiplicities described
by~$\rho$ at some other unspecified points on~$L$;
in addition, $\gamma$~must pass through a  generic
configuration of (appropriately many) points on the plane. 

We note that the numbers $N_{d,g}(\lambda,\rho)$ 
are different from the numbers 
$N^{d,\delta}(\alpha,\beta)$ 
studied by Caporaso and Harris~\cite{caporaso-harris}
since $N_{d,g}(\lambda,\rho)$ only counts irreducible curves. 
The Caporaso-Harris numbers are the generalizations of the Severi
degrees discussed in Section~\ref{sec:Ndg+Severi}, and can be
similarly expressed as sums of the numbers
$N_{d,g}(\lambda,\rho)$ (with positive integer coefficients)
by considering possible partitions of a given configuration into
subconfigurations lying on the irreducible components of a curve. 
Conversely, the numbers $N_{d,g}(\lambda,\rho)$ can be computed from
the Caporaso-Harris numbers 
(also known as generalized Severi degrees) 
by an appropriate inclusion-exclusion procedure. 


\begin{remark}
\label{rem:Ndg-simple}
Let $\lambda$ be a partition with $|\lambda|\le d$, and 
let $\rho=\langle1^{d-|\lambda|}\rangle$.
Then the tan\-gency conditions associated with~$\rho$ become vacuous. 
Consequently, the relative Gro\-mov-Witten invariant 
$N_{d,g}(\lambda,\langle1^{d-|\lambda|}\rangle)$ counts 
irreducible plane curves 
of degree~$d$ and genus~$g$ passing through 
a generic configuration of $\ell(\lambda)$ points on a given line
with tangencies to the line described by the partition~$\lambda$, 
and also passing through a generic
configuration of $3d-1+g-|\lambda|$ points on the plane. 

In particular, 
the number
$N_{d,g}(\langle 1^{k}\rangle,\langle1^{d-k}\rangle)$ counts 
irreducible curves 
of degree~$d$ and genus~$g$ passing through 
a configuration of $3d-1+g$ points on the plane that contains $k$
collinear points and is otherwise generic. 
For $k\le 2$, we obviously get
\begin{equation}
\label{eq:Ndg-11}
N_{d,g}(\emptyset,\langle1^d\rangle)
=N_{d,g}((1),\langle1^{d-1}\rangle)
=N_{d,g}((1,1),\langle1^{d-2}\rangle)
=N_{d,g}\,. 
\end{equation}
\end{remark}

\medskip

At this point, the following analogue of Theorem~\ref{th:trop=combin}
should come as no surprise. 
(A~proof can be given along similar lines.)

\begin{theorem}
\label{th:trop=combin-two-ptns}
Under the assumptions of Theorem~\ref{c-correspondence-alphabeta}, 
there is a multiplicity-pre\-serving bijection $h\mapsto\tDcal$
(that is, $\mu_\rho(h)\!=\!\mu_\rho(\tDcal)$) 
between the plane tropical curves~$h$  described in that theorem 
and the $(\lambda,\rho)$-marked floor diagrams of degree~$d$
and genus~$g$. 
\end{theorem}

Theorems~\ref{c-correspondence-alphabeta} 
and~\ref{th:trop=combin-two-ptns} imply the following
generalization of Theorem~\ref{th:bm-curves}. 
(The case $g=0$ has been stated
in~\cite[Theorem~4.4]{arroyo-brugalle-lopez}.) 

\begin{theorem}
\label{th:combin-rule-two-ptns}
The relative Gromov-Witten invariant $N_{d,g}(\lambda,\rho)$ is given
by 
\begin{equation}
\label{eq:combin-rule-two-ptns}
N_{d,g}(\lambda,\rho)=\sum_{\Dcal}\mu_\rho(\Dcal)\,\nu_{\lambda,\rho}(\Dcal)
\,,
\end{equation}
the sum over all labeled floor diagrams $\Dcal$ of
degree~$d$ and genus~$g$.
\end{theorem}

In other words, the number $N_{d,g}(\lambda,\rho)$ can be computed by
counting $(\lambda,\rho)$-marked (labeled) floor diagrams~$\tDcal$, 
each taken with multiplicity~$\mu_\rho(\tDcal)$.

\pagebreak[4]

\section{Computing relative Gromov-Witten invariants}
\label{sec:Ndg-rel-examples}

This section is devoted to applications of the
combinatorial rule of Theorem~\ref{th:combin-rule-two-ptns}.  

\subsection{Conics and cubics}

As a warm-up, let us look at the cases $d=2$ and $d=3$.
(All of these enumerative invariants have been known since the 19th
century.) 

\begin{example}[\emph{
Plane conics}]
For $d=2$ and $g=0$, there are very few possibilities.
By~\eqref{eq:Ndg-11}, we have
\[
N_{2,0}(\emptyset,(1,1))
=N_{2,0}((1),(1))
=N_{2,0}((1,1),\emptyset)
=N_{2,0}=1, 
\]
the unique plane conic through 5~generic points.
This corresponds to the unique
$(\lambda,\rho)$-marked floor diagram of multiplicity~$1$. 
In the cases $(\lambda,\rho)=((2),\emptyset)$ and
$(\lambda,\rho)=(\emptyset,(2))$, the diagram is unique as well;
the multiplicities are $1$ and~$2$, respectively,
so $N_{2,0}((2),\emptyset)=1$ and $N_{2,0}(\emptyset,(2))=2$. 
This accounts~for: 
\begin{itemize}
\item
the unique plane conic passing through $3$ generic 
points and tangent to a given line at a given point; 
\item
two plane conics passing through $4$ given
points and tangent to a given line. 
\end{itemize}
\end{example}

\begin{example}[\emph{
Elliptic plane cubics}]
For $d=3$ and $g=1$, 
there is only one labeled floor diagram,
so combinatorial calculations are very simple. 
Applying Theorem~\ref{th:combin-rule-two-ptns}, we see that all
relative Gromov-Witten invariants $N_{3,1}(\lambda,\rho)$ are equal
to~$1$, except~for: 

$N_{3,1}((1),(2))=2$ 
(plane cubics passing through $7$ generic 
points and having two distinct common points with a given line,
namely a given intersection point and an unspecified 
point of tangency);

$N_{3,1}(\emptyset,(2,1))=4$
(plane cubics passing through $8$ generic 
points and tangent to a given line; cf.\ Figure~\ref{fig:two-trop-conics});

$N_{3,1}(\emptyset,(3))=3$
(plane cubics passing through 7 generic points and having an
unspecified point of order-3 tangency to a given line). 
\end{example}

\begin{example}[\emph{
Rational plane cubics}]
There are $3$ labeled floor diagrams~$\Dcal$ of degree $d=3$ and genus
$g=0$.
By~\eqref{eq:Ndg-11}, we have
\[
N_{3,0}(\emptyset,(1,1,1))
=N_{3,0}((1),(1,1))
=N_{3,0}((1,1),(1))
=N_{3,0}=12 
\]
(12~rational plane cubics through 8 generic points);
the combinatorial calculation is the same as in Example~\ref{example:N30}. 
The remaining cases are presented in~Figure~\ref{fig:d=3,g=0}. 
For each labeled floor diagram~$\Dcal$ and each pair of partitions
$(\lambda,\rho)$, the table shows the corresponding contribution 
$\mu_\rho(\Dcal)\cdot\nu_{\lambda,\rho}(\Dcal)$
to the right-hand side of~\eqref{eq:combin-rule-two-ptns}. 
These contributions are then added together to
get~$N_{\lambda,\rho}$.
For example, there are $N_{3,0}(\emptyset,(3))=21$ rational cubics passing
through 6~generic points and having tangency of order~3 to a given line
(at an unspecified point of inflection). 
\end{example}

\newcommand{\lrho}[2]{
\hspace{-.05in}\begin{tabular}{l}\hspace{-.01in}$\lambda\!=\!#1$\\ $\rho\!=\!#2$\end{tabular}
\hspace{-.13in}}

\begin{figure}[htbp]
\begin{center}
\begin{tabular}{c|c|c|c|c|c|c|c}
& \lrho{\emptyset}{(2,1)} & \lrho{\emptyset}{(3)} & \lrho{(1)}{(2)} 
& \lrho{(2)}{(1)} & \lrho{\langle1^3\rangle}{\emptyset} 
& \lrho{(2,1)}{\emptyset} & \lrho{(3)}{\emptyset} 
\\
\hline \hline
&&&&&&&\\[-.15in]
\begin{picture}(55,8)(0,-3)\setlength{\unitlength}{2.5pt}\thicklines
\ooo\Eee\eEe
\end{picture}
& $2\cdot4$ & $3\cdot0$ & $2\cdot1$ & $1\cdot3$ & $1\cdot3$ &
$1\cdot1$ & $1\cdot0$\\[.2in] 
\begin{picture}(55,8)(0,-3)\setlength{\unitlength}{2.5pt}\thicklines
\ooo\Eee\eEe
\put(15,2){\makebox(0,0){$\scriptstyle 2$}}
\end{picture}
& $8\cdot2$ & $12\cdot1$ & $8\cdot1$ & $4\cdot1$ & $4\cdot1$ &
$4\cdot1$ & $4\cdot1$\\[.2in] 
\begin{picture}(55,8)(0,-3)\setlength{\unitlength}{2.5pt}\thicklines
\ooo\EEe\eEe
\end{picture}
& $2\cdot6$ & $3\cdot3$ & $2\cdot3$ & $1\cdot3$ & $1\cdot3$ &
$1\cdot3$ & $1\cdot3$ 
\\
\hline \hline
&&&&&&&\\[-.1in]
$N_{3,0}(\lambda,\rho)$& 36& 21& 16 & 10 & 10 & 8 & 7
\\[-.05in]
\end{tabular}
\end{center}
\caption{Combinatorial computation of 
the numbers $N_{3,0}(\lambda,\rho)$}
\label{fig:d=3,g=0}
\end{figure}


\subsection{Nonsingular and uninodal curves}
\label{sec:uninodal}
Nonsingular (or generic) plane algebraic curves have the
maximal possible genus among all curves of degree~$d$, namely
\[
g_{\max}=g_{\max}(d)=\frac{(d-1)(d-2)}{2}.
\]
There is only one labeled floor diagram $\Dcal$ of degree~$d$ and
genus~$g_{\max}$, namely one that looks like this 
(shown for $d=6$ and $g_{\max}=10$):
\begin{equation}
\begin{picture}(60,14)(30,-5)
\setlength{\unitlength}{3pt}
\thicklines
\multiput(0,0)(10,0){6}{\circle{2}}
\Eee\eOe\eeO\eeE
\qbezier(30.8,0.6)(35,4)(39.2,0.6)
\qbezier(30.8,-0.6)(35,-4)(39.2,-0.6)
\qbezier(30.95,0.3)(35,1.5)(39.05,0.3)
\qbezier(30.95,-0.3)(35,-1.5)(39.05,-0.3)

\put(41,0){\line(1,0){8}}
\qbezier(40.8,0.6)(45,4)(49.2,0.6)
\qbezier(40.8,-0.6)(45,-4)(49.2,-0.6)
\qbezier(40.95,0.3)(45,2)(49.05,0.3)
\qbezier(40.95,-0.3)(45,-2)(49.05,-0.3)
\end{picture}
\end{equation}
For this floor diagram, we have $\mu(\Dcal)=\nu(\Dcal)=1$,
implying $N_{d,g_{\max}}=1$ (cf.~\eqref{eq:Nd,gmax}). 
For a pair of partitions $\lambda$ and $\rho$ with 
$|\lambda|+|\rho|=d$, the number of markings 
$\nu_{\lambda,\rho}(\Dcal)$ is equal to the number $o(\rho)$ of distinct
permutations of the parts of~$\rho$.
That is, if 
\begin{equation}
\label{eq:rho-exp}
\rho=\langle 1^{\beta_1} \,2^{\beta_2}\cdots\rangle
\end{equation}
(cf.~\eqref{eq:lambda-exp}),
then 
\begin{equation}
\label{eq:nu-gmax}
\nu_{\lambda,\rho}(\Dcal)
=o(\rho)=\frac{(\ell(\rho))!}{\beta_1!\,\beta_2!\cdots}
\end{equation}
Combining \eqref{eq:nu-gmax} with~\eqref{eq:mu-rho}
and~\eqref{eq:combin-rule-two-ptns}, we obtain the following formula.

\begin{corollary}
\label{cor:nonsing-rel}
Let $\lambda$ and $\rho$ be partitions such that 
$|\lambda|+|\rho|=d$. 
Then the number of 
irreducible plane complex algebraic
curves of degree~$d$ passing through a generic configuration of
$\frac{d(d+1)}{2}+\ell(\rho)$ points and having tangencies to a given
line described by $\lambda$ and~$\rho$
(that is, $\lambda$~describes tangencies at given points
whereas $\rho$~describes tangencies at unspecified points)
is equal to
\begin{equation}
\label{eq:Ndg-max-lr}
N_{d,g_{\max}}(\lambda,\rho)
=\rho_1\,\rho_2\cdots \frac{(\ell(\rho))!}{\beta_1!\,\beta_2!\cdots}\,. 
\end{equation}
\end{corollary}

Note that this number does not depend on~$\lambda$. 

\medskip

Let us now turn to counting \emph{uninodal} curves, i.e., those of genus 
\[
g=g_{\max}-1=\frac{d(d-3)}{2} \qquad (d\ge3).
\]
It is easy to see that there are $2d-3$ labeled floor diagrams of
degree~$d$ and genus~$g_{\max}-1$
(cf.~\cite[Proposition~6.1]{brugalle-mikhalkin08});
the $5$~diagrams for $d=4$ and $g=2$ are shown in 
the second table of Appendix~A. 
More specifically, there are $d-1$ diagrams 
$\Dcal_1,\dots,\Dcal_{d-1}$ with $\mu(\Dcal_i)=1$
and $\nu(\Dcal_i)=2i+1$; 
and $d-2$ additional diagrams 
$\Dcal'_1,\dots,\Dcal'_{d-2}$ with $\mu(\Dcal'_i)=4$
and $\nu(\Dcal'_i)=i$. 
This gives
\begin{equation}
\label{eq:3(d-1)^2}
N_{d,g_{\max}-1}=(3+5+\cdots+(2d-1))+4(1+2+\cdots+(d-2))
=3(d-1)^2;
\end{equation}
cf.~\eqref{eq:Nd,gmax-1}. 

We next generalize the formula \eqref{eq:3(d-1)^2} to the setting
involving tangency conditions. 

\begin{corollary}
\label{cor:uninodal-rel}
Let $\lambda=\langle1^{\alpha_1} \,2^{\alpha_2}\cdots\rangle$ and 
$\rho=\langle 1^{\beta_1} \,2^{\beta_2}\cdots\rangle$ be partitions such that 
$|\lambda|+|\rho|=d$. 
Then the number of uninodal irreducible plane complex algebraic
curves of degree~$d$ passing through a generic configuration of
$\frac{(d-1)(d+2)}{2}+\ell(\rho)$ points and having tangencies to a given
line described by $\lambda$ and~$\rho$
is given by the formula 
\begin{equation*}
N_{d,g_{\max}-1}(\lambda,\rho)
=
\begin{cases}
\Bigl(
(d\!-\!2)(3d\!-\!2)+\alpha_1+\beta_1+(d\!-\!1)\dfrac{\beta_1}{\ell(\rho)}
\Bigr)
  N_{d,g_{\max}}(\lambda,\rho) & \text{if $\rho\neq\emptyset$;}\\
\hspace{.1in}
(\!d-\!2)(3d\!-\!2)+\alpha_1& \text{if $\rho=\emptyset$. }
\end{cases}
\end{equation*}
(Recall that $N_{d,g_{\max}}(\lambda,\rho)$ is given
by~\eqref{eq:Ndg-max-lr}.)
\end{corollary}

\begin{proof}
We need to compute the quantities $\mu_\rho$ and $\nu_{\lambda,\rho}$
for each of the diagrams $\Dcal_i$ and~$\Dcal'_i$,
then use~\eqref{eq:combin-rule-two-ptns}. 
First, a little preparation. If $\beta_1\ge1$, then denote 
\[
\bar\rho=\langle 1^{\beta_1-1}
\,2^{\beta_2}\,3^{\beta_3}\cdots\rangle,  
\]
so that (cf.~\eqref{eq:nu-gmax}) 
\[
o(\bar\rho)=\frac{(\ell(\rho)-1)!}{(\beta_1-1)!\,\beta_2!\cdots}
=\frac{\beta_1}{\ell(\rho)}\,o(\rho).  
\]
Now, calculations give (cf.~\eqref{eq:nu-gmax}):
\begin{align*}
\mu_\rho(\Dcal_i)&=\rho_1\rho_2\cdots,\\
\mu_\rho(\Dcal'_i)&=4\rho_1\rho_2\cdots,\\
\nu_{\lambda,\rho}(\Dcal_i)&
            =(2i+1)\, o(\rho) 
	    \qquad (i\leq d-2),\\
\nu_{\lambda,\rho}(\Dcal_{d-1})&
            =(d-1)\,\frac{\beta_1}{\ell(\rho)}\,o(\rho)
	    +(\alpha_1+\beta_1)\,o(\rho),\\ 
\nu_{\lambda,\rho}(\Dcal'_i)&=i\,o(\rho)\qquad (i\leq d-2). 
\end{align*}
Collecting everything, we get
\begin{align*}
N_{d,g_{\max}-1}(\lambda,\rho)
&=\rho_1\rho_2\cdots\,o(\rho)\biggl(\sum_{i=1}^{d-2}(6i+1)
+(d-1)\,\frac{\beta_1}{\ell(\rho)}+\alpha_1+\beta_1
\biggr)\\
&=N_{d,g_{\max}}(\lambda,\rho)(3d^2-8d+4+(d-1)\,
\frac{\beta_1}{\ell(\rho)}+\alpha_1+\beta_1).
\qedhere
\end{align*}
\end{proof}

\subsection{Curves passing through collinear points}
Let us next consider the cases $(\lambda,\rho)=(\langle
1^{k}\rangle,\langle1^{d-k}\rangle)$ discussed in
Remarks~\ref{rem:mfd-rho-simple} and~\ref{rem:Ndg-simple}. 
Combining the latter with Theorem~\ref{th:combin-rule-two-ptns},
we obtain the following corollary. 

\begin{corollary} 
\label{cor:Ndg-sinks}
The number $N_{d,g}(\langle 1^{k}\rangle,\langle1^{d-k}\rangle)$ 
of irreducible plane complex algebraic curves of degree~$d$
and genus~$g$ passing through a generic configuration of 
$3d+g-k-1$ points and a generic
configuration of $k$ collinear points is equal to
$\sum_{\tDcal} \mu(\tDcal)$, 
the sum over marked floor diagrams $\tDcal$ of
degree~$d$ and genus~$g$ whose last $k$ vertices are sinks. 
\end{corollary}


The special cases $k=0,1,2$ of Corollary~\ref{cor:Ndg-sinks}  
yield the ordinary Gromov-Witten numbers (cf.~\eqref{eq:Ndg-11}). 
Let us examine the case $k=3$. 

\begin{corollary} 
\label{cor:collinear-triple}
The number 
of irreducible plane complex algebraic curves of degree~$d$
and genus~$g$ passing through a generic configuration of 
$3d+g-4$ points and through a generic
triple of collinear points is 
given by the formula 
\[
N_{d,g}(\langle 1^{3}\rangle,\langle
1^{d-3}\rangle)=N_{d,g}-(d-1)N_{d-1,g}\,. 
\]
\end{corollary}

For example, there are 
$N_{3,0}(\langle 1^{3}\rangle,\emptyset)=N_{3,0}-2N_{2,0}=10$
irreducible plane rational cubics passing through 5 generic points and
3 generic collinear points---matching the value in
Figure~\ref{fig:d=3,g=0}. 

\begin{proof}
Apply Corollary~\ref{cor:Ndg-sinks} with $k=3$. 
Note that in every (ordinary) marked floor diagram,
the last two vertices are sinks. 
Hence $N_{d,g}(\langle 1^{3}\rangle,\langle1^{d-3}\rangle)$ 
is equal to $N_{d,g}$
minus $\sum_{\tDcal} \mu(\tDcal)$, 
the sum over marked floor diagrams $\tDcal$ of
degree~$d$ and genus~$g$ in which the 3rd largest vertex~$v$ is not a
sink. Such a diagram $\tDcal$ looks like this
(we only draw the edges of interest to us): 
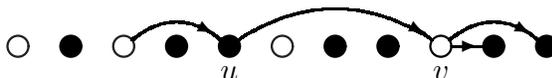
\begin{center}
\begin{picture}(80,35)(30,-15)
\setlength{\unitlength}{4pt}\thicklines
\multiput(0,0)(10,0){2}{\circle{2}}
\multiput(25,0)(10,0){1}{\circle{2}}
\multiput(5,0)(10,0){2}{\circle*{2}}
\multiput(20,0)(10,0){2}{\circle*{2}}
\multiput(45,0)(10,0){1}{\circle*{2}}
\put(35,0){\circle*{2}}
\put(40,0){\circle{2}}
\put(50,0){\circle*{2}}
\put(41,0){\line(1,0){3}}
\put(42.5,0){\vector(1,0){1}}

\qbezier(10.8,0.6)(15,4)(19.2,0.6)
\put(17.5,1.75){\vector(2,-1){1}}
\qbezier(20.8,0.6)(30,6.5)(39.2,0.6)
\put(37.5,1.75){\vector(2,-1){1}}
\qbezier(40.8,0.6)(45,4)(49.2,0.6)
\put(47.5,1.75){\vector(2,-1){1}}
\put(40,-2.5){\makebox(0,0){$v$}}
\put(20,-2.5){\makebox(0,0){$u$}}

\end{picture}
\end{center}
\noindent
(Since there are two edges pointing away from~$v$, there must be a
unique edge $u\to v$ arriving at~$v$, by the divergence condition.)
Removing the three edges incident to~$v$ creates a marked floor
diagram $\tDcal'$ of degree~$d-1$ and genus~$g$, with a distinguished sink~$u$.
Conversely, given such a marked diagram~$\tDcal'$ with a sink~$u$ in it,
$\tDcal$~is uniquely recovered. Note that each~$\tDcal'$ has $d-1$
sinks. Furthermore, $\mu(\tDcal')=\mu(\Dcal)$, and the claim follows. 
\end{proof}




\subsection{
Curves with maximal tangency to a given line}
%
Let us now look at the problem of counting irreducible plane
curves of degree~$d$ and genus~$g$ passing through 
appropriately many points and having maximal
tangency (of order~$d$) to a given line.  
The corresponding relative Gromov-Witten invariants come in two flavors,
depending on whether the point of tangency is prescribed or not:
\begin{itemize}
\item
$N_{d,g}((d),\emptyset)$ is the number of irreducible plane
curves of degree~$d$ and genus~$g$ which pass through a generic
configuration of $2d+g-1$ points and have tangency of order~$d$ to a
given line~$L$ at a given point $x\in L$; 

\item 
$N_{d,g}(\emptyset,(d))$ is the number of irreducible plane
curves of degree~$d$ and genus~$g$ which pass through a generic
configuration of $2d+g$ points and have tangency of order~$d$ to a given
line~$L$ at some point $x\in L$. 
\end{itemize}
One surprising corollary of Theorem~\ref{th:combin-rule-two-ptns}
is that these two numbers are related to each other in a very simple
way. 

\begin{corollary}
\label{cor:Nd00d}
We have $N_{d,g}(\emptyset,(d))=d\cdot N_{d,g}((d),\emptyset)$. 
\end{corollary}

\begin{proof}
It is easy to see that the corresponding two sets of
$(\lambda,\rho)$-marked floor diagrams are the same
whereas their multiplicities $\mu_\rho$ differ by a factor of~$d$. 
\end{proof}

As we learned from R.~Vakil, Corollary~\ref{cor:Nd00d} can be seen to
be a particular case of the Caporaso-Harris formula. 

\medskip

In the special case $g=0$ (counting irreducible plane rational curves
maximally tangent to a given line), the relevant problem
of combinatorial enumeration can be solved completely.
As a result, we obtain a recurrence (see
Theorem~\ref{th:Nd0-recurrence} below) that can be used to calculate as 
many numbers $N_{d,0}((d),\emptyset)$ and $N_{d,0}(\emptyset,(d))$ as
one's computing resources allow. 
See Figure~\ref{fig:max-tangent}.


\begin{figure}[htbp]
\begin{center}
\begin{tabular}{r|r|r}
$d$ & $N_{d,0}((d),\emptyset)$ & $N_{d,0}(\emptyset,(d))$
\\
\hline
&\\[-.15in]
1 & 1 & 1\\
2 & 1 & 2\\
3 & 7 & 21\\
4 & 138 & 552\\
5 & 5477 & 27385\\
6 & 367640 & 2205840\\
7 & 37541883 & 262793181\\
8 & 5432772352 & 43462178816\\
9 & 1059075055273 & 9531675497457\\
10 & 267757626501504 & 2677576265015040\\
11 & 85244466165571535 & 937689127821286885\\
12 & 33379687015338236672 & 400556244184058840064\\
13 & 15770655073870516443597 & 205018515960316713766761\\
14 & 8847780392111931116474368 & 123868925489567035630641152\\
15 & 5815426547948880787678282627 & 87231398219233211815174239405 \\
16 & 4426738320076692932937846865920 & 70827813121227086927005549854720
\end{tabular}
\end{center}
\caption{\hbox{Number of irreducible rational curves maximally tangent
    to a line}} 
\label{fig:max-tangent}
\end{figure}

\begin{theorem}
\label{th:Nd0-recurrence}
The numbers $z(d)=N_{d,0}((d),\emptyset)$ satisfy the recurrence
relation
\begin{equation}
\label{eq:z(d)-recurrence}
z(d+1)=
\sum_{k=1}^d \frac{(2d)!}{k!} 
\,\,\,\sum_{\substack{a_1+\cdots+a_k=d \\ a_1,\dots,a_k>0}}
  \,\,\,\prod_{i=1}^k \frac{a_i^2 \,z(a_i)}{(2a_i)!}. 
\end{equation}
\end{theorem}

\begin{proof}
In order for a labeled floor diagram $\Dcal$ to allow for a marking containing
an edge of weight~$d$ (necessarily pointing from the last vertex
$v_\text{last}$ to the unique sink~$v_1$), 
two conditions must be satisfied at each vertex $v$ in~$\Dcal$:
\begin{itemize}
\item
there is exactly one outgoing edge emanating from~$v$;
\item
the inequality \eqref{eq:div} holds with an equality sign.
\end{itemize}
These conditions mean that the labeled floor diagrams~$\Dcal$ 
under consideration can be identified with \emph{increasing rooted trees} on
the vertex set $\{1,\dots,d\}$, i.e., with the labeled trees on $d$
vertices in which the labels increase along each simple path ending at
the root vertex~$d=v_\text{last}$. 
(To be literally precise, such trees are \emph{decreasing} in the
terminology of~\cite{ec1,ec2}, but this term would be misleading since we
orient the edges towards the root rather than away from it.) 
It is immediate from the definitions that
$\mu_{\emptyset}(\Dcal)=\mu(\Dcal)$ and 
$\nu_{(d),\emptyset}(\Dcal)=\nu(\Dcal)$,
so we have 
$z(d)=N_{d,0}((d),\emptyset)=\sum_\Dcal \mu(\Dcal)\,\nu(\Dcal)$,
the sum over all increasing trees $\Dcal$ on $d$ vertices. 
For example, referring to the first table in Appendix~A,
$z(3)=4\cdot 1+1\cdot 3=7$ and
$z(4)=36\cdot 1+9\cdot 3+4\cdot 3+ 4\cdot 7+4\cdot 5+1\cdot 15=138$. 

The \emph{hooklength} $h(v)$ of a vertex $v$ in~$\Dcal$ 
(cf.\ \cite[3.12.18]{sagan}) is, by definition, the number of vertices
$u$ which precede~$v$ in~$\Dcal$ (including $u=v$). 
It is easy to see from the divergence condition
that the hooklengths of the non-root vertices are
precisely the edge weights of~$\Dcal$.
Thus $\mu(\Dcal)=\prod_{v\neq v_\text{last}} (h(v))^2$. 

Consistent with the above, let a \emph{marked increasing tree}
on $2d$ vertices be an increasing (rooted) tree $\tDcal$ obtained
from an increasing tree $\Dcal$ on $d$ vertices by appending an extra
root vertex beyond the old one, introducing an
extra vertex at the middle of each edge of~$\Dcal$, 
and extending the linear ordering to the resulting tree. 
We then~have
\[
z(d)=\sum_{\tDcal} \prod_{v\neq v_\text{last}} (h(v))^2, 
\]
where the sum is over all marked increasing trees $\tDcal$ on $2d$
vertices, 
and the product is over all non-root vertices of the corresponding
tree~$\Dcal$. 

The recurrence~\eqref{eq:z(d)-recurrence} can now be obtained using
standard techniques of combinatorial enumeration
(cf., e.g., \cite[Chapter~5]{ec2} or \cite[Section~5.2]{bll}). 
A marked increasing tree $\tDcal$ on $2(d+1)$
vertices is uniquely decomposed, by cutting off the last two vertices 
(the old root and the new one) together with the edges incident to
them, into a \emph{shuffle} of some number~$k$ of marked rooted trees
$\tDcal_1,\dots,\tDcal_k$ on $2a_1,\dots,2a_k$ vertices, respectively,
where 
$\sum_i a_i=d$, and we numbered the subtrees arbitrarily by the integers
$1,\dots,k$.
(To compensate for this additional choice, 
we will need to divide by~$k!$ at the end.) 
The multiset of non-root hooklengths of $\tDcal$ is the disjoint union of the
multisets of hooklengths of $\tDcal_1,\dots,\tDcal_k$ 
(including the root hooklengths $a_1,\dots,a_k$). 
Hence $\mu(\tDcal)=\prod_i \mu(\tDcal_i) a_i^2$. 
Finally, for a given ordered $k$-tuple of marked trees
$\tDcal_1,\dots,\tDcal_k$, the number of possible shuffles is the
multinomial coefficient
\[
\binom{2d}{2a_1,\dots,2a_k}=\frac{(2d)!}{(2a_1)!\cdots(2a_k)!}. 
\]
Putting everything together, we obtain~\eqref{eq:z(d)-recurrence}. 
\end{proof}

\begin{theorem}
\label{th:Nd0-ODE}
The generating function
\begin{equation}
y(x)=\sum_{d=1}^\infty \frac{d^2 N_{d,0}((d),\emptyset)}{(2d)!}\, x^d
=\frac12 x+\frac16 x^2 + \frac{7}{80}x^3 + \frac{23}{420}x^4+\cdots 
\end{equation}
is the unique solution of the 
initial value problem 
\begin{equation}
\label{eq:Nd0-ODE}
x(4y'-e^y-x e^y y')=2y, \qquad y(0)=0. 
\end{equation}
\end{theorem}

\begin{proof}
Let us define
$\tilde z(d)=\frac{d^2\,z(d)}{(2d)!}$ and $\tilde z(0)=0$, 
so that we have $y(x)=\sum_{d=0}^\infty \tilde z(d) x^d$.
The recurrence~\eqref{eq:z(d)-recurrence} can be rewritten as
\[
\frac{2(2d+1)}{d+1}\,\tilde z(d+1)=
\sum_{k=0}^\infty \frac{1}{k!} 
\,\,\,\sum_{\substack{a_1+\cdots+a_k=d \\ a_1,\dots,a_k\ge 0}}
  \,\,\,\prod_{i=1}^k \tilde z(a_i) ,  
\]
which implies
\[
\sum_{d=0}^\infty \bigl(4-\frac{2}{d+1}\bigr)\,\tilde z(d+1)\,x^{d+1} 
=x\sum_{k=0}^\infty \frac{1}{k!} (y(x))^k 
=xe^y. 
\]
Differentiating, we get
\[
4y'-2\sum_{d=0}^\infty \tilde z(d+1)\,x^d
=4y'-2\frac{y}{x}
=e^y+xe^y y', 
\]
and \eqref{eq:Nd0-ODE} follows. 
\end{proof}

\subsection{Curves with prescribed tangency at a given point}
\begin{corollary} 
\label{cor:Ndg-2}
The 
number $N_{d,g}((k),\langle1^{d-k}\rangle)$ 
of irreducible plane complex algebraic curves of degree~$d$
and genus~$g$ passing through a generic configuration of 
$3d+g-k-1$ points and having tangency of order~$k$ 
to a given line at a given point 
is equal to $\sum_{\tDcal} \mu(\tDcal)$, 
the sum over marked floor diagrams $\tDcal$ of
degree~$d$ and genus~$g$ in which the last $k$ vertices are sinks
connected to the same vertex. 
\end{corollary}

\begin{proof}
Indeed, the (ordinary) marked floor diagrams of this kind are in
multiplicity-preserving bijection with the
$((k),\langle1^{d-k}\rangle)$-marked floor diagrams:
simply glue the edges pointing to the last $k$ sinks into a single
edge of weight~$k$. 
\end{proof}

\begin{example}
Let us compute (once again) the number $N_{3,0}((2),(1))$.
Among the (ordinary) marked floor diagrams of degree $d=3$ 
and genus $g=0$, there are exactly two in which the last two sinks are
\emph{not} connected to the same vertex;
they are shown in Figure~\ref{fig:mfd-not-N3021}. 
Each of the two has multiplicity~1, and we conclude that 
$N_{3,0}((2),(1))=N_{3,0}-2=12-2=10$, in agreement with
Figure~\ref{fig:d=3,g=0}. 
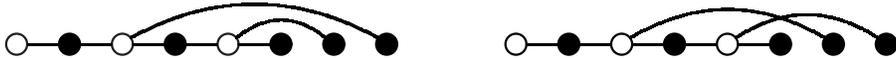
\begin{figure}[htbp]
\begin{center}
\begin{picture}(55,10)(110,3)
\setlength{\unitlength}{4pt}\thicklines
\multiput(0,0)(10,0){3}{\circle{2}}
\multiput(5,0)(10,0){4}{\circle*{2}}
\multiput(30,0)(10,0){1}{\circle*{2}}
\multiput(1,0)(5,0){5}{\line(1,0){3}}
\qbezier(20.8,0.6)(25,4)(29.2,0.6)
\qbezier(10.8,0.6)(22.5,7)(34.2,0.6)
\end{picture}
\begin{picture}(55,10)(-20,3)
\setlength{\unitlength}{4pt}\thicklines
\multiput(0,0)(10,0){3}{\circle{2}}
\multiput(5,0)(10,0){4}{\circle*{2}}
\multiput(30,0)(10,0){1}{\circle*{2}}
\multiput(1,0)(5,0){5}{\line(1,0){3}}
\qbezier(10.8,0.6)(20,6)(29.2,0.6)
\qbezier(20.8,0.6)(27.5,5)(34.2,0.6)
\end{picture}
\end{center}
\caption{Marked floor diagrams not contributing to $N_{3,0}((2),(1))$}
\label{fig:mfd-not-N3021}
\end{figure}
\end{example}

\section{Node polynomials}
\label{sec:node-polynomials}


This section is devoted to the classical problem of 
determining the Severi degrees~$N^{d,\delta}$ 
(see Section~\ref{sec:Ndg+Severi}) when the cogenus~$\delta$ is
fixed. 
In other words, how does~the number $N^{d,\delta}$ of $\delta$-nodal
(possibly reducible) plane curves depend on the degree~$d$? 
As already noted in~\eqref{eq:severi=Ndg}, if $d$ is large enough 
(specifically $d\ge\delta+2$), then all curves counted by
$N^{d,\delta}$ are irreducible, and the Severi degree coincides with the
corresponding Gromov-Witten invariant:
$N^{d,\delta}=N_{d,\frac{(d-1)(d-2)}{2}-\delta}\,$. 

Substantial efforts have been expended by various researchers 
to determine~$N^{d,\delta}$, as a function of~$d$, 
for specific small values of~$\delta$; 
see \cite[Remark~3.7]{kleiman-piene} for a thorough historical review.  
For $\delta\le 3$, the formulas go back to the 19th century
(J.~Steiner, A.~Cayley, G.~Salmon, and S.~Roberts; see the references
in~\cite{kleiman-piene}). 
In particular: 
\begin{align}
\label{eq:Nd,gmax} 
&N^{d,0}
=1,\\
\label{eq:Nd,gmax-1}
&N^{d,1}
=3(d-1)^2, 
\\
\label{eq:Nd,gmax-2}
&N^{d,2}
=\frac32(d-1)(d-2)(3d^2-3d-11). 
\end{align}
For $\delta\le 6$, the problem has been solved by
I.~Vainsencher~\cite{vainsencher};
see~\cite[Proposition~2]{difrancesco-itzykson} for explicit formulas. 
This has been extended to $\delta\le 8$ by S.~Kleiman and
R.~Piene~\cite{kleiman-piene} 
by further refining Vainsencher's method;
see Remark~\ref{rem:kleiman-piene} below. 
Other approaches to computing $N^{d,\delta}$ when $\delta$ is small were
developed by J.~Harris and
R.~Pandharipande~\cite{harris-pandharipande}, L.~G\"ottsche
\cite{goettsche}, 
and Y.~Choi~\cite{choi4, choi-97} (based on the work
of~Z.~Ran~\cite{ran89, ran98}). 

The following polynomiality property has been first suggested by 
P.~Di~Francesco and C.~Itzykson 
\cite[Remark~(b) after Proposition~2]{difrancesco-itzykson}, 
and then stated as a special case of a more general
conjecture by L.~G\"ottsche \cite[Conjecture~4.1 and
  Remark~4.2(2)]{goettsche}. 

\begin{theorem}
\label{th:goettsche}
For any fixed~$\delta$, there exist a polynomial 
$N_\delta(d)\in\mathbb{Q}[d]$ of degree~$2\delta$
and a threshold value $d_0(\delta)$ 
such that for $d\ge d_0(\delta)$, we have $N^{d,\delta}=N_\delta(d)$. 
\end{theorem}

Besides establishing polynomiality, 
our proof provides a method (admittedly tortuous)
for computing the polynomials~$N_\delta(d)$. 
Cf.\ Remark~\ref{rem:block-computing}. 

\begin{remark}
We prove Theorem~\ref{th:goettsche} with $d_0(\delta)=2\delta$. 
By further refining the argument, 
F.~Block~\cite{block-fpsac} has recently improved this to
$d_0(\delta)=\delta$. 

L.~G\"ottsche formulated his conjecture with 
$d_0(\delta)=\lceil\frac{\delta}{2}\rceil+1$.
This was verified by Block for all $\delta\le 13$
(cf.\ Remark~\ref{rem:block-computing}). 
P.~Di~Francesco and C.~Itzykson seem to suggest the 
threshold value 
$d_0(\delta)=\frac{3}{2}+\sqrt{2\delta+\frac14}$
(so that $d\ge d_0(\delta)$ is equivalent to
$\delta\le \frac{(d-1)(d-2)}{2}$). 
Block's computations show this to fail, for the first time, for $\delta=13$. 
\end{remark}

\begin{remark}
According to Y.~Choi (\cite[Section~3]{choi-97}, unpublished;
see citations in \cite[Remark~3.7]{kleiman-piene} and
\cite[Remark~4.2(2)]{goettsche}),
Theorem~\ref{th:goettsche}, with $d_0(\delta)=\delta$, 
can be deduced from \cite[Theorem~5]{ran89}.

In an unpublished preprint~\cite{liu} (cf.\ also~\cite{liu-jdg}),  
A.-K.~Liu put forward a proof of polynomiality  
of Severi degrees, in a more general setting of counting
curves on an arbitrary surface. 

Another proof of this result has been recently announced by
Y.-J.~Tzeng~\cite{tzeng}. 
\end{remark} 

\begin{remark}
\label{rem:kleiman-piene}
In the terminology of S.~Kleiman and R.~Piene~\cite{kleiman-piene}, 
$N_\delta(d)$ is called a \emph{node polynomial}.
The node polynomials $N_0(d)$, $N_1(d)$, and $N_2(d)$ are given by
\eqref{eq:Nd,gmax}, \eqref{eq:Nd,gmax-1}, and \eqref{eq:Nd,gmax-2},
respectively. 
The corresponding minimal threshold values $d_0(\delta)$ are all equal
to~1. 

Kleiman and Piene \cite[Section~3]{kleiman-piene} 
computed the node polynomials $N_\delta(d)$ for~$\delta\le 8$,
thereby establishing the corresponding instances of G\"ottsche's
conjecture (Theorem~\ref{th:goettsche}). 
Their computations can be summarized as follows: 

The generating function for the node polynomials $N_\delta(d)$ is given by
\begin{equation}
\label{eq:genf-node-poly}
\sum_{\delta\ge 0} N_\delta(d) \,t^\delta
=\exp\left(\sum_{j\ge 1} \frac{A_j(d)}{j}\,t^j
\right),
\end{equation}
where 
\begin{align*}
A_1(d) &= 3 (d - 1)^2,\\ 
A_2(d) &= -3 (d - 1) (14 d - 25),\\
A_3(d) &= 3(230 d^2-788d+633),\\
A_4(d) &= 9(-1340 d^2+5315 d-5023),\\
A_5(d) &= 9(24192d^2-107294d+114647),\\
A_6(d) &= 9(-445592d^2+2161292d-2545325),\\
A_7(d) &= 54(1386758d^2-7245004d+9242081),\\
A_8(d) &= 9(-156931220d^2+873420627d-1191950551).
\end{align*}
For $\delta\le8$,
the polynomials $N_\delta(d)$ computed via~\eqref{eq:genf-node-poly}
yield correct values of all nonvanishing Severi degrees~$N^{d,\delta}$,
and even some zero values, namely
$N^{1,1}=N^{1,2}=N^{2,2}=N^{3,4}=0$. 
\end{remark}

\begin{remark}
\label{rem:block-computing}
Our method for computing the node polynomials 
was implemented, with substantial algorithmic improvements, 
by F.~Block~\cite{block-fpsac},  
and used to compute $N_\delta(d)$ for all~$\delta\le 13$. 
Block's calculations confirmed that the polynomials $A_j(d)$ defined
by~\eqref{eq:genf-node-poly} are indeed quadratic in~$d$
(for $\delta\le 13$), in agreement with the strong form of G\"ottsche's
conjecture. 

\end{remark}

\begin{proof}[Proof of Theorem~\ref{th:goettsche}]
Our proof is purely combinatorial, and directly based on 
Theorem~\ref{th:bm-curves}. 

Let us call an edge~$e$ in a labeled floor diagram~$\Dcal$ \emph{short} 
if $e$ has weight~$1$, and connects consecutive vertices. 
If $\Dcal$ has small cogenus~$\delta$,
then ``almost all'' edges in~$\Dcal$ are short.
Removing those edges and considering the ``components'' of what remains,
we arrive at the following notion, which will play a key role in the
proof. 

\begin{definition}[\emph{Templates}]
\label{def:template}
A \emph{template} $\Gamma$ is a finite nonempty collection of weighted 
edges on a finite linearly ordered vertex set
$\{v_0<v_1<\cdots<v_\ell\}$ such that
\begin{itemize}
\item
for each edge $v_i\stackrel{e}{\longrightarrow} v_j$ in $\Gamma$,
we have $i<j$; 
\item
the weight $w(e)$ of every edge $e$ in $\Gamma$ is a positive integer; 
\item
the weight of an edge of the form $v_i\to v_{i+1}$ must be $\ge2$
(``no short edges''); 
\item
multiple edges are allowed, but loops are not;
\item
for every $j\in\{1,\dots,\ell-1\}$, there is at least one edge 
$v_i\stackrel{e}{\longrightarrow} v_k$ with $i<j<k$.
\end{itemize} 
With a template~$\Gamma$, we associate several quantities. 
The number $\ell=\ell(\Gamma)$ is called the \emph{length} of~$\Gamma$.
The product of squares of edge weights is the \emph{multiplicity}
of~$\Gamma$, denoted by~$\mu(\Gamma)$; cf.~\eqref{eq:mult-D}. 
The number
\begin{equation}
\label{eq:delta(Gamma)}
\delta(\Gamma)
=\sum_{v_i\,\stackrel{e}{\longrightarrow}\, v_j}
((j-i)\,w(e)-1) 
\end{equation}
is the \emph{cogenus} of~$\Gamma$. 
(The terminology will be justified by
Lemma~\ref{lem:cogenus-template}.) 
We set 
\[
\varepsilon(\Gamma)=\begin{cases}
1 & \text{if all edges arriving at $v_\ell$ have weight~1;}\\
0 & \text{otherwise.}
\end{cases}
\]
For $j\in\{1,\dots,\ell\}$, 
let $\varkappa_j=\varkappa_j(\Gamma)$ denote the total weight of all
edges $v_i\stackrel{e}{\longrightarrow} v_k$ with $i<j\le k$. 
By definition of a template, we have $\varkappa_j>0$. 
Let $\varkappa(\Gamma)=(\varkappa_1,\dots,\varkappa_\ell)$. 
Set
\begin{equation}
\label{eq:kmin}
k_{\min}(\Gamma)=\max_{1\le j\le\ell} (\varkappa_j-j+1). 
\end{equation}
Figure~\ref{fig:templates} shows all templates~$\Gamma$ with 
$\delta(\Gamma)=1$ or $\delta(\Gamma)=2$,
and the respective values of $\delta(\Gamma)$, $\ell(\Gamma)$, 
$\mu(\Gamma)$, $\varepsilon(\Gamma)$, $\varkappa(\Gamma)$, 
and~$k_{\min}(\Gamma)$. 
\end{definition}
\begin{figure}[htbp]
\begin{center}
\begin{tabular}{c|c|c|c|c|c|c|c}
$\Gamma$ &
$\delta(\Gamma)$ & $\ell(\Gamma)$ 
& $\mu(\Gamma)$ 
& $\varepsilon(\Gamma)$ & $\varkappa(\Gamma)$ 
& $k_{\min}(\Gamma)$ & $P(\Gamma,k)$
\\
\hline \hline
&&&&&&&\\[-.1in]
\begin{picture}(95,8)(-10,-4)\setlength{\unitlength}{2.5pt}\thicklines
\multiput(0,0)(10,0){2}{\circle{2}}
\put(5,2){\makebox(0,0){$\scriptstyle 2$}}
\Eee
\end{picture}
& 1 & 1 & 4 & 0 & (2) & 2 & $k-1$
\\[.15in]
\begin{picture}(95,8)(-10,-4)\setlength{\unitlength}{2.5pt}\thicklines
\ooo
\qbezier(0.8,0.6)(10,5)(19.2,0.6)
\end{picture}
& 1 & 2 & 1 & 1 & (1,1) & 1 & $2k+1$
\\
\hline \hline
&&&&&&&\\[-.1in]
\begin{picture}(95,8)(-10,-4)\setlength{\unitlength}{2.5pt}\thicklines
\multiput(0,0)(10,0){2}{\circle{2}}
\put(5,2){\makebox(0,0){$\scriptstyle 3$}}
\Eee
\end{picture}
& 2 & 1 & 9 & 0 & (3) & 3 & $k-2$
\\[.15in]
\begin{picture}(95,8)(-10,-4)\setlength{\unitlength}{2.5pt}\thicklines
\multiput(0,0)(10,0){2}{\circle{2}}
\put(5,3.5){\makebox(0,0){$\scriptstyle 2$}}
\put(5,-3.5){\makebox(0,0){$\scriptstyle 2$}}
\qbezier(0.8,0.6)(5,2)(9.2,0.6)
\qbezier(0.8,-0.6)(5,-2)(9.2,-0.6)
\end{picture}
& 2 & 1 & 16 & 0 & (4) & 4 & $\binom{k-2}{2}$
\\[.15in]
\begin{picture}(95,8)(-10,-4)\setlength{\unitlength}{2.5pt}\thicklines
\ooo
\qbezier(0.8,0.6)(10,4)(19.2,0.6)
\qbezier(0.8,-0.6)(10,-4)(19.2,-0.6)
\end{picture}
& 2 & 2 & 1 & 1 & (2,2) & 2 & $\binom{2k}{2}$
\\[.15in]
\begin{picture}(95,8)(-10,-4)\setlength{\unitlength}{2.5pt}\thicklines
\ooo
\qbezier(0.8,0.6)(10,4)(19.2,0.6)
\put(5,-2){\makebox(0,0){$\scriptstyle 2$}}
\Eee
\end{picture}
& 2 & 2 & 4 & 1 & (3,1) & 3 & $2k(k-2)$
\\[.15in]
\begin{picture}(95,8)(-10,-4)\setlength{\unitlength}{2.5pt}\thicklines
\ooo
\qbezier(0.8,0.6)(10,4)(19.2,0.6)
\put(15,-2){\makebox(0,0){$\scriptstyle 2$}}
\eEe
\end{picture}
& 2 & 2 & 4 & 0 & (1,3) & 2 & $2k(k-1)$
\\[.15in]
\begin{picture}(95,8)(-10,-4)\setlength{\unitlength}{2.5pt}\thicklines
\oooo
\qbezier(0.8,0.6)(15,6)(29.2,0.6)
\end{picture}
& 2 & 3 & 1 & 1 & (1,1,1) & 1 & $3(k+1)$
\\[.15in]
\begin{picture}(95,8)(-10,-4)\setlength{\unitlength}{2.5pt}\thicklines
\oooo
\qbezier(0.8,0.6)(10,5)(19.2,0.6)
\qbezier(10.8,0.6)(20,5)(29.2,0.6)
\end{picture}
& 2 & 3 & 1 & 1 & (1,2,1) & 1 & $k(4k+5)$
\\[-.05in]
\end{tabular}
\end{center}
\caption{Templates with $\delta(\Gamma)\le 2$}
\label{fig:templates}
\end{figure}

For the remainder of this proof, we allow disconnected labeled floor
diagrams~$\Dcal$. The degree~$d$ and cogenus~$\delta$ 
of such a diagram are determined 
from the degrees $d_j$ and cogenera~$\delta_j$ of its connected
components by the
formulas~\eqref{eq:total-degree}--\eqref{eq:total-cogenus}.
As we noted earlier (see the comment following
Corollary~\ref{cor:severi-combin}), 
the Severi degree~$N^{d,\delta}$ is obtained by 
counting the markings of all such diagrams 
(with the given~$d$ and~$\delta$)
with the usual multiplicities~$\mu(\Dcal)$. 

Let $\Dcal$ be a (possibly disconnected)
labeled floor diagram of degree~$d$ and cogenus~$\delta$. 
It will be convenient to add an extra vertex $d+1$ to the vertices
$1,\dots,d$ of~$\Dcal$, and connect each vertex $v$ in~$\Dcal$ to this
new vertex by $1-\divv(v)$ new edges of weight~1.
Let $\Dcal'$ denote the resulting diagram. 
To illustrate, applying this procedure to the labeled floor
diagram~$\Dcal$ shown in~\eqref{eq:fd4} results in the
diagram~$\Dcal'$ drawn below: 
\begin{equation}
\label{eq:fd4-extended}
\begin{picture}(80,25)(30,-4)\setlength{\unitlength}{4pt}\thicklines
\oooo
\put(40,0){\circle{2}}
\Eee\eOe\eeE
\qbezier(30.8,0.6)(35,4)(39.2,0.6)
                  \qbezier(30.8,-0.6)(35,-4)(39.2,-0.6)
\put(31,0){\line(1,0){8}}
\qbezier(20.6,0.6)(30,9)(39.4,0.8)
\put(25,1.5){\makebox(0,0){$2$}}
\end{picture}
\end{equation}

Upon removal of all short edges from~$\Dcal'$, one obtains a (uniquely
defined) collection of non-overlapping templates. 
To be pedantic, let $\Gamma_1,\dots,\Gamma_m$ be these templates,
listed left to right. 
Denoting the leftmost vertex of $\Gamma_i$ by~$k_i$
(we call $k_i$ the \emph{offset} of~$\Gamma_i$), we 
then have 
\begin{equation}
\label{eq:ki+l}
k_i+\ell(\Gamma_i)\le k_{i+1} \ \ \text{for $1\le i\le m-1$}
\end{equation}
(so that the $\Gamma_i$ do not overlap), and also
\begin{equation}
\label{eq:km+l}
k_m+\ell(\Gamma_m)\le d+\varepsilon(\Gamma_m)
\end{equation}
(so that $\Gamma_m$ properly fits at the right end). 

For $q\le d+1$, let $a_q$ denote the number of short edges connecting 
$q-1$ to~$q$ in~$\Dcal'$. 
The divergence condition implies that in~$\Dcal'$,
the total weight of the edges $p\stackrel{e}{\longrightarrow} r$ 
with $p<q\le r$ is precisely $q-1$. 
Then $q-1-a_q$ is the total weight of such edges contained in 
(one of) the templates $\Gamma_1,\dots,\Gamma_m$. 
If say $q=k_i+j$ is located in~$\Gamma_i$ 
(so~that $q$ corresponds to~$v_j$ in the notation of
Definition~\ref{def:template}), 
then we have $\varkappa_j=q-1-a_q=k_i+j-1-a_q\le k_i+j-1$,
implying (cf.~\eqref{eq:kmin}) that 
\begin{equation}
\label{eq:ki>kmin}
k_i\ge k_{\min}(\Gamma_i) \ \ \text{for $1\le i\le m$.}
\end{equation}

Conversely, given a sequence of isomorphism types of 
templates $\Gamma_1,\dots,\Gamma_m$
and an increasing sequence of positive integers 
$k_1<\cdots<k_m$ satisfying the inequalities 
\eqref{eq:ki+l}--\eqref{eq:ki>kmin},
there is a unique (possibly disconnected) labeled floor
diagram~$\Dcal$ whose modification~$\Dcal'$ is obtained by placing
each $\Gamma_i$ with an offset~$k_i$,
and adding short edges as needed. 
Specifically, we add $a_q$ edges between $q-1$ and~$q$, with $a_q$ 
given by
\begin{equation}
\label{eq:aq}
a_q=\begin{cases}
q-1-\varkappa_j & \text{if $q=k_i+j$ for $1\le i\le m$ 
  and $1\le j\le\ell(\Gamma_i)$;}\\
q-1 & \text{otherwise.}
\end{cases}
\end{equation}

\begin{lemma}
\label{lem:cogenus-template}
The cogenus $\delta=\delta(\Dcal)$ 
(as defined by~\eqref{eq:total-cogenus}) is equal to 
\begin{equation}
\label{eq:delta=sum=of-deltas}
\delta=\sum_{i=1}^m \delta(\Gamma_i), 
\end{equation}
where the numbers $\delta(\Gamma_i)$ are defined
by~\eqref{eq:delta(Gamma)}. 
\end{lemma}

\begin{proof}
Note that removing a template~$\Gamma_k$ from the list
would result in replacing each of its edges
$v_i\,\stackrel{e}{\longrightarrow}\, v_j$ of weight~$w(e)$ 
by the appropriate collection of short edges 
\begin{equation*}
\begin{picture}(50,16)(25,-6)
\setlength{\unitlength}{3pt}
\thicklines
\multiput(0,0)(10,0){5}{\circle{2}}
\qbezier(0.8,0.6)(5,4)(9.2,0.6)
\qbezier(0.8,-0.6)(5,-4)(9.2,-0.6)
\qbezier(0.95,0.3)(5,1.5)(9.05,0.3)
\qbezier(0.95,-0.3)(5,-1.5)(9.05,-0.3)

\qbezier(10.8,0.6)(15,4)(19.2,0.6)
\qbezier(10.8,-0.6)(15,-4)(19.2,-0.6)
\qbezier(10.95,0.3)(15,1.5)(19.05,0.3)
\qbezier(10.95,-0.3)(15,-1.5)(19.05,-0.3)

\qbezier(20.8,0.6)(25,4)(29.2,0.6)
\qbezier(20.8,-0.6)(25,-4)(29.2,-0.6)
\qbezier(20.95,0.3)(25,1.5)(29.05,0.3)
\qbezier(20.95,-0.3)(25,-1.5)(29.05,-0.3)

\qbezier(30.8,0.6)(35,4)(39.2,0.6)
\qbezier(30.8,-0.6)(35,-4)(39.2,-0.6)
\qbezier(30.95,0.3)(35,1.5)(39.05,0.3)
\qbezier(30.95,-0.3)(35,-1.5)(39.05,-0.3)

\put(0,-3){\makebox(0,0){$v_i$}}
\put(41,-3){\makebox(0,0){$v_j$}}

\end{picture}
\end{equation*}
in which each pair of consecutive vertices is connected by $w(e)$
edges. Such a replacement increases the total number of edges by 
$(j-i)\,w(e)-1$, thus decreasing the cogenus by the same amount.
Since removing all templates would yield a diagram of cogenus~$0$, the
claim follows.
\end{proof}

In order to write down the formula for~$N^{d,\delta}$, 
we will need to keep track of the markings of~$\Dcal$ (or of~$\Dcal'$);
these break down into markings of individual templates and
associated short edges. 
For a template~$\Gamma$ and an offset value~$k\in\ZZ_{>0}$, 
let $\Gamma_{(k)}$ denote the poset
obtained from $\Gamma$ by first adding $k+j-1-\varkappa_j$ short edges 
connecting $j-1$ to~$j$ (for $1\le j\le\ell(\Gamma)$;
cf.~\eqref{eq:aq}),
then inserting an extra vertex in the middle of each edge of the
resulting graph. 
Let $P(\Gamma,k)$ denote the number of linear extensions
of~$\Gamma_{(k)}$, considered modulo automorphisms of~$\Gamma_{(k)}$
which fix~$\Gamma$ (as in Definition~\ref{def-marking}). 
With this definition, the number $\nu(\Dcal)$ of markings of a labeled
floor diagram~$\Dcal$ is given by
\[
\prod_{i=1}^m P(\Gamma_i,k_i), 
\]
where the templates $\Gamma_i$ and their respective offsets $k_i$ are
the same as before. 

\begin{lemma}
For a fixed template $\Gamma$, 
the values $P(\Gamma,k)$, for $k\ge k_{\min}(\Gamma)$, 
are given by a polynomial in~$k$ whose degree is equal to 
the number of edges in~$\Gamma$. 
\end{lemma}

\begin{proof}
First linearly order the vertices of $\Gamma$ together
with the midpoints of its edges.
For each of these (finitely many) choices, 
we need to count the number of ways of completing it to a linear
extension of~$\Gamma_{(k)}$ (here everything is done modulo the appropriate
automorphism group). 
Such a completion amounts to choosing, for each $j\le\ell(\Gamma)$, 
a particular shuffle of the $k+j-1-\varkappa_j$ (unordered) midpoints
of short edges connecting $j-1$ to~$j$ with a fixed (i.e., independent
of~$k$) number $b_j$ of midpoints of edges of~$\Gamma$. 
Hence the answer is
\[
\prod_j \binom{k+j-1-\varkappa_j+b_j}{b_j}\,, 
\]
and the claim follows. 
\end{proof}

Figure~\ref{fig:templates} shows 
the polynomials $P(\Gamma,k)$ for the templates~$\Gamma$ with
$\delta(\Gamma)\le 2$. 

\medskip

Putting all the ingredients together, 
we see that the Severi degree $N^{d,\delta}$ is given by
\begin{equation}
\label{eq:Nd-delta=sumsum}
N^{d,\delta}
=\sum_{m=1}^\delta\ \sum_{\Gamma_1,\dots,\Gamma_m}
\Biggl(\prod_{i=1}^m \ \mu(\Gamma_i)\Biggr) 
\sum_{k_1,\dots,k_m} \ \prod_{i=1}^m \ P(\Gamma_i,k_i), 
\end{equation}
where the second sum is over all $m$-tuples of templates
$(\Gamma_1,\dots,\Gamma_m)$ satisfying~\eqref{eq:delta=sum=of-deltas},
and the third sum is over $m$-tuples of integer offsets
$(k_1,\dots,k_m)$ satisfying \eqref{eq:ki+l}--\eqref{eq:ki>kmin}. 
(This calculation is illustrated for $\delta=2$ in
Example~\ref{example:delta=2} below.) 
Now let us write the aforementioned third sum as 
\[
\sum_{\substack{k_m\ge k_{\min}(\Gamma_m)\\[.05in] 
    k_m\le d-\ell(\Gamma_m)+\varepsilon(\Gamma_m)}}P(\Gamma_m,k_m)
\cdots
\sum_{\substack{k_2\ge k_{\min}(\Gamma_2)\\[.05in] 
    k_2\le k_3-\ell(\Gamma_2)}}P(\Gamma_2,k_2)
\ \sum_{\substack{k_1\ge k_{\min}(\Gamma_1)\\[.05in] 
    k_1\le k_2-\ell(\Gamma_1)}}P(\Gamma_1,k_1).  
\]
If $P(k)$ is given by a polynomial in~$k$ 
for $k\ge c$, and $a$ and~$b$ are positive integer constants,
then $\displaystyle\sum_{\substack{k\ge a\\ k\le n-b}} P(k)$ 
is given by a polynomial in~$n$ (of one degree higher)
for $n\ge \max(a+b,c)$. 
Iterating this argument, we conclude that $N^{d,\delta}$ is given by a
polynomial in~$d$ if 
\begin{equation}
\label{eq:d-lower-bound}
d\ge
k_{\min}(\Gamma_1)+\ell(\Gamma_1)+\cdots+\ell(\Gamma_m)
-\varepsilon(\Gamma_m) 
\end{equation}
for any allowable choice of $\Gamma_1,\dots,\Gamma_m$. 
It is not hard to see that the degree of the resulting polynomial 
is~$2\delta$; indeed, the maximal value of~$m$ is~$\delta$, and the maximal
total number of edges in the templates involved is $\delta$ as well. 
Also, the right-hand side of~\eqref{eq:d-lower-bound} can be seen to
be at most~$2\delta$, providing a threshold value. 
\end{proof}

\begin{example}
\label{example:delta=2}
For $\delta=2$, substituting the data from Figure~\ref{fig:templates} 
into formula~\eqref{eq:Nd-delta=sumsum}, we obtain:
\begin{align*}
N^{d,2}&=\sum_{k=2}^{d-1}
\biggl(9(k-2)+16\textstyle\binom{k-2}{2}+\binom{2k}{2}+8k(k-2)\biggr)
\\
&+\sum_{k=1}^{d-2}\bigl(8k(k-1)+3(k+1)+k(4k+5)\bigr)
\\
&+\sum_{k_2\le d-1} \sum_{k_1\le k_2-1} 
  \bigl(16(k_1-1)(k_2-1) +4(k_1-1)(2k_2+1)\bigr)
\\
&+\sum_{k_2\le d-1} \sum_{k_1\le k_2-2}
   \bigl((2k_1+1)\cdot 4(k_2-1) + (2k_1+1)(2k_2+1)\bigr),
\end{align*}
which after tedious calculations (or with a help of your favourite
software) yields~\eqref{eq:Nd,gmax-2}. 
\end{example}

\section{Enumeration of labeled floor diagrams of genus~$0$}
\label{sec:enumeration}

\begin{theorem}
\label{th:cayley}
The number of labeled floor diagrams of degree~$d$ and genus~$0$
is~$d^{d-2}$.
\end{theorem}

Recall that $d^{d-2}$ is also the number of labeled trees on~$d$
vertices, or equivalently the number of trees on the vertex set
$\{1,\dots,d\}$.
This classical result is commonly known as Cayley's formula
(see for example \cite[Theorem~2.1]{van-lint-wilson} and
\cite[pages~25 and~66]{ec2}). 

\begin{proof}[Proof of Theorem~\ref{th:cayley}]
Let $\ell(d)$  denote the number of labeled floor diagrams of
  degree~$d$ and genus~$0$.
Let $t(d)=d^{d-2}$ denote the number of labeled trees on $d$ vertices.
Our goal is to show that $\ell(d)=t(d)$.

It is well known, and easy to deduce,
that
\begin{equation}
\label{eq:stirling}
t(d)=\sum_k \sum_{\ S_1\cup\cdots\cup S_k=\{1,\dots,d-1\}\ }
  \prod_{i=1}^k t(|S_i|)|S_i|,
\end{equation}
where the second sum is over all \emph{unordered} set partitions of
the set $\{1,\dots,d-1\}$ into $k$ nonempty blocks $S_1,\dots,S_k$.
(Thus the number of summands is $S(d-1,k)$,
the Stirling number of the second kind \cite[Section~1.4]{ec1}.)
Formula~\eqref{eq:stirling} is an enumerative encoding of the
decomposition of a labeled tree into a root vertex~$d$ and a
collection of disjoint subtrees attached to~$d$, whose vertex sets
form a set partition of $\{1,\dots,d-1\}$.

Together with the initial condition $t(1)=1$, the
recurrence~\eqref{eq:stirling} uniquely determines the sequence~$t(d)$.
Hence the theorem will follow if we show that the numbers~$\ell(d)$
satisfy the same recurrence.
This can be established in the manner analogous to the proof
of~\eqref{eq:stirling} given above.
Every labeled floor diagram $\Dcal$ of genus~$0$ and degree~$d$
is a particular kind of a weighted labeled tree on the vertex set
$\{1,\dots,d\}$.
As such, it can be disassembled into the root vertex~$d$ and an
unordered collection of genus-$0$ floor diagrams $\Dcal_1,\dots,\Dcal_k$
supported on disjoint nonempty vertex sets $S_1,\dots,S_k$
and joined with~$d$ by weighted edges. There is exactly one such
weighted edge connecting $d$ to each set~$S_i$.
It remains to check that for a given labeled floor diagram~$\Dcal_i$
on a vertex set $S_i\subset \{1,\dots,d-1\}$,
there are exactly $|S_i|$ ways to connect
a vertex $v\in S_i$ to~$d$ by a weighted edge
$v\stackrel{e}{\longrightarrow}d$ so that the resulting
weighted tree on $|S_i|+1$ vertices is a labeled floor diagram.
(Gluing those diagrams together will produce a labeled floor diagram
on $\{1,\dots,d\}$.)
To prove this, note that for a fixed~$v$,
the only restriction on the weight $w(e)$ comes from the divergence
condition~\eqref{eq:div}.
The latter implies that $w(e)$ can be chosen to be any positive
integer not exceeding $1-\divv_i(v)$, where $\divv_i(v)$ denotes the
divergence at~$v$ within the subdiagram~$\Dcal_i$.
It follows that the number of choices in question is
\begin{equation}
\label{eq:1-div}
\sum_{v\in S_i} (1-\divv_i(v))=|S_i|,
\end{equation}
as desired.
\end{proof}

The proof of Theorem~\ref{th:cayley} given above
can be used to construct explicit bijections between
labeled floor diagrams of genus~$0$ and degree~$d$, on one hand,
and labeled trees on $d$ vertices, on another.
One family of bijections of this kind,
illustrated in Appendix~A, 
is built recursively as follows.

\begin{definition}[\emph{A bijection between labeled floor diagrams and
   labeled trees}]
\label{def:diag-to-tree}
Suppose we have already defined such bijections for labeled floor
diagrams of genus~$0$ and any degree~$<d$.
Now, for a diagram $\Dcal$ of degree~$d$, do the following.
Decompose $\Dcal$ into the root vertex~$d$ and
subdiagrams~$\Dcal_1,\dots,\Dcal_k$ on vertex sets $S_1,\dots,S_k$, as
described in the proof of Theorem~\ref{th:cayley} above.
Let $T_1,\dots,T_k$ denote the trees on vertex sets $S_1,\dots,S_k$
that correspond to $\Dcal_1,\dots,\Dcal_k$, respectively,
under the appropriate bijections.
The tree~$T$ associated to~$\Dcal$ under the bijection in question is
constructed by connecting each tree $T_i$ to the root vertex~$d$ by a single
edge $e'$ that is going to be determined by the unique weighted edge
$v\stackrel{e}{\longrightarrow}d$
connecting $\Dcal_i$ to~$d$ in~$\Dcal$.
It remains to describe the rule that determines~$e'$ from~$e$.
We have already checked (see~\eqref{eq:1-div}) that the total number
of choices for~$e$ is equal to~$|S_i|$.
Let us record those choices in an ordered list as the vertex~$v$ moves left
to right within~$\Dcal_i$; for a given~$v$, we record the choices
starting with the \emph{largest} possible weight value $w(e)=1-\divv_i(v)$ and
decreasing it until we reach the smallest possible weight value
$w(e)=1$.
We similarly list the $|S_i|$ choices available for the
edge $v'\stackrel{e'}{\longrightarrow}d$ connecting $S_i$ to~$d$ ;
the ordering is determined by the (left-to-right) ordering of the
vertices $v'\in S_i$.
We finally match the choices on both lists in the order they are
listed.
\end{definition}

\begin{example}
To illustrate, consider the labeled floor diagram~$\Dcal_i$ shown
below alongside with the tree~$T_i$ associated to it:
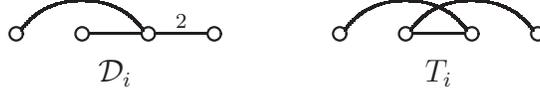
\begin{center}
\begin{picture}(95,38)(-10,-20)\setlength{\unitlength}{2.5pt}\thicklines
\oooo\EEe\eEe\eeE
\put(25,2){\makebox(0,0){$\scriptstyle 2$}}
\put(15,-6){\makebox(0,0){$\Dcal_i$}}
\end{picture}
\qquad
\begin{picture}(95,38)(-10,-20)\setlength{\unitlength}{2.5pt}\thicklines
\oooo\EEe\eEe\eEE
\put(15,-6){\makebox(0,0){$T_i$}}
\end{picture}
\end{center}
There are $4$~ways to augment~$\Dcal_i$ by a single weighted edge
$v\stackrel{e}{\longrightarrow}d$
(here $d$~is a vertex located to the right of~$\Dcal_i$)
so that the resulting tree on~$5$ vertices is a valid  labeled floor
diagram.
These $4$ possibilities are shown in the first column of
Figure~\ref{fig:biject}.
Similarly, there are $4$ possibilities, shown in the second column,
to connect $T_i$ to such a vertex~$d$ by a single edge.
The bijection matches each labeled floor diagram to the labeled tree
shown in the same row of the table.
\end{example}

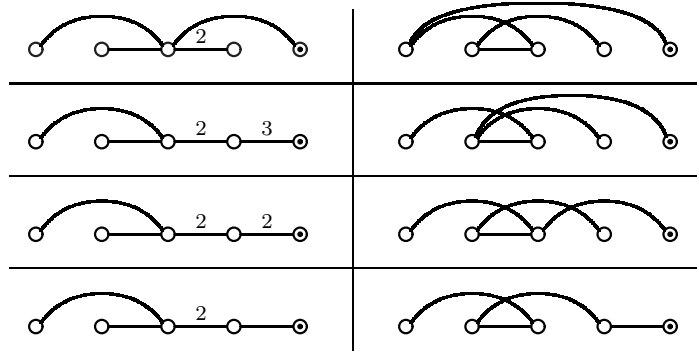
\begin{figure}[htbp]
\begin{center}
\begin{tabular}{c|c}
\begin{picture}(120,20)(-5,-5)\setlength{\unitlength}{2.5pt}\thicklines
\oooo\put(40,0){\circle{2}}\put(40,0){\circle*{0.5}}
\EEe\eEe\eeE
\put(25,2){\makebox(0,0){$\scriptstyle 2$}}
\qbezier(20.8,0.6)(24,5)(30,5)\qbezier(30,5)(36,5)(39.2,0.6)
\end{picture}
&
\begin{picture}(120,20)(-15,-5)\setlength{\unitlength}{2.5pt}\thicklines
\oooo\put(40,0){\circle{2}}\put(40,0){\circle*{0.5}}
\EEe\eEe\eEE
\qbezier(0.6,0.8)(3,7)(20,7)\qbezier(20,7)(37,7)(39.4,0.8)
\end{picture}
\\[.05in]
\hline
&\\[-.1in]
\begin{picture}(120,20)(-5,-5)\setlength{\unitlength}{2.5pt}\thicklines
\oooo\put(40,0){\circle{2}}\put(40,0){\circle*{0.5}}
\EEe\eEe\eeE
\put(25,2){\makebox(0,0){$\scriptstyle 2$}}
\put(31,0){\line(1,0){8}}
\put(35,2){\makebox(0,0){$\scriptstyle 3$}}
\end{picture}
&
\begin{picture}(120,20)(-15,-5)\setlength{\unitlength}{2.5pt}\thicklines
\oooo\put(40,0){\circle{2}}\put(40,0){\circle*{0.5}}
\EEe\eEe\eEE
\qbezier(10.6,0.8)(13,7)(25,7)\qbezier(25,7)(37,7)(39.4,0.8)
\end{picture}
\\[.05in]
\hline
&\\[-.1in]
\begin{picture}(120,20)(-5,-5)\setlength{\unitlength}{2.5pt}\thicklines
\oooo\put(40,0){\circle{2}}\put(40,0){\circle*{0.5}}
\EEe\eEe\eeE
\put(25,2){\makebox(0,0){$\scriptstyle 2$}}
\put(31,0){\line(1,0){8}}
\put(35,2){\makebox(0,0){$\scriptstyle 2$}}
\end{picture}
&
\begin{picture}(120,20)(-15,-5)\setlength{\unitlength}{2.5pt}\thicklines
\oooo\put(40,0){\circle{2}}\put(40,0){\circle*{0.5}}
\EEe\eEe\eEE
\qbezier(20.8,0.6)(24,5)(30,5)\qbezier(30,5)(36,5)(39.2,0.6)
\end{picture}
\\[.05in]
\hline
&\\[-.1in]
\begin{picture}(120,20)(-5,-5)\setlength{\unitlength}{2.5pt}\thicklines
\oooo\put(40,0){\circle{2}}\put(40,0){\circle*{0.5}}
\EEe\eEe\eeE
\put(25,2){\makebox(0,0){$\scriptstyle 2$}}
\put(31,0){\line(1,0){8}}
\end{picture}
&
\begin{picture}(120,20)(-15,-5)\setlength{\unitlength}{2.5pt}\thicklines
\oooo\put(40,0){\circle{2}}\put(40,0){\circle*{0.5}}
\EEe\eEe\eEE
\put(31,0){\line(1,0){8}}
\end{picture}

\end{tabular}
\end{center}
\caption{Building a bijection between labeled floor diagrams and trees}
\label{fig:biject}
\end{figure}

\begin{lemma}
\label{lem:short-edges}
Let $\Dcal$  be a labeled floor diagram of genus~$0$ and degree~$d$,
and let $T$ be the tree on the vertex set $\{1,\dots,d\}$
associated to it by the bijection
described in Definition~\ref{def:diag-to-tree}.
Then, for any $i\in\{1,\dots,d-1\}$, the following are equivalent:
\begin{itemize}
\item
$\Dcal$ contains an edge $i\to i+1$, and this edge has weight~$1$;
\item
$T$ contains the edge $(i,i+1)$.
\end{itemize}
\end{lemma}

\begin{proof}
This lemma is a consequence of the following observation:
in the ordered lists of choices involved in a
recursive step of the bijection described in
Definition~\ref{def:diag-to-tree},
the \emph{last} elements are:
\begin{itemize}
\item
the edge of weight~$1$ connecting the rightmost vertex in $\Dcal_i$ to~$d$;
\item
the edge connecting the rightmost vertex in $T_i$ to~$d$.
\end{itemize}
Cf.\ for example the last row in Figure~\ref{fig:biject}.
\end{proof}

Lemma~\ref{lem:short-edges} directly implies the following enumerative
corollary.

\begin{corollary}
\label{cor:short-edges}
For any subset $A\subset\{1,\dots,d-1\}$, the following are equal:
\begin{itemize}
\item
the number of labeled floor diagrams of genus~$0$ and degree~$d$
in which all the edges of the form $a\rightarrow a+1$,
for $a\in A$, are present, each with weight~1;
\item
the number of trees on the vertex set $\{1,\dots,d\}$
containing all the edges of the form $(a,a+1)$,
for $a\in A$.
\end{itemize}
\end{corollary}

As an application of Corollary~\ref{cor:short-edges}, we obtain:

\begin{corollary}
For a given $a\in\{1,\dots,d-1\}$,
the number of labeled floor diagrams of genus~$0$ and degree~$d$
containing an edge $a\rightarrow a+1$, with weight~1, is
equal~to~$2d^{d-3}$. In particular, this number does not depend
on~$a$.

More generally, for $1\le a<a+b\le d$,
the number of labeled floor diagrams of genus~$0$ and degree~$d$ which
contain the edges $a\rightarrow a+1\rightarrow\cdots\rightarrow a+b$,
all with weight~1, is equal to~$(b+1)d^{d-b-2}$.
\end{corollary}

\begin{proof}
By Corollary~\ref{cor:short-edges}, the quantity in question is equal
to the number of trees on the vertex set $\{1,\dots,d\}$ containing
all the edges $(a,a+1)$, $(a+1,a+2)$, \dots, $(a+b-1,a+b)$.
Equivalently, this is the number of spanning trees in the graph
obtained from the $d$-vertex complete graph~$K_d$ by contracting all edges
connecting some $b+1$ vertices to each other.
Computing the number of such spanning trees is a straightforward
application of the Matrix-Tree Theorem (see, e.g.,
\cite[Theorem~36.1]{van-lint-wilson} or \cite[Theorem~5.6.8]{ec2}),
which is left to the reader.
\end{proof}


Another curious enumerative result concerning labeled floor diagrams
is the following byproduct of our proof of
Theorem~\ref{th:Nd0-recurrence}. 

\begin{proposition}
The number of labeled floor diagrams of genus~$0$ and degree~$d$
containing an edge $e$ of weight $w(e)=d-1$ is equal to $(d-2)!$.
\end{proposition}

\begin{proof}
These labeled floor diagrams are in bijection with increasing rooted
trees on $d-1$ vertices 
(see the proof of Theorem~\ref{th:Nd0-recurrence}).
The number of such trees is well known to be equal to~$(d-2)!$;
see, e.g., \cite[Proposition~1.3.16]{ec1}.
\end{proof}

It is easy to see from Definition~\ref{def:lfd} that $d-1$ is the
largest possible edge weight in a floor diagram
of degree~$d$.

\section{Conjectures and open problems}
\label{sec:conjectures}

\subsection{Higher genera}
One would obviously like to extend Theorem~\ref{th:cayley}
beyond the case $g=0$.

\begin{problem}
Enumerate labeled floor diagrams of degree~$d$ and genus~$g>0$.
\end{problem}

At this moment, we do not have a conjectural formula for the number
$l_{d,g}$ of such diagrams.
Using the data in Appendix~A, 
one concludes that
\[
l_{3,1}=1,\ \
l_{4,1}=13,\ \
l_{4,2}=5,\ \
l_{4,3}=1.
\]

\medskip

The rest of this section is devoted exclusively to the case $g=0$.

\subsection{Tree statistics}
Let $\varphi_d$ be a bijection that maps a
labeled tree~$T$ on $d$ vertices to a labeled floor diagram
$\Dcal=\varphi_d(T)$ of genus~$0$ and degree~$d$;
an example of such a bijection was given in
Definition~\ref{def:diag-to-tree}.
We can then lift the functions $\Dcal\mapsto\mu(\Dcal)$ and
$\Dcal\mapsto\nu(\Dcal)$ to the corresponding \emph{tree statistics}
\begin{align*}
T&\mapsto \hat\mu(T)=\mu(\varphi_d(T)), \\
T&\mapsto \hat\nu(T)=\nu(\varphi_d(T))
\end{align*}
which are obviously equidistributed with $(\mu,\nu)$.
Then Theorem~\ref{th:bm-curves} yields a formula for
the Gromov-Witten numbers~$N_{d,0}$ in terms of $\hat\mu$ and~$\hat\nu$:
\begin{equation}
\label{eq:N=sum-hats}
N_{d,0}=\sum_T \hat\mu(T)\,\hat\nu(T),
\end{equation}
the sum over all labeled trees~$T$ on $d$ vertices.

The tree statistics $\hat\mu$ and~$\hat\nu$ derived
from the bijection of Definition~\ref{def:diag-to-tree} turn out to be
quite complicated, so the resulting formula~\eqref{eq:N=sum-hats} is
not as elegant as one might desire.
This naturally leads to the following problem.

\begin{problem}
Find tree statistics $\hat\mu$ and~$\hat\nu$, as conceptually simple as
possible, whose joint distribution on the set of labeled trees on
$d$ vertices coincides with the joint distribution of $\mu$ and~$\nu$
on the set of labeled floor diagrams of degree~$d$ and genus~$0$.
\end{problem}

\subsection{Recurrences of Kontsevich and Caporaso-Harris}
\label{sec:kontsevich-caporaso-harris}

As mentioned in \cite[Exercice~6.2]{brugalle}
(cf.\ also \cite{arroyo-brugalle-lopez}), it is possible to use
Theorem~\ref{th:combin-rule-two-ptns} to provide a combinatorial
derivation of the Caporaso-Harris recurrence, 
somewhat similar in spirit to the proof given by
A.~Gathmann and H.~Markwig~\cite{gathmann-markwig}.

The celebrated formula of Kontsevich \cite[(5.17)]{kontsevich-manin}
determines the genus-$0$ Gromov-Witten invariants~$N_{d,0}$ by means of the
recursive relation
\begin{equation}
\label{eq:kontsevich}
N_{d,0}=\sum_{k+l=d} N_{k,0}\,N_{l,0}\, k^2\,l\,
\Bigl(
l\binom{3d-4}{3k-2}-k\binom{3d-4}{3k-1}
\Bigr),
\end{equation}
for $d\ge 2$.
Even though this recurrence looks much simpler than Caporaso-Harris's,
deriving it directly from Theorem~\ref{th:bm-curves} requires
nontrivial effort. 
The blueprint for such a derivation is provided by Kontsevich's
original proof; in order to translate this proof into a
combinatorial language, one likely needs to extend the notion of a
(labeled) floor diagram to allow for a real parameter,
corresponding to the tropical cross-ratio of the appropriate point
configuration. 
While we foresee no insurmountable obstacles to implementing this plan,
the technical difficulties involved are substantial enough to require 
a separate paper. 

\begin{problem}
\label{problem:kontsevich}
Give a direct proof of Kontsevich's recursion \eqref{eq:kontsevich}
based on the combinatorial definition of the numbers~$N_{d,0}$ given
by formula~\eqref{th:bm-curves}, with $g=0$.
\end{problem}

This problem was also posed independently in
\cite[Exercice~7.2]{brugalle}.

\subsection{Alternating trees}
%
This class of labeled trees was introduced by A.~Postnikov \emph{et al.}\
\cite{gelfand-graev-postnikov, postnikov};
see~\cite[Section~4]{stanley-pnas} for a survey of related topics.
An \emph{alternating tree} $T$ is a tree on the vertex set $\{1,\dots,d\}$
such that the vertices adjacent to any given vertex~$v$
either all have smaller labels than~$v$, or all have larger labels.
That is, $T$~must not contain a $3$-vertex subtree
$a-\!\!\!-b-\!\!\!-c$ with $a<b<c$.

Let $a_d$ denote the number of alternating trees on $d$ vertices.
For example, direct inspection of 
the first table in Appendix~A 
shows that
$a_1=a_2=1$, $a_3=2$, $a_4=7$.
Postnikov has shown
\cite[Theorem~1]{postnikov}\cite[Exercise~5.41(b)]{ec2}
that
\[
a_d=\frac{1}{d\cdot 2^{d-1}}\sum_{k=1}^d \binom{d}{k} k^{d-1}.
\]

\begin{problem}
Prove or disprove:
The number of trees~$T$ on the vertex set $\{1,\dots,d\}$ for which there
is a labeled floor diagram whose underlying tree is~$T$ is equal
to~$a_d\,$.
\end{problem}

In other words, if we ignore edge weights, then the trees
obtained from labeled floor diagrams of genus~$0$ are conjecturally
equinumerous to the alternating trees.

A related (but different) enumerative problem is the following.

\begin{problem}
Enumerate multiplicity-free labeled floor diagrams of degree~$d$ and
genus~$0$, i.e., those diagrams in which $w(e)=1$ for every edge~$e$.
\end{problem}

Our calculations show that for $d=1,2,3,4,5,6$ the number of such
diagrams is $1,1,2,7,36, 245$, respectively.
These values match the first terms of the sequence
\cite[A029768]{sloane} (see also \cite[Exercise~5.2.20]{bll})
that enumerates increasing rooted trees with
cyclically ordered branches.

\subsection{Generalizations to other Lie types}

\begin{problem}
Assume that $g=0$.
Is there a natural generalization of Theorem~\ref{th:bm-curves}
(including the numbers~$N_{d,0}$ and
the notion of a labeled floor diagram)
associated with an arbitrary finite indecomposable crystallographic
root system~$\Phi$?
\end{problem}

The possibility of such a generalization is prompted by existing
interpretations of labeled trees on $d$ vertices
(hence, by extension, labeled floor diagrams of genus~$0$ and degree~$d$)
as ``type~$A$ objects,'' i.e.,
combinatorial gadgets associated with a root system of type~$A$.
Such an association can actually be made in at least two substantially
different ways, involving root systems of types $A_{d-2}$
and~$A_{d-1}$, respectively.
Both constructions are fairly well known,
so we present them cursorily,
referring the reader to relevant sources for further
details.

Let $\Phi_{>0}$ denote the set  of positive roots in the root system~$\Phi$.
The first construction is based on the notion of
the \emph{Shi arrangement}, the arrangement of hyperplanes in the
(co-)root space of $\Phi$ defined by the equations
\[
\langle x,\alpha\rangle\in \{0,1\}, \quad \alpha\in\Phi_{>0}\,.
\]
It was conjectured by R.~W.~Carter and proved by
J.-Y.~Shi~\cite[Theorem~8.1]{shi}
that the number of regions of this arrangement
(i.e., the number of connected components in the
complement to the union of these hyperplanes) is equal to $(h+1)^n$,
where $n$ is the rank of $\Phi$ and $h$ is the Coxeter number.
If $\Phi$ is of type~$A_{d-2}$, then $n=d-2$ and $h=d-1$, so
$(h+1)^n=d^{d-2}$, the number of labeled trees on $d$ vertices.
Explicit bijections are known (see \cite[Section~5]{stanley-pnas}
and~\cite{athanasiadis-linusson})
that identify such trees,
and therefore labeled floor diagrams of genus~$0$ and degree~$d$,
with the regions of the Shi arrangement of type~$A_{d-2}$.
It is not unreasonable to anticipate a generalization of this
correspondence to suitably defined labeled floor diagrams of other
types.

It is worth mentioning that the alternating trees discussed
earlier in this section have a natural analogue for
any root system~$\Phi$, since they are equinumerous to (and can be
identified with) the regions of the \emph{Linial arrangement}
\[
\langle x,\alpha\rangle=1, \quad \alpha\in\Phi_{>0}\,.
\]
(This was conjectured by R.~Stanley and proved by
A.~Postnikov~\cite[Section~4.2]{postnikov}\cite[Theorem~8.2]{postnikov-stanley}\cite[Exercise~5.41(h)]{ec2}
and later by
C.~Athanasiadis~\cite[Theorem~4.1]{athanasiadis}.)




The second construction involves the \emph{noncrossing partition
  lattice} $\operatorname{NC}(\Phi)$
associated with the root system~$\Phi$.
By a theorem of F.~Chapoton \cite[Proposition~9]{chapoton},
the number of maximal chains in $\operatorname{NC}(\Phi)$ is equal
  to~$\dfrac{h^n n!}{|W|}$, where $W$ denotes the associated Weyl
  group.
For $\Phi$ of type~$A_{d-1}$, one recovers Cayley's formula.

\section{Welschinger invariants and odd floor diagrams}
\label{sec:welschinger}

The Gromov-Witten number $N_{d,0}$ has a ``real'' counterpart,
the \emph{Welschinger invariant}~$W_d$ \cite{welschinger1,
  welschinger2}
that counts \emph{real} rational curves of degree~$d$ through generic
$3d-1$ points on the real projective plane, each with a certain sign. 
To be specific, to a nodal algebraic curve in $\R{\mathbb P}^2$,
let us associate a sign that equals $(-1)$ to the power of the number of
its solitary nodes (i.e., points locally given by $x^2+y^2=0$). 
These signs were introduced by Welschinger~\cite{welschinger2}
who proved that the signed count~$W_d$  
does not depend on the choice of a configuration.

The following result is a restatement of (a part of)
\cite[Theorem~2]{brugalle-mikhalkin}.

\begin{theorem} 
\label{th:bm-curves-welsch}
The Welschinger  invariant~$W_d$ is equal to
\begin{equation}
\label{eq:N=sum-mu-nu-welsch}
W_d=\sum_{\mu(\Dcal)\equiv 1\bmod 2} \nu(\Dcal),
\end{equation}
the sum over all labeled floor diagrams~$\Dcal$ of degree~$d$ and
genus~$0$ whose edge weights are all odd.
\end{theorem}

\begin{example} 
For $d=3$ and $d=4$ (cf.\ Appendix~A), 
formula~\eqref{eq:N=sum-mu-nu-welsch} gives 
\begin{align*}
W_3&= 5 + 3=8,\\
W_4&= 40+35+45+3+24+46+32+15=240. 
\end{align*}
\end{example}

A labeled floor diagram $\Dcal$ is called \emph{odd} if
the weight $w(e)$ of every edge $e$ in~$\Dcal$ is an odd number.
Thus, the summation in \eqref{eq:N=sum-mu-nu-welsch} is over all odd
labeled floor diagrams~$\Dcal$ of degree~$d$ and
genus~$0$.

\begin{problem}
Enumerate the odd labeled floor diagrams of degree~$d$ and
genus~$0$.
\end{problem}

Let $b_d$ denote the number of such diagrams.
Our calculations show that
\[
b_1=1,\ \
b_2=1, \ \
b_3=2, \ \
b_4=8, \ \
b_5=46, \ \
b_6=352.
\]
Curiously, these numbers match the first terms of the sequence
\cite[A099765]{sloane} given~by
\[
b_d=\frac{1}{d} \sum_{k=0}^{\lfloor d/2\rfloor}
           (-1)^k \binom{d}{k}(d-2k)^{d-1}.
\]

Theorem~\ref{th:bm-curves-welsch} was recently used
in~\cite{arroyo-brugalle-lopez}
to obtain a Caporaso-Harris-type
recurrence for Welschinger invariants.

\newpage

\section*{Appendix~A: Labeled floor diagrams with $d\le 4$}

\medskip

\begin{center}
\begin{tabular}{c|c|c|c|c}
 &
$\Dcal$ 
& labeled tree
& $\mu(\Dcal)$ 
& $\nu(\Dcal)$ 
\\
\hline \hline
&&&&\\[-.15in]
$d=1\quad g=0\ $ &
\begin{picture}(20,8)(-10,-4)\setlength{\unitlength}{2.5pt}\thicklines
\multiput(0,0)(10,0){1}{\circle{2}}
\end{picture}
&
\begin{picture}(20,8)(-10,-4)\setlength{\unitlength}{2.5pt}\thicklines
\multiput(0,0)(10,0){1}{\circle{2}}
\end{picture}
& 1 & 1
\\[.05in]
\hline \hline
&&&&\\[-.1in]
$d=2\quad g=0\ $ &
\begin{picture}(45,8)(-10,-4)\setlength{\unitlength}{2.5pt}\thicklines
\multiput(0,0)(10,0){2}{\circle{2}}
\Eee
\end{picture}
&
\begin{picture}(45,8)(-10,-4)\setlength{\unitlength}{2.5pt}\thicklines
\multiput(0,0)(10,0){2}{\circle{2}}
\Eee
\end{picture}
& 1 & 1
\\[.05in]
\hline \hline
&&&\\[-.1in]
&
\begin{picture}(70,8)(-10,-3)\setlength{\unitlength}{2.5pt}\thicklines
\ooo\Eee\eEe
\end{picture}
&
\begin{picture}(70,8)(-10,-3)\setlength{\unitlength}{2.5pt}\thicklines
\ooo\Eee\eEe
\end{picture}
& 1 & 5 \\[.2in]
$d=3\quad g=0\ $&\begin{picture}(70,8)(-10,-3)\setlength{\unitlength}{2.5pt}\thicklines
\ooo\Eee\eEe
\put(15,2){\makebox(0,0){$\scriptstyle 2$}}
\end{picture}
&
\begin{picture}(70,8)(-10,-3)\setlength{\unitlength}{2.5pt}\thicklines
\ooo\Eee\EEe
\end{picture}
& 4 & 1 \\[.2in]
&\begin{picture}(70,8)(-10,-3)\setlength{\unitlength}{2.5pt}\thicklines
\ooo\EEe\eEe
\end{picture}
&
\begin{picture}(70,8)(-10,-3)\setlength{\unitlength}{2.5pt}\thicklines
\ooo\EEe\eEe
\end{picture}
& 1 & 3
\\[.05in]
\hline\hline
&&&\\[-.1in]
&\begin{picture}(95,8)(-10,-1)\setlength{\unitlength}{2.5pt}\thicklines
\oooo\Eee\eEe\eeE
\end{picture}
&
\begin{picture}(95,8)(-10,-1)\setlength{\unitlength}{2.5pt}\thicklines
\oooo\Eee\eEe\eeE
\end{picture}
& 1 & 40 \\[.2in]
&\begin{picture}(95,8)(-10,-1)\setlength{\unitlength}{2.5pt}\thicklines
\oooo\Eee\eEe\eeE
\put(25,2){\makebox(0,0){$\scriptstyle 2$}}
\end{picture}
&
\begin{picture}(95,8)(-10,-1)\setlength{\unitlength}{2.5pt}\thicklines
\oooo\Eee\eEe\eEE
\end{picture}
& 4 & 8 \\[.2in]
&\begin{picture}(95,8)(-10,-1)\setlength{\unitlength}{2.5pt}\thicklines
\oooo\Eee\eEe\eEE
\end{picture}
&
\begin{picture}(95,8)(-10,-1)\setlength{\unitlength}{2.5pt}\thicklines
\oooo\Eee\eEe\EEE
\end{picture}
& 1 & 35 \\[.2in]
&\begin{picture}(95,8)(-10,-1)\setlength{\unitlength}{2.5pt}\thicklines
\oooo\Eee\eEe\eeE
\put(15,2){\makebox(0,0){$\scriptstyle 2$}}
\end{picture}
&
\begin{picture}(95,8)(-10,-1)\setlength{\unitlength}{2.5pt}\thicklines
\oooo\Eee\EEe\eeE
\end{picture}
& 4 & 15 \\[.2in]
&\begin{picture}(95,8)(-10,-1)\setlength{\unitlength}{2.5pt}\thicklines
\oooo\Eee\eEe\eeE
\put(15,2){\makebox(0,0){$\scriptstyle 2$}}
\put(25,2){\makebox(0,0){$\scriptstyle 2$}}
\end{picture}
&
\begin{picture}(95,8)(-10,-1)\setlength{\unitlength}{2.5pt}\thicklines
\oooo\Eee\EEe\eEE
\end{picture}
& 16 & 6 \\[.2in]
&\begin{picture}(95,8)(-10,-1)\setlength{\unitlength}{2.5pt}\thicklines
\oooo\Eee\eEe\eeE
\put(15,2){\makebox(0,0){$\scriptstyle 2$}}
\put(25,2){\makebox(0,0){$\scriptstyle 3$}}
\end{picture}
&
\begin{picture}(95,8)(-10,-1)\setlength{\unitlength}{2.5pt}\thicklines
\oooo\Eee\EEe\EEE
\end{picture}
& 36 & 1 \\[.2in]
&\begin{picture}(95,8)(-10,-1)\setlength{\unitlength}{2.5pt}\thicklines
\oooo\EEe\eEe\eeE
\end{picture}
&
\begin{picture}(95,8)(-10,-1)\setlength{\unitlength}{2.5pt}\thicklines
\oooo\EEe\eEe\eeE
\end{picture}
& 1 & 45 \\[.2in]
&\begin{picture}(95,8)(-10,-1)\setlength{\unitlength}{2.5pt}\thicklines
\oooo\EEe\eEe\eeE
\put(25,2){\makebox(0,0){$\scriptstyle 2$}}
\end{picture}
&
\begin{picture}(95,8)(-10,-1)\setlength{\unitlength}{2.5pt}\thicklines
\oooo\EEe\eEe\eEE
\end{picture}
& 4 & 18 \\[.2in]
$d=4\quad g=0\ $&\begin{picture}(95,8)(-10,-1)\setlength{\unitlength}{2.5pt}\thicklines
\oooo\EEe\eEe\eeE
\put(25,2){\makebox(0,0){$\scriptstyle 3$}}
\end{picture}
&
\begin{picture}(95,8)(-10,-1)\setlength{\unitlength}{2.5pt}\thicklines
\oooo\EEe\eEe\EEE
\end{picture}
& 9 & 3 \\[.2in]
&\begin{picture}(95,8)(-10,-1)\setlength{\unitlength}{2.5pt}\thicklines
\oooo\Eee\eEE\eeE
\end{picture}
&
\begin{picture}(95,8)(-10,-1)\setlength{\unitlength}{2.5pt}\thicklines
\oooo\Eee\eEE\eeE
\end{picture}
& 1 & 24 \\[.2in]
&\begin{picture}(95,8)(-10,-1)\setlength{\unitlength}{2.5pt}\thicklines
\oooo\Eee\eEE\eeE
\put(20,7){\makebox(0,0){$\scriptstyle 2$}}
\end{picture}
&
\begin{picture}(95,8)(-10,-1)\setlength{\unitlength}{2.5pt}\thicklines
\oooo\Eee\EEE\eeE
\end{picture}
& 4 & 3 \\[.2in]
&\begin{picture}(95,8)(-10,-1)\setlength{\unitlength}{2.5pt}\thicklines
\oooo\EEe\eEE\eeE
\end{picture}
&
\begin{picture}(95,8)(-10,-1)\setlength{\unitlength}{2.5pt}\thicklines
\oooo\EEe\eEE\eeE
\end{picture}
& 1 & 46 \\[.2in]
&\begin{picture}(95,8)(-10,-1)\setlength{\unitlength}{2.5pt}\thicklines
\oooo\EEe\eEE\eeE
\put(25,2){\makebox(0,0){$\scriptstyle 2$}}
\end{picture}
&
\begin{picture}(95,8)(-10,-1)\setlength{\unitlength}{2.5pt}\thicklines
\oooo\EEe\eEE\EEE
\end{picture}
& 4 & 7 \\[.2in]
&\begin{picture}(95,8)(-10,-1)\setlength{\unitlength}{2.5pt}\thicklines
\oooo\EEE\eEe\eeE
\end{picture}
&
\begin{picture}(95,8)(-10,-1)\setlength{\unitlength}{2.5pt}\thicklines
\oooo\eEe\eeE\EEE
\end{picture}
& 1 & 32 \\[.2in]
&\begin{picture}(95,8)(-10,-1)\setlength{\unitlength}{2.5pt}\thicklines
\oooo\EEE\eEe\eeE
\put(25,2){\makebox(0,0){$\scriptstyle 2$}}
\end{picture}
&
\begin{picture}(95,8)(-10,-1)\setlength{\unitlength}{2.5pt}\thicklines
\oooo\eEe\eEE\EEE
\end{picture}
& 4 & 5 \\[.2in]
&\begin{picture}(95,8)(-10,-1)\setlength{\unitlength}{2.5pt}\thicklines
\oooo\EEE\eEE\eeE
\end{picture}
&
\begin{picture}(95,8)(-10,-1)\setlength{\unitlength}{2.5pt}\thicklines
\oooo\eeE\eEE\EEE
\end{picture}
& 1 & 15
\\[-.05in]
\end{tabular}
\end{center}

\newpage


\begin{center}
\begin{tabular}{c|c|c|c}
 & $\Dcal$ 
& $\mu(\Dcal)$ 
& $\nu(\Dcal)$ 
\\ 
\hline \hline
&&&\\[-.05in]
$d=3\quad g=1\ $ &\begin{picture}(70,8)(-10,-3)\setlength{\unitlength}{2.5pt}\thicklines
\ooo\Eee\eOe
\end{picture}
& 1 & 1
\\[.1in]
\hline\hline
&&&\\[-.05in]
&\begin{picture}(95,8)(-10,-1)\setlength{\unitlength}{2.5pt}\thicklines
\oooo\Eee\eEe\eeE\eEE
\end{picture}
& 1 & 26 \\[.2in]
&\begin{picture}(95,8)(-10,-1)\setlength{\unitlength}{2.5pt}\thicklines
\oooo\Eee\eEe\eeE\eEE
\put(25,2){\makebox(0,0){$\scriptstyle 2$}}
\end{picture}
& 4 & 4 \\[.2in]
&\begin{picture}(95,8)(-10,-1)\setlength{\unitlength}{2.5pt}\thicklines
\oooo\Eee\eOe\eeE
\end{picture}
& 1 & 15 \\[.2in]
&\begin{picture}(95,8)(-10,-1)\setlength{\unitlength}{2.5pt}\thicklines
\oooo\Eee\eOe\eeE
\put(25,2){\makebox(0,0){$\scriptstyle 2$}}
\end{picture}
& 4 & 6 \\[.2in]
&\begin{picture}(95,8)(-10,-1)\setlength{\unitlength}{2.5pt}\thicklines
\oooo\Eee\eOe\eeE
\put(25,2){\makebox(0,0){$\scriptstyle 3$}}
\end{picture}
& 9 & 1 \\[.2in]
&\begin{picture}(95,8)(-10,-1)\setlength{\unitlength}{2.5pt}\thicklines
\oooo\Eee\eOO\eeE
\end{picture}
& 1 & 6 \\[.2in]
$d=4\quad g=1\ $ &\begin{picture}(95,8)(-10,-1)\setlength{\unitlength}{2.5pt}\thicklines
\oooo\Eee\eEe\eeO
\end{picture}
& 1 & 9 \\[.2in]
&\begin{picture}(95,8)(-10,-1)\setlength{\unitlength}{2.5pt}\thicklines
\oooo\Eee\eEe\eeO
\put(15,2){\makebox(0,0){$\scriptstyle 2$}}
\end{picture}
& 4 & 7 \\[.2in]
&\begin{picture}(95,8)(-10,-1)\setlength{\unitlength}{2.5pt}\thicklines
\oooo\Eee\eEe\eeO
\put(15,2){\makebox(0,0){$\scriptstyle 2$}}
\put(25,4.5){\makebox(0,0){$\scriptstyle 2$}}
\end{picture}
& 16 & 2 \\[.2in]
&\begin{picture}(95,8)(-10,-1)\setlength{\unitlength}{2.5pt}\thicklines
\oooo\EEe\eEe\eeO
\end{picture}
& 1 & 21 \\[.2in]
&\begin{picture}(95,8)(-10,-1)\setlength{\unitlength}{2.5pt}\thicklines
\oooo\EEe\eEe\eeO
\put(25,4.5){\makebox(0,0){$\scriptstyle 2$}}
\end{picture}
& 4 & 6 \\[.2in]
&\begin{picture}(95,8)(-10,-1)\setlength{\unitlength}{2.5pt}\thicklines
\oooo\EEe\eEE\eeO
\end{picture}
& 1 & 9 \\[.2in]
&\begin{picture}(95,8)(-10,-1)\setlength{\unitlength}{2.5pt}\thicklines
\oooo\EEE\eEe\eeO
\end{picture}
& 1 & 6
\\[.1in]
\hline\hline
&&&\\[-.05in]
&\begin{picture}(95,8)(-10,-1)\setlength{\unitlength}{2.5pt}\thicklines
\oooo\EEe\eEe\eeE\eeO
\end{picture}
& 1 & 3 \\[.2in]
&\begin{picture}(95,8)(-10,-1)\setlength{\unitlength}{2.5pt}\thicklines
\oooo\Eee\eEe\eEE\eeO
\end{picture}
& 1 & 5 \\[.2in]
$d=4\quad g=2\ $&\begin{picture}(95,8)(-10,-1)\setlength{\unitlength}{2.5pt}\thicklines
\oooo\Eee\eOe\eeO
\end{picture}
& 1 & 7 \\[.2in]
&\begin{picture}(95,8)(-10,-1)\setlength{\unitlength}{2.5pt}\thicklines
\oooo\Eee\eEe\eeE\eeO
\put(15,2){\makebox(0,0){$\scriptstyle 2$}}
\end{picture}
& 4 & 1 \\[.2in]
&\begin{picture}(95,8)(-10,-1)\setlength{\unitlength}{2.5pt}\thicklines
\oooo\Eee\eOe\eeO
\put(25,4.5){\makebox(0,0){$\scriptstyle 2$}}
\end{picture}
& 4 & 2 \\[.1in]
\hline\hline
&&&\\[-.1in]
$d=4\quad g=3\ $&\begin{picture}(95,8)(-10,-1)\setlength{\unitlength}{2.5pt}\thicklines
\oooo\Eee\eOe\eeO\eeE
\end{picture}
& 1 & 1
\\[.05in]
\end{tabular}
\end{center}

\newpage

\section*{Appendix~B: \hbox{Tropical rational cubics and their marked floor
  diagrams}} 


\begin{center}
\vspace{.1in}
\begin{tabular}{c|c|c}
\setlength{\unitlength}{0.6pt}
\begin{picture}(90,300)(5,0)
\thicklines

\darkred{\linethickness{1pt}
\put(50,150){\line(0,1){40}}
\put(70,230){\line(0,1){50}}
}

\lightgreen{\linethickness{1pt}
\put(10,0){\line(0,1){150}}
\put(20,0){\line(0,1){100}}
\put(30,0){\line(0,1){110}}
}

\put(0,100){\line(1,0){20}}
\put(0,150){\line(1,0){10}}
\put(0,280){\line(1,0){70}}

\put(20,100){\line(1,1){10}}
\put(10,150){\line(1,1){40}}
\put(50,150){\line(1,1){40}}
\put(70,230){\line(1,1){20}}

\put(70,280){\line(1,1){8.6}}
\put(90,300){\line(-1,-1){8.6}}

\put(30,110){\line(1,2){9.1}}
\put(50,150){\line(-1,-2){9.1}}

\put(50,190){\line(1,2){9.1}}
\put(70,230){\line(-1,-2){9.1}}

\put(10,10){\circle*{4}}
\put(20,50){\circle*{4}}
\put(30,90){\circle*{4}}
\put(40,130){\circle{4}}
\put(50,170){\circle*{4}}
\put(60,210){\circle{4}}
\put(70,250){\circle*{4}}
\put(80,290){\circle{4}}

\end{picture}
\ \ 
\begin{picture}(20,300)(5,0)
\thicklines

\darkred{\linethickness{1pt}
\put(0,248){\line(0,-1){36}}
\put(0,230){\vector(0,-1){5}}
\put(0,288){\line(0,-1){36}}
\put(0,270){\vector(0,-1){5}}
\put(0,208){\line(0,-1){36}}
\put(0,190){\vector(0,-1){5}}
\put(0,168){\line(0,-1){36}}
\put(0,150){\vector(0,-1){5}}

}

\lightgreen{\linethickness{1pt}
\put(0,128){\line(0,-1){36}}
\put(0,110){\vector(0,-1){5}}
\qbezier(0.6,51.9)(14,90)(0.6,128.1)
\put(8,90){\vector(0,-1){1}}
\qbezier(1.2,11.6)(40,120)(1.2,208.4)
\put(21,110){\vector(0,-1){1}}
}

\put(0,10){\circle*{4}}
\put(0,50){\circle*{4}}
\put(0,90){\circle*{4}}
\put(0,130){\circle{4}}
\put(0,170){\circle*{4}}
\put(0,210){\circle{4}}
\put(0,250){\circle*{4}}
\put(0,290){\circle{4}}

\end{picture}
{\ }&{\ }
\setlength{\unitlength}{0.6pt}
\begin{picture}(90,300)(5,0)
\thicklines

\darkred{\linethickness{1pt}
\put(50,150){\line(0,1){40}}
\put(70,230){\line(0,1){50}}
}

\lightgreen{\linethickness{1pt}
\put(10,0){\line(0,1){90}}
\put(20,0){\line(0,1){160}}
\put(30,0){\line(0,1){110}}
}

\put(0,90){\line(1,0){10}}
\put(0,160){\line(1,0){20}}
\put(0,280){\line(1,0){70}}

\put(10,90){\line(1,1){20}}
\put(20,160){\line(1,1){30}}
\put(50,150){\line(1,1){40}}
\put(70,230){\line(1,1){20}}

\put(70,280){\line(1,1){8.6}}
\put(90,300){\line(-1,-1){8.6}}

\put(30,110){\line(1,2){9.1}}
\put(50,150){\line(-1,-2){9.1}}

\put(50,190){\line(1,2){9.1}}
\put(70,230){\line(-1,-2){9.1}}

\put(10,10){\circle*{4}}
\put(20,50){\circle*{4}}
\put(30,90){\circle*{4}}
\put(40,130){\circle{4}}
\put(50,170){\circle*{4}}
\put(60,210){\circle{4}}
\put(70,250){\circle*{4}}
\put(80,290){\circle{4}}

\end{picture}
\ \ 
\begin{picture}(20,300)(5,0)
\thicklines

\darkred{\linethickness{1pt}
\put(0,248){\line(0,-1){36}}
\put(0,230){\vector(0,-1){5}}
\put(0,288){\line(0,-1){36}}
\put(0,270){\vector(0,-1){5}}
\put(0,208){\line(0,-1){36}}
\put(0,190){\vector(0,-1){5}}
\put(0,168){\line(0,-1){36}}
\put(0,150){\vector(0,-1){5}}

}

\lightgreen{\linethickness{1pt}
\put(0,128){\line(0,-1){36}}
\put(0,110){\vector(0,-1){5}}
\qbezier(1.2,11.6)(30,80)(1.2,128.4)
\put(16,70){\vector(0,-1){1}}
\qbezier(1.2,51.6)(40,140)(1.2,208.4)
\put(21,130){\vector(0,-1){1}}
}

\put(0,10){\circle*{4}}
\put(0,50){\circle*{4}}
\put(0,90){\circle*{4}}
\put(0,130){\circle{4}}
\put(0,170){\circle*{4}}
\put(0,210){\circle{4}}
\put(0,250){\circle*{4}}
\put(0,290){\circle{4}}

\end{picture}
{\ }&{\ }
\setlength{\unitlength}{0.6pt}
\begin{picture}(90,300)(5,0)
\thicklines

\darkred{\linethickness{1pt}
\put(50,150){\line(0,1){40}}
\put(70,230){\line(0,1){50}}
}

\lightgreen{\linethickness{1pt}
\put(10,0){\line(0,1){80}}
\put(20,0){\line(0,1){90}}
\put(30,0){\line(0,1){170}}
}

\put(0,80){\line(1,0){10}}
\put(0,170){\line(1,0){30}}
\put(0,280){\line(1,0){70}}

\put(10,80){\line(1,1){10}}
\put(30,170){\line(1,1){20}}
\put(50,150){\line(1,1){40}}
\put(70,230){\line(1,1){20}}

\put(70,280){\line(1,1){8.6}}
\put(90,300){\line(-1,-1){8.6}}

\put(20,90){\line(1,2){19.1}}
\put(50,150){\line(-1,-2){9.1}}

\put(50,190){\line(1,2){9.1}}
\put(70,230){\line(-1,-2){9.1}}

\put(10,10){\circle*{4}}
\put(20,50){\circle*{4}}
\put(30,90){\circle*{4}}
\put(40,130){\circle{4}}
\put(50,170){\circle*{4}}
\put(60,210){\circle{4}}
\put(70,250){\circle*{4}}
\put(80,290){\circle{4}}

\end{picture}
\ \ 
\begin{picture}(20,300)(5,0)
\thicklines

\darkred{\linethickness{1pt}
\put(0,248){\line(0,-1){36}}
\put(0,230){\vector(0,-1){5}}
\put(0,288){\line(0,-1){36}}
\put(0,270){\vector(0,-1){5}}
\put(0,208){\line(0,-1){36}}
\put(0,190){\vector(0,-1){5}}
\put(0,168){\line(0,-1){36}}
\put(0,150){\vector(0,-1){5}}

}

\lightgreen{\linethickness{1pt}
\qbezier(1.2,11.6)(30,80)(1.2,128.4)
\put(16,70){\vector(0,-1){1}}
\qbezier(1.2,91.6)(30,160)(1.2,208.4)
\put(16,150){\vector(0,-1){1}}
\qbezier(0.6,51.9)(14,90)(0.6,128.1)
\put(8,90){\vector(0,-1){1}}
}

\put(0,10){\circle*{4}}
\put(0,50){\circle*{4}}
\put(0,90){\circle*{4}}
\put(0,130){\circle{4}}
\put(0,170){\circle*{4}}
\put(0,210){\circle{4}}
\put(0,250){\circle*{4}}
\put(0,290){\circle{4}}

\end{picture}
\\[.1in]
\hline
&&\\[-.1in]
\setlength{\unitlength}{0.6pt}
\begin{picture}(90,300)(5,0)
\thicklines

\darkred{\linethickness{1pt}
\put(50,130){\line(0,1){60}}
\put(70,230){\line(0,1){50}}
}

\lightgreen{\linethickness{1pt}
\put(10,0){\line(0,1){60}}
\put(20,0){\line(0,1){70}}
\put(40,0){\line(0,1){180}}
}

\put(0,60){\line(1,0){10}}
\put(0,180){\line(1,0){40}}
\put(0,280){\line(1,0){70}}

\put(10,60){\line(1,1){10}}
\put(40,180){\line(1,1){10}}
\put(50,130){\line(1,1){40}}
\put(70,230){\line(1,1){20}}

\put(70,280){\line(1,1){8.6}}
\put(90,300){\line(-1,-1){8.6}}

\put(20,70){\line(1,2){9.1}}
\put(50,130){\line(-1,-2){19.1}}

\put(50,190){\line(1,2){9.1}}
\put(70,230){\line(-1,-2){9.1}}

\put(10,10){\circle*{4}}
\put(20,50){\circle*{4}}
\put(30,90){\circle{4}}
\put(40,130){\circle*{4}}
\put(50,170){\circle*{4}}
\put(60,210){\circle{4}}
\put(70,250){\circle*{4}}
\put(80,290){\circle{4}}

\end{picture}
\ \ 
\begin{picture}(20,300)(5,0)
\thicklines

\darkred{\linethickness{1pt}
\put(0,248){\line(0,-1){36}}
\put(0,230){\vector(0,-1){5}}
\put(0,288){\line(0,-1){36}}
\put(0,270){\vector(0,-1){5}}
\qbezier(1.2,91.6)(30,130)(1,168.4)
\put(16,130){\vector(0,-1){1}}
\put(0,208){\line(0,-1){36}}
\put(0,190){\vector(0,-1){5}}

}

\lightgreen{\linethickness{1pt}
\put(0,88){\line(0,-1){36}}
\put(0,70){\vector(0,-1){5}}
\qbezier(1.2,11.6)(30,50)(1,88.4)
\put(16,50){\vector(0,-1){1}}
\qbezier(1.2,131.6)(30,170)(1.2,208.4)
\put(16,170){\vector(0,-1){1}}
}

\put(0,10){\circle*{4}}
\put(0,50){\circle*{4}}
\put(0,90){\circle{4}}
\put(0,130){\circle*{4}}
\put(0,170){\circle*{4}}
\put(0,210){\circle{4}}
\put(0,250){\circle*{4}}
\put(0,290){\circle{4}}

\end{picture}
{\ }&{\ }
\setlength{\unitlength}{0.6pt}
\begin{picture}(90,300)(5,0)
\thicklines

\darkred{\linethickness{1pt}
\put(40,110){\line(0,1){70}}
\put(70,230){\line(0,1){50}}
}

\lightgreen{\linethickness{1pt}
\put(10,0){\line(0,1){60}}
\put(20,0){\line(0,1){70}}
\put(50,0){\line(0,1){190}}
}

\put(0,60){\line(1,0){10}}
\put(0,180){\line(1,0){40}}
\put(0,280){\line(1,0){70}}

\put(10,60){\line(1,1){10}}
\put(40,180){\line(1,1){10}}
\put(40,110){\line(1,1){50}}
\put(70,230){\line(1,1){20}}

\put(70,280){\line(1,1){8.6}}
\put(90,300){\line(-1,-1){8.6}}

\put(20,70){\line(1,2){9.1}}
\put(40,110){\line(-1,-2){9.1}}

\put(50,190){\line(1,2){9.1}}
\put(70,230){\line(-1,-2){9.1}}

\put(10,10){\circle*{4}}
\put(20,50){\circle*{4}}
\put(30,90){\circle{4}}
\put(40,130){\circle*{4}}
\put(50,170){\circle*{4}}
\put(60,210){\circle{4}}
\put(70,250){\circle*{4}}
\put(80,290){\circle{4}}

\end{picture}
\ \ 
\begin{picture}(20,300)(5,0)
\thicklines

\darkred{\linethickness{1pt}
\put(0,128){\line(0,-1){36}}
\put(0,110){\vector(0,-1){5}}
\put(0,248){\line(0,-1){36}}
\put(0,230){\vector(0,-1){5}}
\put(0,288){\line(0,-1){36}}
\put(0,270){\vector(0,-1){5}}
\qbezier(1.2,131.6)(30,170)(1,208.4)
\put(16,170){\vector(0,-1){1}}

}

\lightgreen{\linethickness{1pt}
\put(0,88){\line(0,-1){36}}
\put(0,70){\vector(0,-1){5}}
\put(0,208){\line(0,-1){36}}
\put(0,190){\vector(0,-1){5}}
\qbezier(1.2,11.6)(30,50)(1.2,88.4)
\put(16,50){\vector(0,-1){1}}
}

\put(0,10){\circle*{4}}
\put(0,50){\circle*{4}}
\put(0,90){\circle{4}}
\put(0,130){\circle*{4}}
\put(0,170){\circle*{4}}
\put(0,210){\circle{4}}
\put(0,250){\circle*{4}}
\put(0,290){\circle{4}}

\end{picture}
{\ }&{\ }
\setlength{\unitlength}{0.6pt}
\begin{picture}(90,300)(5,0)
\thicklines

\darkred{\linethickness{1pt}
\put(50,160){\line(0,1){30}}
\put(70,230){\line(0,1){50}}
}

\put(55,180){\makebox(0,0){$\scriptstyle{2}$}}

\lightgreen{\linethickness{1pt}
\put(10,0){\line(0,1){70}}
\put(20,0){\line(0,1){80}}
\put(30,0){\line(0,1){100}}
}

\put(0,70){\line(1,0){10}}
\put(0,190){\line(1,0){50}}
\put(0,280){\line(1,0){70}}

\put(10,70){\line(1,1){10}}
\put(50,160){\line(1,1){40}}
\put(70,230){\line(1,1){20}}

\put(70,280){\line(1,1){8.6}}
\put(90,300){\line(-1,-1){8.6}}

\put(20,80){\line(1,2){10}}

\put(50,190){\line(1,2){9.1}}
\put(70,230){\line(-1,-2){9.1}}

\put(30,100){\line(1,3){9.4}}
\put(50,160){\line(-1,-3){9.4}}

\put(10,10){\circle*{4}}
\put(20,50){\circle*{4}}
\put(30,90){\circle*{4}}
\put(40,130){\circle{4}}
\put(50,170){\circle*{4}}
\put(60,210){\circle{4}}
\put(70,250){\circle*{4}}
\put(80,290){\circle{4}}

\put(70,30){\makebox(0,0){$\mu=4$}}

\end{picture}
\ \ 
\begin{picture}(20,300)(5,0)
\thicklines

\darkred{\linethickness{1pt}
\put(0,248){\line(0,-1){36}}
\put(0,230){\vector(0,-1){5}}
\put(0,288){\line(0,-1){36}}
\put(0,270){\vector(0,-1){5}}
\put(0,208){\line(0,-1){36}}
\put(0,190){\vector(0,-1){5}}
\put(0,168){\line(0,-1){36}}
\put(0,150){\vector(0,-1){5}}

}

\lightgreen{\linethickness{1pt}
\put(0,128){\line(0,-1){36}}
\put(0,110){\vector(0,-1){5}}
\qbezier(1.2,11.6)(30,80)(1.2,128.4)
\put(16,70){\vector(0,-1){1}}
\qbezier(0.6,51.9)(14,90)(0.6,128.1)
\put(8,90){\vector(0,-1){1}}
}

\put(7,150){\makebox(0,0){$\scriptstyle{2}$}}
\put(7,190){\makebox(0,0){$\scriptstyle{2}$}}

\put(0,10){\circle*{4}}
\put(0,50){\circle*{4}}
\put(0,90){\circle*{4}}
\put(0,130){\circle{4}}
\put(0,170){\circle*{4}}
\put(0,210){\circle{4}}
\put(0,250){\circle*{4}}
\put(0,290){\circle{4}}

\end{picture}
\\[.1in]
\hline
&&\\[-.17in]
\hline
&&\\[-.1in]
\setlength{\unitlength}{0.6pt}
\begin{picture}(90,300)(5,0)
\thicklines

\darkred{\linethickness{1pt}
\put(50,160){\line(0,1){70}}
\put(60,180){\line(0,1){90}}
}

\lightgreen{\linethickness{1pt}
\put(10,0){\line(0,1){70}}
\put(20,0){\line(0,1){80}}
\put(30,0){\line(0,1){100}}
}

\put(0,70){\line(1,0){10}}
\put(0,230){\line(1,0){50}}
\put(0,270){\line(1,0){60}}

\put(10,70){\line(1,1){10}}
\put(60,180){\line(1,1){30}}

\put(50,230){\line(1,1){18.6}}
\put(90,270){\line(-1,-1){18.6}}

\put(60,270){\line(1,1){18.6}}
\put(90,300){\line(-1,-1){8.6}}

\put(20,80){\line(1,2){10}}
\put(50,160){\line(1,2){10}}

\put(30,100){\line(1,3){9.4}}
\put(50,160){\line(-1,-3){9.4}}

\put(10,10){\circle*{4}}
\put(20,50){\circle*{4}}
\put(30,90){\circle*{4}}
\put(40,130){\circle{4}}
\put(50,170){\circle*{4}}
\put(60,210){\circle*{4}}
\put(70,250){\circle{4}}
\put(80,290){\circle{4}}

\end{picture}
\ \ 
\begin{picture}(20,300)(5,0)
\thicklines

\darkred{\linethickness{1pt}
\put(0,168){\line(0,-1){36}}
\put(0,150){\vector(0,-1){5}}
\qbezier(1.2,131.6)(30,170)(1,208.4)
\put(16,170){\vector(0,-1){1}}
\qbezier(1.2,171.6)(30,210)(1,248.4)
\put(16,210){\vector(0,-1){1}}
\qbezier(1.2,211.6)(30,250)(1,288.4)
\put(16,250){\vector(0,-1){1}}
}

\lightgreen{\linethickness{1pt}
\put(0,128){\line(0,-1){36}}
\put(0,110){\vector(0,-1){5}}
\qbezier(1.2,11.6)(30,80)(1.2,128.4)
\put(16,70){\vector(0,-1){1}}
\qbezier(0.6,51.9)(14,90)(0.6,128.1)
\put(8,90){\vector(0,-1){1}}
}

\put(0,10){\circle*{4}}
\put(0,50){\circle*{4}}
\put(0,90){\circle*{4}}
\put(0,130){\circle{4}}
\put(0,170){\circle*{4}}
\put(0,210){\circle*{4}}
\put(0,250){\circle{4}}
\put(0,290){\circle{4}}

\end{picture}
{\ }&{\ }
\setlength{\unitlength}{0.6pt}
\begin{picture}(90,300)(5,0)
\thicklines

\darkred{\linethickness{1pt}
\put(50,160){\line(0,1){100}}
\put(60,180){\line(0,1){60}}
}

\lightgreen{\linethickness{1pt}
\put(10,0){\line(0,1){70}}
\put(20,0){\line(0,1){80}}
\put(30,0){\line(0,1){100}}
}

\put(0,70){\line(1,0){10}}
\put(0,240){\line(1,0){60}}
\put(0,260){\line(1,0){50}}

\put(10,70){\line(1,1){10}}
\put(60,180){\line(1,1){30}}

\put(60,240){\line(1,1){8.6}}
\put(90,270){\line(-1,-1){18.6}}

\put(50,260){\line(1,1){28.6}}
\put(90,300){\line(-1,-1){8.6}}

\put(20,80){\line(1,2){10}}
\put(50,160){\line(1,2){10}}

\put(30,100){\line(1,3){9.4}}
\put(50,160){\line(-1,-3){9.4}}

\put(10,10){\circle*{4}}
\put(20,50){\circle*{4}}
\put(30,90){\circle*{4}}
\put(40,130){\circle{4}}
\put(50,170){\circle*{4}}
\put(60,210){\circle*{4}}
\put(70,250){\circle{4}}
\put(80,290){\circle{4}}

\end{picture}
\ \ 
\begin{picture}(20,300)(5,0)
\thicklines

\darkred{\linethickness{1pt}
\put(0,168){\line(0,-1){36}}
\put(0,150){\vector(0,-1){5}}
\put(0,248){\line(0,-1){36}}
\put(0,230){\vector(0,-1){5}}
\qbezier(1.2,131.6)(30,170)(1,208.4)
\put(16,170){\vector(0,-1){1}}
\qbezier(1.2,171.6)(30,240)(1.2,288.4)
\put(16,230){\vector(0,-1){1}}

}

\lightgreen{\linethickness{1pt}
\put(0,128){\line(0,-1){36}}
\put(0,110){\vector(0,-1){5}}
\qbezier(1.2,11.6)(30,80)(1.2,128.4)
\put(16,70){\vector(0,-1){1}}
\qbezier(0.6,51.9)(14,90)(0.6,128.1)
\put(8,90){\vector(0,-1){1}}
}

\put(0,10){\circle*{4}}
\put(0,50){\circle*{4}}
\put(0,90){\circle*{4}}
\put(0,130){\circle{4}}
\put(0,170){\circle*{4}}
\put(0,210){\circle*{4}}
\put(0,250){\circle{4}}
\put(0,290){\circle{4}}

\end{picture}
{\ }&{\ }
\setlength{\unitlength}{0.6pt}
\begin{picture}(90,300)(5,0)
\thicklines

\darkred{\linethickness{1pt}
\put(50,160){\line(0,1){40}}
\put(70,200){\line(0,1){80}}
}

\lightgreen{\linethickness{1pt}
\put(10,0){\line(0,1){70}}
\put(20,0){\line(0,1){80}}
\put(30,0){\line(0,1){100}}
}

\put(0,70){\line(1,0){10}}
\put(0,200){\line(1,0){50}}
\put(0,280){\line(1,0){70}}

\put(10,70){\line(1,1){10}}
\put(70,200){\line(1,1){20}}

\put(50,200){\line(1,1){8.6}}
\put(90,240){\line(-1,-1){28.6}}

\put(70,280){\line(1,1){8.6}}
\put(90,300){\line(-1,-1){8.6}}

\put(20,80){\line(1,2){10}}
\put(50,160){\line(1,2){20}}

\put(30,100){\line(1,3){9.4}}
\put(50,160){\line(-1,-3){9.4}}

\put(10,10){\circle*{4}}
\put(20,50){\circle*{4}}
\put(30,90){\circle*{4}}
\put(40,130){\circle{4}}
\put(50,170){\circle*{4}}
\put(60,210){\circle{4}}
\put(70,250){\circle*{4}}
\put(80,290){\circle{4}}

\end{picture}
\ \ 
\begin{picture}(20,300)(5,0)
\thicklines

\darkred{\linethickness{1pt}
\put(0,288){\line(0,-1){36}}
\put(0,270){\vector(0,-1){5}}
\put(0,208){\line(0,-1){36}}
\put(0,190){\vector(0,-1){5}}
\put(0,168){\line(0,-1){36}}
\put(0,150){\vector(0,-1){5}}
\qbezier(1.2,131.6)(30,200)(1.2,248.4)
\put(16,190){\vector(0,-1){1}}

}

\lightgreen{\linethickness{1pt}
\put(0,128){\line(0,-1){36}}
\put(0,110){\vector(0,-1){5}}
\qbezier(1.2,11.6)(30,80)(1.2,128.4)
\put(16,70){\vector(0,-1){1}}
\qbezier(0.6,51.9)(14,90)(0.6,128.1)
\put(8,90){\vector(0,-1){1}}
}

\put(0,10){\circle*{4}}
\put(0,50){\circle*{4}}
\put(0,90){\circle*{4}}
\put(0,130){\circle{4}}
\put(0,170){\circle*{4}}
\put(0,210){\circle{4}}
\put(0,250){\circle*{4}}
\put(0,290){\circle{4}}

\end{picture}
\end{tabular}
\end{center}

\end{document}